\UseRawInputEncoding
\documentclass[11pt, reqno]{amsart}
\usepackage{textcmds} 
\usepackage{amsmath, amssymb, amsfonts, amstext, verbatim, amsthm, mathrsfs}
\usepackage{microtype}
\usepackage[all]{xy}
\usepackage[modulo]{lineno}
\usepackage[usenames]{color}
\usepackage{aliascnt}
\usepackage{enumitem}
\usepackage{xspace}
\usepackage{amsfonts}
\usepackage{amssymb}
\usepackage[centertags]{amsmath}
\usepackage{amsthm}
\usepackage{rotating}
\usepackage[margin=3.5cm]{geometry}
\usepackage{dsfont}
\usepackage{bm}
\usepackage{color}
\usepackage{subfigure}
\usepackage{amsmath}
\usepackage{array}
\usepackage[all]{xy}
\usepackage{euscript}
\usepackage[T1]{fontenc}
\usepackage{mathbbol}
\usepackage{nccmath}
\usepackage{marginnote}
\usepackage{stmaryrd}
\usepackage{youngtab}
\usepackage{tikz}
\usepackage{blkarray}

\usepackage{tikz}

\usepackage{graphics,graphicx}  

\let\oldtocsection=\tocsection
\let\oldtocsubsection=\tocsubsection

\renewcommand{\tocsection}[2]{\hspace{0em}\oldtocsection{#1}{#2}}
\renewcommand{\tocsubsection}[2]{\hspace{1em}\oldtocsubsection{#1}{#2}}

\usepackage[backref,colorlinks=true,linkcolor=black,citecolor=blue,urlcolor=blue,citebordercolor={0 0 1},urlbordercolor={0 0 1},linkbordercolor={0 0 1}]{hyperref} 

\tikzset{node distance=3cm, auto}

\makeatletter
\def\@secnumfont{\bfseries}
\def\section{\@startsection{section}{1}%
  \z@{.7\linespacing\@plus\linespacing}{.5\linespacing}%
  {\normalfont\Large\bfseries}}
\def\subsection{\@startsection{subsection}{2}%
  \z@{.5\linespacing\@plus.7\linespacing}{-.5em}%
  {\normalfont\large\bfseries}}
  \def\subsubsection{\@startsection{subsubsection}{3}%
  \z@{.5\linespacing\@plus.7\linespacing}{-.5em}%
  {\normalfont\bfseries}}
\makeatother

\newtheorem{thm}{Theorem}[subsection]
\newtheorem{lemma}[thm]{Lemma}
\newtheorem{prop}[thm]{Proposition}
\newtheorem{cor}[thm]{Corollary}

\newtheorem{remark}[thm]{Remark}
\newtheorem{rmk}[thm]{Remark}

\newtheorem{definition}[thm]{Definition}
\newtheorem{example}[thm]{Example}
\theoremstyle{remark}
\numberwithin{equation}{subsection} 

\numberwithin{figure}{section}
\numberwithin{table}{section}

\newcommand{\partitions}{\op{Part}}

\newcommand{\Ee}{\mathcal{E}}
\newcommand{\Y}{\mathcal{Y}}

\newcommand{\Ynonver}{\Y^{\op{nonver}}}
\newcommand{\Ynonhor}{\Y^{\op{nonhor}}}

\newenvironment{itemlist}
   { \begin{list} {$\bullet$}
         { \setlength{\topsep}{.5ex}  \setlength{\itemsep}{.5ex} \setlength{\leftmargin}{2.5ex} } }
   { \end{list} }
   
   \newcommand{\notccirc}{{\,\,{|\mskip-9.75mu\circledcirc}\,}}
   \newcommand{\s}{{\mathfrak s}}
     \renewcommand{\wr}{{{\rm w}}}
      
      \newcommand{\Bl}{{\rm Bl}}
   \newcommand{\PSL}{{\rm PSL}}
      \newcommand{\Tom}{{\widetilde \om}}
   \newcommand{\TJ}{{\widetilde J}}
   \newcommand{\TC}{{\widetilde C}}
   \newcommand{\TA}{{\widetilde A}}
      \newcommand{\Tz}{{\widetilde z}}
      
\newcommand{\bp}{{\bf p}}

\newcommand{\NI}{{\noindent}}
\newcommand{\Pp}{{\mathcal P}}
\newcommand{\Ii}{{\mathcal I}}
\newcommand{\Qq}{{\mathcal Q}}
\newcommand{\Oo}{{\mathcal O}}
\newcommand{\Cc}{{\mathcal C}}
\newcommand{\Nn}{{\mathcal N}}
\newcommand{\Mm}{{\mathcal M}}

\newcommand{\Jj}{{\mathcal J}}
\newcommand{\im}{{\rm im\,}}
\newcommand{\CZ}{{\rm CZ}}
\newcommand{\ov}{\overline}

\newcommand{\ob}{\ov{b}}

\newcommand{\al}{{\alpha}}

\newcommand{\be}{{\beta}}

\newcommand{\om}{{\omega}}
\newcommand{\de}{{\delta}}

\newcommand{\ga}{{\gamma}}

\newcommand{\io}{{\iota}}
\newcommand{\ka}{{\kappa}}

\newcommand{\si}{{\sigma}}
\newcommand{\Si}{{\Sigma}}
\newcommand{\less} {{\smallsetminus}}
\newcommand{\p}{{\partial}}
\newcommand{\MS}{{\medskip}}
\newcommand{\ord}{{\rm ord}}
\newcommand{\er}{{\Diamond}}

\newcommand{\Z}{\mathbb{Z}}

\newcommand{\R}{\mathbb{R}}
\newcommand{\Q}{\mathbb{Q}}
\newcommand{\C}{\mathbb{C}}
\newcommand{\CP}{\mathbb{CP}}

\newcommand{\eps}{\varepsilon}

\newcommand{\bdy}{\partial}

\newcommand{\calM}{\mathcal{M}}

\newcommand{\wh}{\widehat}
\newcommand{\wt}{\widetilde}

\newcommand{\ovl}{\overline}
\newcommand{\op}[1]{{\operatorname{#1}}}

\newcommand{\cz}{{\op{CZ}}}
\newcommand{\ind}{\op{ind}}

\newcommand{\ev}{{\op{ev}}}
\newcommand{\Op}{\mathcal{O}p}

\newcommand{\sss}{\vspace{2.5 mm}}

\newcommand{\gw}{\op{GW}}
\renewcommand{\lll}{\Langle}
\newcommand{\rrr}{\Rangle}

\newcommand{\sk}{{\op{sk}}}
\newcommand{\T}{\mathcal{T}}

\newcommand{\bl}{\op{Bl}}
\newcommand{\E}{\mathcal{E}}

\newcommand{\Aut}{\op{Aut}}
\newcommand{\simp}{s}
\newcommand{\wind}{\op{wind}}
\renewcommand{\sp}{\op{sp}}
\newcommand{\ex}{\op{ex}}

\makeatletter
\newcommand{\dashover}[2][\mathop]{#1{\mathpalette\df@over{{\dashfill}{#2}}}}
\newcommand{\fillover}[2][\mathop]{#1{\mathpalette\df@over{{\solidfill}{#2}}}}
\newcommand{\df@over}[2]{\df@@over#1#2}
\newcommand\df@@over[3]{%
  \vbox{
    \offinterlineskip
    \ialign{##\cr
      #2{#1}\cr
      \noalign{\kern1pt}
      $\m@th#1#3$\cr
    }
  }%
}
\newcommand{\dashfill}[1]{%
  \kern-.5pt
  \xleaders\hbox{\kern.5pt\vrule height.4pt width \dash@width{#1}\kern.5pt}\hfill
  \kern-.5pt
}
\newcommand{\dash@width}[1]{%
  \ifx#1\displaystyle
    2pt
  \else
    \ifx#1\textstyle
      1.5pt
    \else
      \ifx#1\scriptstyle
        1.25pt
      \else
        \ifx#1\scriptscriptstyle
          1pt
        \fi
      \fi
    \fi
  \fi
}
\newcommand{\solidfill}[1]{\leaders\hrule\hfill}
\makeatother

\begin{document}

\title{Counting curves with local tangency constraints} 
\date{\today}

\begin{abstract}
We construct invariants for any closed semipositive symplectic manifold which count rational curves satisfying tangency constraints to a local divisor.
More generally, we introduce invariants involving multibranched local tangency constraints. We give a formula describing how these invariants arise as point constraints are pushed together in dimension four, and we use this to recursively compute all of these invariants in terms of Gromov--Witten invariants of blowups.
As a key tool, we study analogous invariants which count punctured curves with negative ends on a small skinny ellipsoid.
\end{abstract}

\author{Dusa McDuff}
\author{Kyler Siegel}
\thanks{K.S. was partially supported by NSF grant DMS--1606371.}

\maketitle

\tableofcontents

\section{Introduction}

\subsection{Motivation}\label{ss:intro}

A prototypical problem in enumerative geometry asks to count the number of curves in a space satisfying some specified geometric constraints.
These constraints are chosen so that one expects the answer to be a finite number, independent of any auxiliary choices. For example, an ancient observation is that there is exactly one line passing through any two distinct points in the plane.
A modern incarnation of this is the fact that there 
is a unique rational curve of degree one passing through any two distinct points in the complex projective plane $\CP^2$.
This naturally extends to the following question: how many rational curves of degree $d$ pass through $3d-1$ generic points in $\CP^2$?
Note that $3d-1$ is precisely the value needed to make the count a finite number $N_d$, 
and one can show that the answer does not depend on the locations of the points as long as they are in general position.
We have $N_1 = 1$ and $N_2 = 1$, the latter expressing the also long-known fact that five generic points determine a unique conic.
The computation $N_3 = 12$ was given by Steiner in 1848. The late 19th century 
brought enumerative geometry to a pinnacle and produced the number $N_4 = 620$ (see e.g. \cite{KV} for a more thorough history).
The computation of $N_d$ for all $d$ remained out of reach until the mid 1990's, when ideas from string theory began to infuse with algebraic geometry and symplectic topology, leading Kontsevich to discover the beautiful recursive formula
\begin{align}\label{eq:Kontalg}
N_d = \sum_{d_A + d_B = d} N_{d_A}N_{d_B} d_A^2d_B\left( d_B\mbinom{3d-4} {3d_A -2} - d_A\mbinom{3d-4}{3d_A-1}\right).
\end{align}
Using this formula we easily get $N_5 = 87304$, $N_6 = 26312976$, $N_7 = 14616808192$, and we can compute $N_d$ for any $d$ given enough computational power.

One of the key developments which catalyzed Kontsevich's formula was the introduction of Gromov--Witten invariants. Gromov--Witten invariants are defined in terms of stable maps from Riemann surfaces into a given target space. In favorable cases these invariants coincide with corresponding classical curve counts, although in general they count ``virtual'' objects which may not have straightforward classical interpretations. At any rate, the gain is that Gromov--Witten invariants have highly robust structural properties, and in particular they depend only on the underlying symplectic structure (up to deformation) of the target space. Rational Gromov--Witten invariants are used to cook up the quantum cup product on the cohomology of any symplectic manifold, and Kontsevich's formula follows rather directly from the observation that this product is associative.

Although Kontsevich's formula is the model success story, there are many other important enumerative problems.
We cannot possibly do this field justice in this introduction, but one natural extension is to consider a space together with a divisor, and to count curves which intersect the divisor in a specified number of points with specified tangency orders, plus possibly some additional constraints away from the divisor.
For example, it follows from the Caporaso--Harris recursion formula~\cite{CaH} that there are $7$ degree three curves in $\CP^2$ which intersect the line at infinity at a fixed point with tangency order $2$, plus pass through $5$ generic points away from the line at infinity.
The corresponding extension of Gromov--Witten invariants are called relative Gromov--Witten invariants and were implicitly defined in \cite{CaH} in the process of generalizing Kontsevich's formula to higher genus.
Relative Gromov--Witten invariants arise naturally even if one is a priori only interested in absolute counts, since they can be used to decompose the Gromov--Witten invariants of a symplectic manifold into simpler pieces via the degeneration or symplectic sum formula.

The starting point for this paper is a slightly different enumerative problem, namely the case of a {\em local} divisor rather than a global one.
The basic idea is to pick a smooth divisor $D$ defined near a point $p$ and to count curves which pass through $p$ with specified tangency order to $D$.
In symplectic geometry this idea goes back at least to the work of Cieliebak--Mohnke \cite{CM1} (we are not aware of any counterpart in the algebraic geometry literature).
As we explain in \S\ref{ss:Tanhigh}, one can adapt standard symplectic techniques to get well-defined Gromov--Witten type invariants, independent of all auxiliary choices (e.g. generic almost complex structure).
In particular, in contrast to the global divisor case the local divisor $D$ has no relevant invariants, such as degree.
For a closed symplectic manifold $M$, we will denote the resulting count of closed curves in $M$ in homology class $A$ which are tangent to $D$ at $p$ to order $m-1$\footnote{Equivalently, this corresponds to curves with contact order $m$ to $D$.} by $N_{M,A}\lll \T^{m-1} p\rrr$.

It turns out that there are some advantages to replacing several distinct point constraints with a single $\lll \T^{m-1} p\rrr $ constraint, stemming from the fact that the location of the constraint is more ``controlled''.
For instance, in \cite{CM2} Cieliebak--Mohnke compute that there are $(n-1)!$ degree one curves in $\CP^n$ satisfying a $\lll \T^{n-1}p\rrr$ constraint, 
i.e. $$
N_{\CP^n, [L]}\lll \T^{n-1}p\rrr = (n-1)!,
$$
and they use this to put strong restrictions
on the pseudoholomorphic disks bounded by Lagrangians with nonpositive curvature in $\CP^n$.
Their key idea  is that one can stretch the neck along the boundary of a Weinstein neighborhood of such a Lagrangian, with $p$ contained inside that neighborhood, and then the curves counted by  $N_{\CP^n, [L]}\lll \T^{n-1}p\rrr$ must break into quite specific configurations.
More recently, in \cite{tonk} Tonkonog uses a similar neck-stretching idea to describe a relationship between closed curve counts with $\lll \T^{m-1} p\rrr$ constraints and superpotentials of Lagrangian tori in symplectic manifolds, thereby connecting to conjectures in mirror symmetry.

For us, another main motivation for considering curves with local tangency constraints stems from their connections to symplectic embedding problems. Indeed, curves with local tangency constraints play a central role in the capacities recently defined in \cite{HSC}. 
It is explained there that, in contrast to the capacities defined with 
constraints at multiple points in the domain, 
the ones involving a single local tangency constraint are dimensionally stable. 
This principle is closely related to the observation of \cite{HK} that punctured rational curves with precisely one negative end have stable Fredholm index, a fact exploited in \cite{HK,HKerrat,CGH,Ghost,Mint} to give new obstructions for the stabilized ellipsoid embedding problem.
As we explain, there is an alternative characterization of tangency constraints via negative ends on a sufficiently skinny ellipsoid, thereby bridging our work with the punctured curves considered in \cite{HK,CGH,Ghost,Mint}. 
In a followup paper \cite{McDuffSiegel} we focus on punctured curves with tangency constraints in open symplectic manifolds and use these to give new symplectic embedding obstructions.
In a slightly different direction, the recent paper \cite{GanatraSiegel} applies these local tangency invariants to obstruct Liouville embeddings between Liouville domains.

\subsection{Main results}

We begin in \S\ref{sec:tangency} by defining the local tangency invariants for closed symplectic manifolds $(M,\omega)$, restricting our attention to the semipositive case (this includes all symplectic manifolds of dimension at most $6$; see \S\ref{ss:GW} for the precise definition).\footnote{More generally, it should be possible to define these invariants for all closed symplectic manifolds using an appropriate virtual perturbation framework such as polyfolds, but this lies beyond the scope of this paper.}
In particular, the invariant  $N_{M,A}\lll \T^{m-1}p\rrr$ counts genus zero somewhere injective $J$-holomorphic curves for a generic tame almost complex structure $J$ which are tangent to order $m-1$ to a local divisor $D$ through $p$. Using detailed index calculations in the presence of tangency constraints and some transversality results from Cieliebak--Mohnke \cite{CM1}, we prove that 
in any dimension
this count is independent of the point $p$, the local divisor $D$ and the almost complex structure $J$, and that it takes integer values.  

Our main results are in dimension four, and culminate in 
the following: 
\begin{thm}\label{thm:main_thm} For any symplectic four-manifold $M$ and homology class $A \in H_2(M;\Z)$,
there is an explicit recursive algorithm to compute the numbers $N_{M,A}\lll \T^{m-1}p\rrr$
in terms of Gromov--Witten invariants of blowups of $M$.
\end{thm}

 In contrast to algorithms (such as Kontsevich's in \eqref{eq:Kontalg}) that work recursively by reducing the degree, our algorithm works for fixed degree, decomposing the single constraint $\lll \T^{m-1}p\rrr$ into several lower order multibranched constraints using Theorem~\ref{thm:ppt_intro} below.  For an explicit sample calculation, see Example~\ref{ex:deg3}.

In the important special case of $\CP^2$, let us introduce the shorthand 
$$
T_d := N_{\CP^2,d [L]}\lll \T^{3d-2}p\rrr.
$$
We note that rational Gromov--Witten invariants of blowups of $\CP^2$ are well-understood and can be computed using e.g. the recursive algorithm of G\"ottsche--Pandharipande \cite[Theorem 3.6]{GP}.
Combining this with a computer implementation of Theorem~\ref{thm:main_thm}, we find
(see \S\ref{ss:recur1} for more detailed computations):
\begin{cor}
The first few values of $T_d$ are
\[
\begin{array}{llll}
T_1 = 1, & \;\;
T_2 = 1, &\;\;
T_3 = 4, &\;\;
T_4 = 26,\\
T_5 = 217, &\;\;
T_6 = 2110,&\;\; 
T_7 =  22744,&\;\;
T_8 = 264057.\\
\end{array}
\]
\end{cor}
\begin{remark}\label{rmk:lit_comp}\rm (i)
We point out that the local tangency invariants discussed in this paper are closely related to, but subtly different from, both stationary descendant Gromov--Witten invariants and  Gromov--Witten invariants relative to a global divisor. 
In short, the descendant counts generally involve extra contributions coming from configurations with ghost components, while global relative Gromow--Witten invariants tend to exclude curves contained in the divisor which could contribute to the analogous local tangency invariant.
This is discussed in more detail in \S\ref{subsec:new_subsec}. We note that the local tangency constraint is particularly well-suited for applications in symplectic geometry, especially for symplectic manifolds having no natural global divisors (c.f. \cite{HSC,chscI,McDuffSiegel,GanatraSiegel}).
\MS

\NI (ii)
As we explain later, some of the technical machinery entering into our proof of Theorem~\ref{thm:main_thm} works only in dimension four;
it is interesting to ask whether this result has a natural analog in higher dimensions. \hfill$\er$
\end{remark}

\MS

We now describe the main ideas and concepts underlying Theorem~\ref{thm:main_thm}, which are also of independent interest.
Our recursion is based on a general principle which describes what happens when constraints at different points in a 
 four dimensional
  target space $M$  are ``pushed together''.
A special case of this appears in the work of Gathmann \cite{Gath_blowup}. For the basic heuristic, 
 suppose that $q$ and $p$ are two points in 
 $M$, and suppose we have a curve which is constrained to pass through both $q$ and $p$.
Now consider what happens as we move $q$ towards $p$ along the tangent direction $v \in T_pM$.
In the limit as $q$ approaches $p$, there are two possibilities:
\begin{enumerate}
\item the pre-images in the domain curve remain distinct, and the limiting curve acquires a double point at $p$
\item the pre-images in the domain curve collide, and the limiting curve still passes through $p$ but now with tangent space constrained to be the complex line spanned by $v$.
\end{enumerate}
See Figure \cite[p.41]{Gath_blowup} for a cartoon.
Using his algorithm for Gromov--Witten invariants of blowups, Gathmann makes this heuristic precise by proving the following formula for curve counts in $\CP^2$:
\begin{align}\label{eq:Gath}
 \gw_{\CP^2,d [L]}\lll q,p,-\rrr  =  \gw_{\CP^2,d [L]}\lll \T p,-\rrr + 2\gw_{\bl^1\CP^2,d [L]-2[\Ee]}\lll -\rrr.
\end{align}
Here $\bl^1\CP^2$ denotes the $1$-point blowup of $\CP^2$, $[\Ee]$ denotes the homology class of the exceptional divisor, $[L]$ is the homology class of the line, and the symbol $-$ denotes some additional evaluation class constraints which we suppress from the notation.
Note that, for generic $J$, curves in $\bl^1\CP^2$ in homology class $d [L] - 2[\Ee]$ 
correspond bijectively to
 degree $d$ curves in $\CP^2$ with a simple double point at the blowup point. 
 
\MS

In \S\ref{ss:comb} we present a generalization of Gathmann's formula which describes what happens when more complicated geometric constraints at different points are pushed together.
In order to formulate it, we first introduce (in \S\ref{ss:tan_multi}) invariants counting curves 
(in a four dimensional manifold) 
with multibranched tangency constraints
\begin{align}\label{eq:NTPp0}
N_{M,A}\lll (\T^{m^1_1-1}p_1,...,\T^{m^1_{b_1}-1}p_1),...,(\T^{m^r_1-1}p_r,...,\T^{m^r_{b_r}-1}p_r)\rrr \in \Z
\end{align}
for $r \geq 1$, $b \geq 1$, and $m^i_j \geq 1$. 
Heuristically, given points  $p_1,\dots, p_r$, where $p_i$ lies  on a local divisor $D_i$,
 this invariant counts the number of rational curves in $M$ in homology class $A$ that, for each $i$,  have $b_i>0$ local branches through $p_i$ 
such that the $j$th local branch at $p_i$ is tangent to $D_i$ to order $m^i_j-1$.
Here we assume that the points $p_1,...,p_r \in M$ are pairwise distinct, and for instance the count 
$$
N_{M,A}\lll \underbrace{p_1,...,p_1}_{b_1},...,\underbrace{p_r,...,p_r}_{b_r} \rrr
$$ 
 coincides with the blowup Gromov--Witten invariant
$\gw_{M,A-b_1[\Ee_1]-...-b_r[\Ee_r]}$ (see Corollary~\ref{cor:T=Er}~(ii)).
More concisely, we will denote the invariant in \eqref{eq:NTPp0} by $$N_{M,A}\lll {\Pp_1},\dots, {\Pp_r}\rrr,$$
 where we introduce partitions $\Pp_i := (m^i_1,\dots,m^i_{b_i})$, and without loss of generality we assume $m^i_1 \geq ... \geq m^i_{b_i}$.
 We prove in \S\ref{ss:tan_multi} that these counts are well-defined and independent of all choices in dimension four (in higher dimensions our index bounds break down, necessitating transversality techniques beyond the scope of this paper).
We will see  in Example~\ref{ex:deg3} that, even in the case of cubics in $\C P^2$, these invariants are rather different from the Caporaso--Harris invariants 
mentioned earlier that count curves satisfying tangency conditions to a line.

The following theorem explicitly describes the outcome when two of the point constraints (which we can take to be $p_1,p_2$) are pushed together.

\begin{thm}\label{thm:ppt_intro} Let $(M,\om)$ be four dimensional and $A\in H_2(M;\Z)$.  
For any partitions $\Pp_1,...,\Pp_r$, we have
\begin{align*}
N_{M,A}\lll \Pp_1,\Pp_2,\Pp_3...,\Pp_r\rrr = \sum_{\Pp \in \partitions_{|\Pp_1| + |\Pp_2|}} \frac{\langle \Pp_1*\Pp_2,\Pp\rangle \;|\Aut(\Pp)|}{|\Aut(\Pp_1)| \; |\Aut(\Pp_2)|}   \;N_{M,A}\lll \Pp,\Pp_3,...,\Pp_r\rrr.
\end{align*}
\end{thm}

\NI Here $\partitions_k$ denotes the set of all partitions of $k$, $\Aut(\Pp)$ is defined in Definition~\ref{def:AutP} and the combinatorial term  $\langle \Pp_1*\Pp_2, \Pp\rangle $ is defined by \eqref{eq:*P}.

Observe that we can apply this formula iteratively to reduce constraints at $r$ different points to constraints at a single point.
In order to prove Theorem~\ref{thm:main_thm}, the idea (explained in \S\ref{sec:recur}) is to apply Theorem~\ref{thm:ppt_intro} ``in reverse'' in order to reduce curve counts with tangency constraints to curve counts without any tangency constraints, which in turn coincide with certain blowup Gromov--Witten invariants.
In fact, our recursive algorithm computes all of the multibranched tangency invariants $N_{M,A}\lll \Pp_1,\dots,\Pp_r\rrr$ for a symplectic four-manifold $M$ (detailed sample computations for the case of $\CP^2$ are provided in \S\ref{ss:recur1}).

\MS

Although one might be able to prove Theorem~\ref{thm:ppt_intro} purely in the context of tangency constraints, it turns out to be illuminating to reformulate this theorem as a result about punctured curves with ends on a  \lq\lq skinny ellipsoid'' $E_{\sk}$.
This perspective also embeds our work into the wide vista of symplectic field theory (SFT) and opens up some powerful tools such as neck stretching. 
More precisely, in \S\ref{sec:skin} we count curves in the symplectic completion of $M\less \io(E_{\sk})$  with negative ends on multiples of the short simple Reeb orbit in $\p E_{\sk}$. 
   Here, we take $\io: E_{\sk}\hookrightarrow M$ to be a symplectic  embedding of $$E_{\sk}:= E(\eps s_1, \eps s_2,\dots \eps s_n)$$ into $M$, 
   for $0 < s_1 \ll s_2 < \dots < s_n$ and $\eps > 0$ sufficiently small (see Definition~\ref{def:skin} for a more precise formulation).
We show in \S\ref{ss:skinhigh} that in any dimension there is a well-defined count of punctured curves with a single negative end on a skinny ellipsoid, and in particular this is independent of the shape parameters $\eps,s_1,\dots,s_n$ (provided that they satisfy appropriate inequalities).
In dimension four we also consider punctured curve counts involving several negative on a skinny ellipsoid, and we denote the resulting invariants $N^E_{M,A}\lll \Pp\rrr$, where $\Pp$ is a partition recording the negative asymptotic Reeb orbits.
We prove that these counts are well-defined and independent of all choices. In contrast to the closed curve counts, the presence of punctures necessitates that we utilize the SFT compactness theorem and also several new techniques, most notably a higher dimensional version of Siefring's intersection theory (see \S\ref{ss:skinhigh}) and the four-dimensional writhe estimates from embedded contact homology (see \S\ref{ss:skin4D}).

In \S\ref{ss:same}, we prove that these two approaches give the same counts, at least in dimension four. More precisely, we have:
\begin{thm}\label{thm:counts}  Let $(M,\om)$ be four dimensional and $A\in H_2(M;\Z)$.  Then for every partition $\Pp$ of $m: = c_1(A)-1$ we have
$$
N^E_{M,A}\lll \Pp\rrr  =  N_{M,A}\lll \Pp\rrr.
$$
Further, putting $\Pp = (m_1,\dots,m_b)$ and $\Pp!: = \frac {m!}{m_1!\dots m_b!}$, we have
$$
\gw_{M,A}\lll p_1,\dots,p_{m}\rrr = \sum_{\Pp \in \partitions_m} \Pp! \, N_{M,A}\lll \Pp\rrr.
$$
\end{thm}

\begin{remark}\rm
It follows from Corollary~\ref{cor:count_pos} that we have $T_d > 0$ for all $d \in \Z_{\geq 1}$.
This observation, which is not manifest from the recursion, is important for applications to stabilized symplectic embeddings, and it is equivalent to the nontrivial fact that there exist degree $d$ punctured curves in $\CP^2 \setminus E_{\sk}$ with one negative end on the $(3d-1)$-fold cover of the short simple Reeb orbit --- see \cite{HKerrat,Mint}.
\end{remark}

Our proof of Theorem~\ref{thm:counts} in \S\ref{ss:same} utilizes 
the delicate methods of embedded contact homology (ECH)
in order to bypass technical difficulties about gluing curves with tangency conditions, 
while Wendl's automatic transversality results  implies that both of these counts are nonnegative integers.  
In \S\ref{ss:comb}, we prove Theorem~\ref{thm:ppt_intro} by combining Theorem~\ref{thm:counts} with a neck stretching argument and an application of the obstruction bundle gluing results of Hutchings--Taubes.
 Although most of these methods work only in dimension four,  Moreno--Siefring \cite{Sief2,MorS} have extended some 
 parts of the theory to higher dimensions, a result that we use in our proof in \S\ref{ss:skinhigh} that the count of curves with one end on a skinny ellipsoid is well defined in all dimensions. 
Finally, in \S\ref{ss:recur2} we present the recursive algorithm of Theorem~\ref{thm:main_thm} as a combinatorial application of Theorem~\ref{thm:ppt_intro}.

\vspace{1em}

\NI \textbf{Acknowledgements} We wish to thank the referees very warmly for their (sometimes very detailed) comments, which have all greatly helped to improve the accuracy and clarity of the paper.  
We also want to acknowledge informative discussions with Melissa Liu.

\section{Curves with local tangency constraints}\label{sec:tangency}

In this section we define Gromov--Witten type invariants counting 
pseudoholomorphic curves with various types of tangency constraints in any semipositive symplectic manifold. After reviewing the symplectic approach to Gromov--Witten theory in the semipositive case in \S\ref{ss:GW}, we consider the case of a single tangency constraint at a point in \S\ref{ss:Tanhigh}. In \S\ref{ss:tan_multi}, we mostly restrict to dimension four and generalize this to tangency constraints involving several branches of a curve passing through the same point in the target.

\subsection{Review of symplectic Gromov--Witten theory}\label{ss:GW}

As a prelude to defining curve counts with local tangency constraints, in this subsection we give a brief review of Gromov--Witten theory for semipositive symplectic manifolds as in \cite{JHOL}.
Let $(M^{2n},\omega)$ be a closed symplectic manifold.
We will denote its (primary) Gromov--Witten invariants by 
$$
\gw_{M,g,A}\lll \gamma_1,...,\gamma_k\rrr \in \Q,
$$ 
where $A \in H_2(M;\Z)$ is a homology class and $\gamma_1,...,\gamma_k \in H^*(M;\Z)$ are cohomology classes. 
These numbers are independent of all choices made during the construction (most notably that of an almost complex structure) and are invariant under symplectomorphisms and a fortiori under symplectic deformation equivalences.\footnote
{Recall that two symplectic manifolds $(M,\omega)$ and $(M',\omega')$ are said to be symplectic deformation equivalent if there is a diffeomorphism $\Phi: M \rightarrow M'$ such that $\Phi^*(\omega')$ can be joined to $\omega$ by a $1$-parameter family of symplectic forms on $M$.} 
Roughly, these invariants count (for a suitable meaning of ``count'')  the number of genus $g$ pseudoholmorphic curves in $M$ in the homology class $A$ which pass through cycles in $M$ representing the Poincar\'e dual homology classes of $\gamma_1,...,\gamma_k$.
Although Gromov--Witten invariants have been defined in both the algebraic and symplectic categories in great generality, for concreteness we will restrict ourselves to the case that $(M,\omega)$ is {\bf semipositive}, i.e. for any $A \in H_2(M;\Z)$ with $[\omega] \cdot A > 0$ and $c_1(A) \geq 3 - n$ we have $c_1(A) \geq 0$.\footnote{Here we denote by $c_1(A)$ the pairing of the first Chern class of the symplectic manifold $(M,\omega)$ with the homology class $A$.} 
We will also restrict throughout to the case of rational (i.e. genus zero) curves and hence we will omit $g$ from the notation.
The discussion here of Gromov--Witten invariants and the subsequent generalization to curve counts with tangency constraints can be extended to more general symplectic manifolds using various virtual perturbation schemes or e.g. the approach of \cite{CM1}, but we will not need to do so in this paper.
We note that the class of semipositive symplectic manifolds is already quite large and includes all symplectic manifolds of dimension less than or equal to six. 

Let $J$ be an $\omega$-tame almost complex structure.\footnote{Recall that an almost complex structure is {\em $\omega$-tame} is $\omega(v,Jv) > 0$ for all nonzero tangent vectors $v \in TM$. An almost complex structure is furthermore called {\em $\omega$-compatible} if $\omega(-,J-)$ is a Riemannian metric. Tameness tends to be sufficient for most purposes.}
We denote by $\calM_{M,A}^{J}$ the moduli space\footnote{By default ``moduli space'' will mean that we have already quotiented out by the relevant group of biholomorphisms of the domain acting by reparametrizations.} of rational $J$-holomorphic spheres in $M$ representing the homology class $A \in H_2(M;\Z)$.
That is, we put
\begin{align*}
\calM_{M,A}^J := \left\{ u: S^2 \rightarrow M \;:\; du \circ j = J \circ du,\; [u] = A\right\} / \Aut(S^2),
\end{align*}
where $j$ denotes the standard (integrable) almost complex structure on the Riemann sphere $S^2$ and $\Aut(S^2) = \PSL(2,\C)$ is its biholomorphism group.
We denote by 
$$ \calM^{J,\simp}_{M,A} \subset \calM^J_{M,A}$$
the subspace of {\bf simple} (or equivalently {\bf somewhere injective}) curves, i.e. those not factoring through any branched cover of the domain.
By a standard argument involving the Sard--Smale theorem, for generic $\omega$-tame $J$ every curve in $\calM^{J,\simp}_{M,A}$ is regular (i.e. its associated linearized Cauchy--Riemann operator is surjective), and hence $\calM_{M,A}^{J,\simp}$ is a smooth (but not necessarily compact) manifold. 
Here and henceforth we say that a condition holds for {\em generic} $\omega$-tame $J$ if the subset of all $\omega$-tame almost complex structures $J$ for which the condition holds is Baire, i.e. is a countable intersection of dense open sets.
Moreover, the (real) dimension of $\calM_{M,A}^{J,\simp}$ is given by the Fredholm index of any representative curve $u$, which can be shown (using say Riemann--Roch or Atiyah--Singer) to be 
$$ \ind(u) = 2n - 6 + 2c_1(A).$$

Similarly, we define the moduli space $\calM_{M,A,k}^{J}$ of $J$-holomorphic spheres in $M$ in the homology class $A$ with $k$ distinct (ordered) marked points,
as well as the subspace $\calM_{M,A,k}^{J,\simp} \subset \calM_{M,A,k}^{J}$ of simple curves.
We have a natural evaluation map
$$ \ev: \calM^{J}_{M,A,k} \rightarrow \underbrace{M \times ... \times M}_k =: M^{\times k}.$$
The key point of semipositivity is that it leads to the following result.

\begin{prop}\label{prop:pseu_basic}{\rm(\cite{JHOL})}
Let $(M,\omega)$ be a semipositive symplectic manifold, and let $A \in H_2(M;\Z)$ be a homology class.
Assume further that $A$ is not of the form $\kappa B$ for $B \in H_2(M;\Z)$ with $c_1(B) = 0$ and $\kappa > 1$.
For generic $\omega$-tame $J$, the evaluation map $\ev$ defines a pseudocycle in $M^{\times k}$ of dimension $2n-6+2c_1(A) + 2k$. Moreover, this pseudocycle is independent of $J$ up to pseudocycle bordism.
\end{prop}

Recall (see \cite[Ch.6]{JHOL}) that a $d$-dimensional {\bf pseudocycle} in a smooth manifold $Q$ is by definition
a smooth map $f: V \rightarrow Q$, with $V$ an oriented $d$-dimensional smooth manifold, such that
$\ovl{f(V)} \subset Q$ is compact and $\dim \Omega_f \leq d-2$.
Here $\Omega_f$ is the limit set, defined by 
$$\Omega_f := \bigcap\limits_{\substack{K \subset V\\ K\text{ compact}}}\ovl{f(V \setminus K)},$$
and the inequality $\dim \Omega_f \leq d-2$ means that $\Omega_f$ is contained in the image of a smooth map $g: W \rightarrow Q$ for $W$ a smooth manifold of dimension $\dim W \leq d-2$.
A bordism between $d$-dimensional pseudocycles $f_0: V_0\rightarrow Q$ and $f_1: V_1 \rightarrow Q$ is a smooth map $F: W \rightarrow X$
with $W$ a smooth manifold of dimension $d+1$ such that $\bdy W = V_1 \cup (-V_0)$, $F|_{V_0} = f_0$, $F|_{V_1} = f_1$, and
$\dim \Omega_F \leq d-1$.

\begin{proof}[Proof sketch of Proposition \ref{prop:pseu_basic}]
The basic idea is to show that the evaluation map $\ev: \calM_{M,A,k}^{J,\simp} \rightarrow M^{\times k}$ extends to a compactification
$\ovl{\calM}_{M,A,k}^J$  in such a way that the image of the 
added points has codimension at least two, i.e.
\begin{align*}
\dim \ev\left( \ovl{\calM}^J_{M,A,k} \setminus \calM_{M,A,k}^{J,\simp}\right) \leq 2n - 8 + 2c_1(A) + 2k.
\end{align*}
We take $\ovl{\calM}^J_{M,A,k}$ to be the usual stable map compactification of $\calM^J_{M,A,k}$, which is defined by allowing multiple covers and by adding various nodal configurations indexed by decorated graphs.
First consider curves $u \in \calM^J_{M,A,k} \setminus \calM^{J,\simp}_{M,A,k}$,
i.e. multiple covers with smooth domain in homology class $A$. 
Since $u$ is multiply covered, we cannot necessarily assume that it is regular for generic $J$, so it might appear in a family with higher than expected dimension. However, note that the image of $u$ under the evaluation map is the same as that of its underlying simple curve $\ovl{u}$.
Since we can assume that the curve $\ovl{u}$ 
is regular and hence appears with its expected dimension, it suffices to show that we have $\ind(\ovl{u}) \leq \ind(u) - 2$.
Say that $u$ is a $\kappa$-fold cover of $\ovl{u}$ for some $\kappa > 1$, so that $\ovl{u}$ lies in the homology class $A/\kappa$.
We can also assume that the curve $\ovl{u}$ after forgetting the $k$ marked points is regular, and hence its index is nonnegative, i.e. we have
$2c_1(A)/\kappa + 2n - 6 \geq 0$.
Therefore by semipositivity we must have $0 \leq \frac{1}{\kappa}c_1(A) \leq c_1(A)$.
By the assumption on the homology class $A$ we must in fact have strict inequalities $0 < \frac{1}{\kappa}c_1(A) < c_1(A)$.
This implies that we have $\ind(\ovl{u}) \leq \ind(u) - 2$, as desired.

Now we consider nodal configurations $u \in \ovl{\calM}_{M,A,k}^J \setminus \calM_{M,A,k}^J$.
Due to the presence of a node, the expected codimension of $u$ is at least two, but since $u$ may involve one or more multiply covered components, it could appear in a family with higher than expected dimension. Fortunately, every nodal configuration has an underlying simple configuration with the same image under $\ev$.
Roughly, this is obtained by replacing each multiply covered component by its underlying simple component, and
discarding repetitions of components when two or more have the same image.
 With some care, we can arrange that the result is still a connected stable curve modeled on a tree; see \cite[Prop.6.1.2, Thm.6.6.1]{JHOL} for more details.
Each such simple configuration is regular for generic $J$ and hence appears with its expected dimension. Arguing as above, semipositivity then guarantees that the dimension of the image of the corresponding boundary stratum of $\ovl{\calM}_{M,A,k}^J$ is at least two less than that of the main stratum $\calM_{M,A,k}^{J,\simp}$.

As for the bordism statement, the argument is similar. 
Namely, given two generic $\omega$-tame almost complex structures $J_0$ and $J_1$, we can pick a family $\{J_t\;:\; t \in [0,1]\}$ of $\omega$-tame almost complex structures interpolating between them. 
We can then consider the $t$-parametrized moduli space
$\calM_{M,A,k}^{\{J_t\},\simp}$ consisting of all pairs $(t,u)$ with $t \in [0,1]$ 
and $u \in \calM_{M,A,k}^{J_t,\simp}$. For a generic such family $\{J_t\}$, this parametrized moduli space is a smooth manifold of dimension $2n - 5 + 2c_1(A)$ and comes with a natural evaluation map to $M^{\times k}$.
Arguing as above, this defines a pseudocycle bordism between the pseudocycles defined for the $t=0$ and $t=1$ data.
\end{proof}

It was shown by Schwarz~\cite{Sch} that pseudocycles up to bordism are equivalent to integral homology classes.
Thus we
 get 
a well-defined integral homology class, which we denote by 
$$
\bigl[\ovl{\calM}_{M,A,k}\bigr]\;\;\in\;\; H_{2n - 6 + 2c_1(A) + 2k}(M^{\times k};\Z).
$$
Given cohomology classes $\gamma_1,...,\gamma_k \in H^*(M;\Z)$ of total summed degree $2n-6+2c_1(A) + 2k$, we define Gromov--Witten invariants as per the usual prescription by
$$
\gw_{M,A}\lll \gamma_1,...,\gamma_k\rrr := (\pi_1^*\gamma_1\cup ... \cup \pi_k^*\gamma_k) \cdot \bigl[\ovl{\calM}_{M,A,k}\bigr] \in \Z,
$$
where $\pi_i: M^{\times k} \rightarrow M$ denotes the projection to the $i$th factor.
Strictly speaking, we must have $k \geq 1$ in order to have an evaluation map and hence a pseudocycle, but in the case $c_1(A) = 3-n$ we can define the invariant $\gw_{M,A} \in \Z$ not involving any marked points as follows:
\begin{align}\label{eq:k=00}
 \gw_{M,A} : =  \frac 1{\ga(A)} \gw_{M,A,1} \lll \ga\rrr\;\in\; \Z,
\end{align}
where $\ga \in H^2(M;\Z)$ is any cohomology class such that $\ga(A) \ne 0$.
This behaves as expected and is independent of the choice of $\gamma$ in light of the divisor axiom of Gromov--Witten theory (see \cite[Prop. 7.5.7]{JHOL}), which reads 
$$
\gw_{M,A,k+1}(\gamma_1,\dots,\gamma_k,\gamma) = {\rm GW}_{M,A,k}(\gamma_1,\dots,\gamma_{k})\cdot \int_A \gamma
$$
for any $k \geq 1$ and cohomology classes $\gamma_1,\dots,\gamma_k,\gamma \in H^2(M;\Z)$.
Note that if $\gamma$ is Poincar\'e dual to a $J$-holomorphic submanifold $D_\ga$, then for generic $J$ every $A$-curve (not entirely contained in $V$) meets $D_\ga$ transversely in $\int_A \gamma$ distinct points. 
Moreover, by \cite[Prop.8.11]{CM1}, given $A$ one can choose $J, D_\ga$ so that no holomorphic curve of energy $\le \om(A)$ is entirely contained in $D_\ga$.

\begin{remark}\label{rmk:multcover0}\rm
As explained in \cite[Ch.6.7]{JHOL}, one can also remove the assumption on the homology class $A$ in Proposition \ref{prop:pseu_basic} while still working in the realm of classical transversality techniques. 
This involves using domain-dependent almost complex structures and generally leads to rational (rather than integral) Gromov--Witten invariants.
Also, note that in dimension four the Gromov--Witten invariant $\gw_{M^4,A}$ can only be nonzero if $c_1(A) \geq 1$, i.e. the situation $c_1(A) = 0$ does not arise for index reasons.
\hfill$\er$
\end{remark}

\subsection{Curves with a single tangency constraint}\label{ss:Tanhigh}

Let $(M,\omega)$ be a semipositive symplectic manifold.
In this subsection we define the numbers $N_{M,A}\lll \T^{m-1}p\rrr$
counting pseudoholomorphic spheres in $M$ in homology class $A$ satisfying an order $m-1$ local tangency constraint at a point $p \in M$, where $m = c_1(A)-1$.
As in the case of Gromov--Witten invariants, these numbers are independent of all choices involved in the construction, including the point $p \in M$ (but we find it helpful to nevertheless include it in the notation).
Although these invariants could in principle be defined in much greater generality, in the semipositive case they are defined rather concretely as counts of somewhere injective curves and take values in the integers.

Following Cieliebak--Mohnke \cite{CM1}, let $J$ be a $\omega$-tame almost complex structure on $M$ which is integrable in a small neighborhood $\Op(p)$\footnote{In general we will adopt Gromov's convention that $\Op(p)$ denotes a small unspecified open neighborhood of $p$.} of a point $p \in M$,
and let $D$ be a smooth complex codimension one holomorphic submanifold in $\Op(p)$ which passes through $p$.
Given a Riemann surface $\Sigma$ with a marked point $z$ and a $J$-holomorphic map $u: \Sigma \rightarrow M$ with $u(z) = p$, 
we say that $u$ has {\bf tangency order} $m-1$\footnote{We note that by tangency order $m-1$ we will always mean tangency order $m-1$ {\em or greater}, unless explicitly stated otherwise.}
 (or equivalently {\bf contact order} $m$) to $(D,p)$ at the marked point $z$ if we have
\begin{align}\label{eq:gu}
 \frac{d^j (g\circ u\circ f)}{d\zeta^j}\bigr|_{\zeta=0} = 0\;\;\;\;\;\mbox{ for } j = 1,...,m-1,
\end{align}
where $f: \C \supset \Op(0) \rightarrow \Op(z) \subset \Sigma$ is a choice of local complex coordinates for $\Sigma$ with $f(0) = z$,
and $g: M \supset \Op(p) \rightarrow \C$ is a holomorphic function such that $D = g^{-1}(0)$ and $dg(p)\ne 0$.
As shown in \cite{CM1}, this notion is independent of the choice of $f$ and $g$,
and it only depends on the germ of $D$ near $p$.
Assuming $m$ is maximal such that $u$ is tangent to $(D,p)$ to order $m-1$ at $z$, we will denote by 
\begin{align}\label{eq:orddef}
\ord(u,D;z) = m
\end{align}
the {\bf local contact order} of $u$ to $(D,p)$ at $z$. It is shown in \cite[Prop~7.1]{CM1} that $\ord(u,D;z)$ can indeed be interpreted as the local intersection number of $\im u$ with $D$ in the following sense: if $U$ is a small neighborhood of $z$ such that $u^{-1}(p) \cap U = \{z\}$, then any small generic $J$-holomorphic perturbation of $u|_{U}$ has precisely $\ord(u,D;z)$ transverse intersection points with $D$. 

Let $\calM^{J,\simp}_{M,A}\lll \T^{m-1}p\rrr$ denote the moduli space of simple rational curves in $M$ in homology class $A$ which are tangent to order $m-1$ to $(D,p)$ at a marked point $z$. 
Let $\Jj_D$ denote the set of $\omega$-tame almost complex structures on $M$ which are integrable on $\Op(p)$ and for which $D$ is a holomorphic submanifold. 
 We have: 

\begin{lemma} \label{lem:cusps} {\rm (\cite[Prop. 6.9]{CM1})}
For generic $J \in \Jj_D$, the 
 space $\calM^{J,\simp}_{M,A}\lll \T^{m-1}p\rrr$ is a smooth, oriented  (but not necessarily compact) manifold of dimension 
\begin{align}\label{eq:indA}
2(n-3) +2c_1(A) + (2-2n) - 2(m-1) = 2c_1(A) - 2 -2m.
\end{align}
Moreover, if  $J_0, J_1\in \Jj_D$ are two such generic elements they may be joined by a generic $1$-parameter family $J_t$ such that the manifold 
$\bigcup\limits_{t\in [0,1]}\calM^{J_t,\simp}_{M,A}\lll \T^{m-1}p\rrr$ provides an oriented  (but not necessarily compact) cobordism between the manifolds at $t=0,1$.
\end{lemma}
\begin{proof}
This lemma slightly extends \cite[Prop. 6.9]{CM1} since our divisor $D$ is local rather than global and 
we have fixed the intersection point $p$ on $D$.   However, once $p$ is fixed the rest of the divisor is irrelevant, and the extension to the case when $p$ is fixed is mentioned in the followup paper~\cite[Prop~3.1]{CM2}.
Moreover, as mentioned in \cite[Prop. 10.2]{CM1} the extension from the case of generic fixed $J$ to generic $1$-parameter families of $J$ follows exactly as in the analogous result in \cite[Thm.3.1.7]{JHOL}.\end{proof}

Note that the summand $(2-2n)$ corresponds to the point constraint at $p$ and the summand $-2(m-1)$ corresponds to the tangency condition to $D$,
so that the resulting  formula is independent of the dimension $n$.
We can also consider the moduli space $\calM^{J,\simp}_{M,A,k}\lll \T^{m-1} p\rrr$ defined in the same way except that the domains of curves are equipped with $k$ additional (ordered) marked points.

\begin{prop}\label{prop:tanpseud}
Let $(M,\omega)$ be a semipositive symplectic manifold of dimension $2n$.
Then, for generic $J \in \Jj_D$, the evaluation map $\ev: \calM^{J,\simp}_{M,A,k}\lll \T^{m-1} p \rrr \rightarrow M^{\times k}$
defines a pseudocycle of dimension $ 2c_1(A) - 2 - 2m + 2k$
which is independent of $J$,$p$, and $D$ up to pseudocycle bordism.
\end{prop}

\begin{proof}
We proceed along similar lines to the proof of Proposition~\ref{prop:pseu_basic}.
The basic strategy is to show that the evaluation map 
$\ev:\calM^{J,\simp}_{M,A,k}\lll \T^{m-1} p\rrr \rightarrow M^{\times k}$ extends to a compactification
$\ovl{\calM}^{J}_{M,A,k}\lll \T^{m-1} p \rrr$ in such a way that the image of the added points has  codimension at least two, i.e.
$$
\dim \ev\left(\ovl{\calM}^J_{M,A,k}\lll \T^{m-1} p \rrr \setminus \calM^{J,\simp}_{M,A,k}\lll \T^{m-1}  p\rrr \right) \;\; \leq \;\; 2c_1(A) - 4 - 2m + 2k.
$$
Note that there is a forgetful inclusion of $\calM_{M,A,k}^{J,\simp}\lll \T^{m-1}p\rrr$ into $\ovl{\calM}_{M,A,k+1}^J$, where the last marked point comes from the one satisfying the $\lll\T^{m-1}p\rrr$ constraint. 
We define the compactification $\ovl{\calM}_{M,A,k}^J$ to be simply the closure of $\calM_{M,A,k}^{J,\simp}\lll \T^{m-1}p\rrr$ in this ambient compact space.
This compactification contains  curves of four types
\begin{itemize}\item[-]  multiply covered curves with smooth domain;
\item[-] nodal curves in which the main component (i.e. that  containing the last marked point) is not constant;
\item[-]  nodal curves in which the main component is constant but no adjacent component (see below) is multiply covered or repeated;
\item[-] nodal curves in which the main component is constant and some adjacent component (see below) is multiply covered or repeated.
\end{itemize}
We now show that in each of these cases the dimension of the image of the relevant stratum under the evaluation map is at least two less than that of the main stratum.  It then follows as in \cite[Thm.6.6.2]{JHOL} that the evaluation map on the main stratum defines a pseudocycle.

Firstly, consider curves $$
u \in \calM^J_{M,A,k}\lll\T^{m-1}p\rrr \setminus \calM^{J,\simp}_{M,A,k}\lll \T^{m-1}p\rrr,
$$
 i.e. multiply covered curves with smooth domain.
If $u$ is a $\kappa$-fold cover of its underlying simple curve $\ovl{u}$, note that $\ovl{u}$ lies in the homology class $A/\kappa$ and satisfies the constraint $\lll\T^{\ovl{m}-1}p\rrr$ for some $\ovl{m} \geq m/\kappa$.
Now equip $\ov u$ with all possible choices of marked points
$(z_1,\dots,z_k)\in (S^2)^k$. 
By regularity $\ovl{u}$ appears with its expected dimension and it suffices to show that $\ind(\ovl{u}) \leq \ind(u) - 2$.
Since $\ovl{u}$ is simple and $J$ is generic, we can further assume that the same curve $\ovl{u}$ after forgetting the $k$ marked points is regular and hence has nonnegative index, so we have 
\begin{align}\label{eq:ovum0} \notag
0 &\leq   2c_1(A)/\kappa - 2 -  2\ov{m} \\ \notag
& \le 2(c_1(A) -m)/\kappa - 2 \\&\le 2(c_1(A) -m) - 4,
\end{align}
where the last inequality holds because $0\le (c_1(A) -m)/\kappa - 1$ implies that   $c_1(A) -m\ge \ka \ge 2$.
We then have 
\begin{align}\label{eq:ovum} \notag
\ind(\ovl{u}) &= 2c_1(A)/\kappa - 2 - 2\ovl{m} + 2k\\  \notag
&\leq 2c_1(A) - 4 - 2m + 2k\\ &= \ind(u) - 2,
\end{align}
as desired.

Next, we must consider curves $u \in \ovl{\calM}^{J}_{M,A,k}\lll \T^{m-1}p\rrr \setminus \calM^{J}_{M,A,k}$,
i.e. nodal degenerations of curves with smooth domain satisfying the tangency constraint $\lll\T^{m-1}p\rrr$.
We will refer to the component of $u$ containing the last marked point $z$ (i.e. the one satisfying the tangency constraint) as the ``main'' one.
Firstly, consider the case that the main component of $u$ is not constant. 
Although $u$ might have one or more multiply covered components, as in the proof of 
Proposition~\ref{prop:pseu_basic} we can pass to the underlying simple configuration $\ovl{u}$. We can assume that $\ovl{u}$ is regular and hence appears in a family of the expected dimension. 
By semipositivity, each component of $u$ lies in a class $A_\al$ with $c_1(A_\al)\ge 0$, 
so that the total Chern class of 
$\ovl{u}$ is at most that of $u$ and hence at most $c_1(A)$. 
 But because $\ovl{u}$ is nodal, the dimension of its image is at least $2$ less than  that of a nonnodal curve; see the proof of \cite[Thm.6.2.6]{JHOL}.
 Thus the expected codimension of $\ovl{u}$ is at least two.

Secondly, we must consider the possibility that the main component of $u$ is a ghost (i.e. constant).  
This case is slightly more subtle because at first glance the tangency constraint appears to be lost, which makes it difficult to argue that $u$ only appears 
in a stratum of the moduli space whose image has codimension at least two.
Indeed, note that, according to the definition in ~\eqref{eq:gu},  a constant map at $p$ is automatically tangent to $(D,p)$ to all orders. 
On the other hand, it turns out that the tangency constraint $\lll \T^{m-1}p\rrr$ gets redistributed amongst the components of $u$ which are ``close to'' the main one. 
More precisely, let $u_1,...,u_a$ (for $a\ge 1$)  denote the nonconstant components of $u$ which are adjacent to the main component $u_0$, or more generally are adjacent to some ghost component of $u$ which is connected to $u_0$ through ghost components.
For $i = 1,...,a$, let $\wt{z}_i \in {\rm dom\,} u_i$ denote the relevant special point of $u_i$ which participates in the node realizing this adjacency. 
Then at the marked point $\wt{z}_i$
we have $u_i(\wt{z}_i) = p$ and $u_i$  
has contact with $D$
to some order $\ord(u,D;\wt{z}_i) \geq 1$,
where $\ord(u,D;\wt{z}_i)$ is defined in \eqref{eq:orddef}.
According to \cite[Lemma 7.2]{CM1}, in this situation we have 
\begin{align}\label{eq:ord}
\sum_{i=1}^a \ord(u,D;\wt{z}_i)  \geq m,
\end{align}
i.e. the total local contact order to $D$ is nondecreasing as a family of smooth curves in $\calM^{J}_{M,A,k}\lll \T^{m-1}  p\rrr$ degenerates to $u$
(see Lemma~\ref{lem:cusp_deg} below for a slightly more general statement).
Further, after we have discarded the ghost components that are connected to $u_0$, the curve has $a$ connected branches, each containing one of the points $\Tz_i$.

Now consider  the (nodal, possibly disconnected) curve $\ovl{u}$ obtained as follows:
\begin{enumerate}
\item remove the main component of $u$, along with all other ghost components which are connected to it through ghost components
\item 
add a new marked point at each of the nearby special points $\wt{z}_i$ on a nonconstant component $u_i$ as above
\item replace each multiply covered component by its underlying simple curve, 
discard extra 
components (other than the $u_i$) in the same branch when two or more have the same image, and finally parametrize each 
branch  of the resulting nodal curve by a new tree   (as in the proof of Proposition~\ref{prop:pseu_basic}).
 \end{enumerate}
Observe that $\ovl{u}$ has the same image under the evaluation map as $u$ (technically we may have removed some of the original $k$ unconstrained marked points but using the fact that they all map to $p$ we recover the missing components of the evaluation map).
If we assume that the maneuver (3) does not involve any components adjacent to the removed ghost components,
 then since $\ovl{u}$ satisfies the tangency constraint $\lll \T^{\ord(u,D;\wt{z}_i)-1}p\rrr$ at each of the new marked points $\wt{z}_i$, 
it follows from \eqref{eq:ord} 
 that steps (1) and (2) do not reduce the codimension of the constraint,
so the index of $\ovl{u}$ is at most that of $u$, 
which implies in turn that its image has dimension at most that of $u$. 
To see that we in fact have $\ind(\ovl{u}) \leq \ind(u) - 2$,
observe that  
because, by stability, each ghost component has at least three special points,
 one of the following must occur:
\begin{enumerate}[label=(\alph*)]
\item  one or more of the unconstrained marked points $z_1,...,z_k$ maps to $p$, or
\item  $\ovl{u}$ is disconnected (recall that the  main component has been removed) and has  $a > 1$ nonconstant components.
\end{enumerate}
Case (a) clearly involves an additional constraint of codimension at least two.
So does case (b) 
because as in \cite[Thm.6.2.6]{JHOL} each node reduces the index by $2$.
(For example, instead of a single curve that is tangent to $D$ at $p$ [of codimension $4$], we might have a nodal curve whose two components both go through $p$ [of codimension $6$].)
 Since such curves $\ovl{u}$ are simple, we can assume that they are regular and hence appear with the expected dimension. It follows that 
 the image under the evaluation map of
 the stratum containing $u$ has codimension at least two.
 
 It remains to check that the above claims still hold if one (or more) of the 
 nonconstant 
 components $u_i$ that carry some part of the  tangency constraint is multiply covered.   In this case, we replace $u_i$ satisfying $\lll\T^{m_i-1}p\rrr$ with $\ov u_i$ 
 satisfying $\lll\T^{\ov m_i-1}p\rrr$ where $u_i = \ka\ov u_i$ and $m_i \le \ka \ov m_i$.  The calculation in \eqref{eq:ovum} shows that 
 the contribution of the component $\ov u_i$ to the total index is at least two less  than that of $u_i$.  Hence again the image under the evaluation map of the stratum containing such curves   has codimension at least two.

As for the bordism statement, the argument is similar.
Namely, suppose we have local divisors $D_0,D_1$ at points $p_0,p_1 \in M$ respectively, and let $J_0,J_1$ be corresponding generic almost complex structures on $M$ such that $J_i$ is integrable near $p_i$ and makes $D_i$ holomorphic. 
Then we can pick a path $p_t$ from $p_0$ to $p_1$, a family $D_t$ of local divisors at $p_t$ interpolating between $D_0$ and $D_1$, and a family $J_t \in \Jj_{D_t}$ of $\omega$-tame almost complex structures interpolating between $J_0$ and $J_1$ such that $J_t$ is integrable near $p_t$ and makes $D_t$ holomorphic.
We can then consider the $t$-parametrized moduli space $\calM^{\{J_t\},\simp}_{M,A,k}\lll \T^{m-1}  p\rrr$ consisting of all pairs $(t,u)$ with $t \in [0,1]$ and $u \in \calM^{J_t,\simp}_{M,A,k}\lll \T^{m-1}p\rrr$.
As in the case of a single almost complex structure, if we take the family $J_t$ to be generic, we have that 
$\calM^{\{J_t\},\simp}_{M,A,k}\lll \T^{m-1}  p\rrr$ is a 
smooth manifold of dimension $$
2c_1(A) - 1 - 2m + 2k,
$$
and it comes equipped with a natural evaluation map to $M^{\times k}$.
Arguing exactly above, this defines a pseudocycle bordism between the pseudocycles defined for the $t = 0$ and $t=1$ data.
\end{proof}

\begin{remark}
\rm
We point out that, in contrast to Proposition~\ref{prop:pseu_basic}, in Proposition~\ref{prop:tanpseud} we do not need 
any assumption ruling out multiples of homology classes with vanishing first Chern class, since  the presence of the point constraint at $p$ forces $c_1(A)>0$.
\hfill$\er$ 
\end{remark}

It follows from Proposition~\ref{prop:tanpseud} that we have a well-defined homology class 
$$
\bigl[\ovl{\calM}_{M,A,k}\lll \T^{m-1}p\rrr\bigr]\;\;\in\;\; H_{2c_1(A) - 2 - 2m + 2k}(M^{\times k};\Z).
$$
Given cohomology classes $\gamma_1,...,\gamma_k \in H^*(M;\Z)$ with total summed degree $2c_1(A) - 2 - 2m + 2k$, we define numerical invariants by
$$
N_{M,A}\lll \T^{m-1}p,\gamma_1,...,\gamma_k\rrr := (\pi_1^*\gamma_1 \cup ... \cup \pi_k^*\gamma_k) \cdot \bigl[\ovl{\calM}_{M,A,k}\lll \T^{m-1}p\rrr\bigr]\;\; \in\;\; \Z.
$$
We are most interested in the special case that $k = 0$ and $c_1(A) = m-1$.
As in \S\ref{ss:GW}, we define the invariants in this case as follows:
\begin{align}\label{eq:k=0}
 N_{M,A} \lll \T^{m-1}p\rrr: =  \frac 1{\ga(A)} N_{M,A,1} \lll \T^{m-1}p,\ga\rrr\;\in\; \Z
\end{align}
where $\ga\in H^2(M,\Z)$ is such that $\ga(A)\ne 0$; see the dicussion after
\eqref{eq:k=00}.

 \sss
 
Given distinct points $p_1,...,p_r \in M$ and corresponding local divisors $D_1,...,D_r$, this construction also straightforwardly extends to define invariants 
\begin{align}\label{eq:Tr}
N_{M,A}\lll \T^{m_1-1}p_1,...,\T^{m_r-1}p_r\rrr \in \Z
\end{align}
and so on.
In a slightly different direction, we can generalize the constraint $\lll \T^{m-1}p\rrr$ as follows. 
As above, let $J$ be an $\omega$-tame almost complex structure on $M^{2n}$ which is integrable near $p \in M$.
Let $D_1,...,D_n$ be smooth holomorphic local divisors near $p$ which are in general position in the sense that their normal bundles span the tangent space $T_pM$.
  (This condition implies that we can choose local holomorphic coordinates $(z_1,\dots,z_n)$ near $p$ such that each $D_i$ is given by $z_i = const$.)
For $c_1,...,c_n \geq 1$, we denote by $$
\calM^{J,\simp}_{M,A}\lll \T_{D_1}^{c_1-1}...\T_{D_n}^{c_n-1}p\rrr
$$
 the moduli space of simple $J$-holomorphic spheres $u: S^2 \rightarrow M$ in the homology class $A$ with a marked point $z$ such that $u$ is tangent to order $c_i-1$ to $(D_i,p)$ at the marked point $z$ for $i = 1,...,n$.
Note that in the case $c_1 = m$, $D_1 = D$, and $c_2 = ... = c_n = 1$, this reduces to the constraint $\lll \T^{m-1}p\rrr$ from earlier. On the other hand, in the case $n = 2$ and $c_1 = c_2 = 2$, the constraint 
$\lll \T^1_{D_1}\T^1_{D_2}p\rrr$ 
is equivalent to having a cusp at $p$, i.e. $u(z) = p$ and $du$ vanishes at the marked point $z$. More generally, for any $n$ the constraint $\lll \T_{D_1}^{k-1}...\T_{D_n}^{k-1}p\rrr$ corresponds to having $u(z) = p$ and the first $k-1$ derivatives of $u$ 
 at $z$ identically equal to zero.

The moduli space $\calM^{J,\simp}_{M,A,k}\lll \T_{D_1}^{c_1-1}...\T_{D_n}^{c_n-1}p\rrr$ is defined in the same way, with an additional $k$ unconstrained marked points.
Denote by $\Jj_{D_1,...,D_n}$  the set of $\omega$-tame almost complex structures on $M$ which are integrable on $\Op(p)$ and for which the local divisors $D_1,...,D_n$ are holomorphic submanifolds.
In this setting, 
we have:

\begin{prop}\label{prop:tanpseud2}
For generic $J \in \Jj_{D_1,...,D_n}$, the evaluation map 
$$
\ev: \calM_{M,A,k}^{J,\simp}\lll \T_{D_1}^{m_1-1}...\T_{D_n}^{m_n-1}p\rrr \rightarrow M^{\times k}
$$ 
defines a pseudocycle of dimension $2c_1(A) - 4 - 2\sum_{i=1}^n (m_i-1) + 2k$ which is independent of $J,p,D_1,...,D_n$ up to pseudocycle bordism (provided that the divisors $D_1,...,D_n$ are in general position).
\end{prop}

\begin{proof}[Sketch of proof]  This follows by essentially the same argument used to prove
 Proposition~\ref{prop:tanpseud}.  Although \cite{CM1} only considers the case of one divisor per marked point, we may choose local coordinates $z_1,\dots,z_n$ at $p$ so that the divisors $D_1,\dots,D_n$ coincide with the coordinate planes  $z_1=0,\dots,z_n=0$.  Hence
 the derivatives in different directions are independent of each other,  and the lower bound in \eqref{eq:ord} holds independently in each of the $n$ directions --- see Lemma~\ref{lem:cusp_deg} for a more precise statement.
The proof of Proposition~\ref{prop:tanpseud} then applies without essential change.  We leave further details to the reader.\end{proof}

\noindent Using this,  we can define invariants of the form $N_{M,A}\lll \T_{D_1}^{c_1-1}...\T_{D_n}^{c_n-1}p\rrr \in \Z$ and so on.

 \begin{rmk}\rm \label{rmk:tanpseud2}  (i) We have stated Proposition~\ref{prop:tanpseud2} for the sake of completeness, making no direct use of it in our arguments below.    However, it does 
 imply that for generic $J$ the curves counted by $N_{M,A}\lll \T^{m-1}p\rrr$ do not have vanishing derivative at 
   the marked point $z$ satisfying the constraint $\lll\T^{m-1}p\rrr$. 
To see this, note that such a curve would satisfy the stronger constraint $\lll\T_D^{m-1}\T_{D_2}\dots\T_{D_n} p\rrr$, which increases the codimension by at least two. 
For more general curve counts, by adding an additional marked point and imposing the vanishing derivative condition at that point, one can show that 
cusps only occur in codimension at least two.
This basic observation will be an essential ingredient in the proof that the multibranched tangency constraints considered in the next subsection are well defined.  For a precise statement see 
Lemma~\ref{lem:cusp_deg} below.

\MS

\NI (ii) In principle, we could also use Proposition~\ref{prop:tanpseud2} to define other curve counts involving cusps, or by considering multidirectional vanishing conditions for derivatives in higher dimensions. However, we do not pursue this here.
 \hfill$\er$
\end{rmk}

\subsection{Curves in dimension four with multibranched tangency constraints}\label{ss:tan_multi}

We now generalize the previous subsection by defining curve counts with constraints involving multiple branches through the same point in the target. For example, the simplest case will be denoted by $\lll p,p\rrr$, which corresponds to a curve having a double point at $p \in M$. 
More generally, given a local divisor $D$ near $p$, we will define invariants of the form $$
N_{M,A}\lll \T^{m_1-1}p,...,\T^{m_b-1}p\rrr \in \Z,\quad m_1,...,m_b \geq 1.
$$
We will assume without loss of generality that we have $m_1 \geq ... \geq m_b$, and we view $\Pp := (m_1,...,m_b)$ as a partition of $|\Pp| := \sum_{i=1}^b m_i$. 
For brevity we will often use the shorthand notation
\begin{align*}
\lll \T^\Pp p \rrr := \lll \T^{m_1-1}p,...,\T^{m_b-1}p\rrr.
\end{align*}
We will also occasionally  write $\T_D^{m-1}p$ instead of $\T^{m-1}p$ if we wish to make the choice of local divisor more explicit.

Using a similar setup to the previous two subsections, let $(M^{4},\omega)$ be a 
four-dimensional symplectic manifold, let $D$ be a local divisor near a point $p \in M$, let $A \in H_2(M;\Z)$ be a homology class, and pick an almost complex structure $J \in \Jj_D$.
For $m_1 \geq ... \geq m_b \geq 1$, we define $$
\calM^{J,\simp}_{M,A,k}\lll \T^{m_1-1}p,...,\T^{m_b-1}p\rrr
$$
 to be the moduli space of simple $J$-holomorphic spheres $u: S^2 \rightarrow M$
  (modulo reparametrizations) 
  in the homology class $A$ with 
\begin{itemize}
\item $k$ marked points $z_1,...,z_k \in S^2$
\item an additional $b$ marked points $z_{k+1},...,z_{k+b} \in S^2$ such
that $u$ is tangent to order $m_i-1$ to $(D,p)$ at $z_{k+i}$ for $i = 1,...,b$.
\end{itemize}
Note that as usual the marked points are all distinct and ordered, 
while the entries in the partition $\Pp$ are ordered but may not be distinct.  Therefore, if any of the multiplicities $m_i$ are repeated, the corresponding branches of $u$ are ordered. 
The main result of this subsection is:

\begin{prop}\label{prop:Ppseudo}  Let $(M,\om)$ have dimension four.
For generic $J \in \Jj_D$, the evaluation map $$
\ev: \calM^{J,\simp}_{M,A,k}\lll \T^\Pp p\rrr \rightarrow M^{\times k}
$$ defines a pseudocycle of dimension
\begin{align}\label{eq:indA1}
 2c_1(A) - 2 - 2\sum_{i=1}^b m_i + 2k.
\end{align}
 which is independent of $J,p,D$ up to pseudocycle bordism. Further, if
 $ 2c_1(A) - 2 - \sum_{i=1}^b m_i =0$ the (signed) count $\wh{N}_{M,A}\lll\Pp\rrr$ of curves 
 in the moduli space  $\calM^{J,\simp}_{M,A,0}\lll \T^\Pp p\rrr$  is  well defined. 
 \end{prop}
\NI Note that  the above index depends on $|\Pp|$ but not on the specific partition $\Pp$. By contrast, for $n >2$ the index depends on the number of parts $b$ in the partition, which causes a key index inequality to fail: see Remark~\ref{rmk:multcover} below.  
Further, we will see in Lemma~\ref{lem:blow} that each element of 
$\calM^{J,\simp}_{M,A,0}\lll \T^\Pp p\rrr$ in fact counts positively.

In order to prove Proposition~\ref{prop:Ppseudo}, we define the compactification $\ovl{\calM}^J_{M,A,k}\lll \T^\Pp p\rrr$ by noting that there is a natural forgetful inclusion of $\calM^{J,\simp}_{M,A,k}\lll \T^\Pp p\rrr$ into $\ovl{\calM}^J_{M,A,k+b}$, and we take the closure in this ambient compact space.
In order to show that we have a pseudocycle, it suffices to show that the image under the evaluation map of all of the strata added by this compactification have codimension two or greater. Compared to Proposition \ref{prop:tanpseud}, we have the new complication that a constraint involving two or more branches could degenerate into a constraint involving  a single branch. 
Therefore we need to prove that this type of degeneration can only occur in codimension two or greater.   
For example, given a double point singularity $\lll p,p\rrr$, in principle the two marked points in the domain could collide, leaving a curve with only one branch through $p$. 
To deal with this, the key observation is that the limiting curve necessarily 
has a cusp at $p$.  From the discussion at the end of the last subsection, we know that cusps appear with high codimension for generic $J$. 

For example,  suppose that in the situation considered above the corresponding two marked points, say $z_{k+1}, z_{k+2}$ lie together on a ghost component $u_0$.
Suppose that $u_1$ is an adjacent nonconstant component, and let $\wt{z}_1$ denote the special point of $u_1$ which participates in this adjacency.
Then we may apply \eqref{eq:ord} at $\Tilde z_1$ in each of the $n$ directions\footnote
{Since this part of the argument works in all dimensions, we will work temporarily in dimension $2n$.} to conclude that 
$$
\ord(u,D_i; \Tilde {z}_1)\ge 2,\quad i=1,\dots,n.
$$
Now observe that the condition of having two marked points both mapping to $p$ is codimension $2(2n-2)$, 
whereas having a single marked point satisfying $\lll \T_{D_1}\dots \T_{D_n}p\rrr$ (i.e. a cusp at $p$) is codimension $2n-2 + 2n$, which is two greater.
The proof of Proposition~\ref{prop:Ppseudo} will generalize this idea to the case of more complicated  cusp-type degenerations.
The following lemma, which follows directly from \cite[Lemma 7.2]{CM1}, describes more precisely the singularities arising as degenerations of $\lll \T^\Pp p\rrr$ constraints.

\begin{lemma}\label{lem:cusp_deg}
Let $D_1=D,D_2,\dots,D_n$ be local divisors at $p \in M$, and pick an almost complex structure $J \in \Jj_{D_1,\dots,D_n}$.
Given a curve $u \in \ovl{\calM}_{M,A,k}^{J}\lll \T_{D}^{m_1-1}p,\dots,\T_D^{m_b-1}p\rrr$, suppose that at least one of its constrained marked points lies on a ghost component $u_0$ of $u$.
Let $u_1,...,u_a$ denote the nonconstant components of $u$ which are adjacent to $u_0$, or more generally are adjacent to some ghost component of $u$ which is connected to $u_0$ through ghost components, and suppose that $\{z_{k+i}: i\in \Ii\}$ is the collection of constrained marked points\footnote{Recall that our ordering is such that the marked points $z_1,...,z_k$ are unconstrained, while 
the map $u$ satisfies a tangency condition to $(D,p)$ at each of the marked points $z_{k+1},...,z_{k+b}$.}
that lie somewhere on this collection of ghost components.
For $j=1,...,a$, let $\wt{z}_j \in u_j$ denote the relevant special point of $u_j$ which participates in the node realizing this adjacency.
Then we have 
\begin{align*}
&\sum_{j=1}^a \ord(u,D;\wt{z}_j) \geq \sum_{i\in \Ii} m_i,\\
&\sum_{j=1}^a \ord(u,D_r;\wt{z}_j) \geq |\Ii|,\quad  r=2,...,n.
\end{align*}
 \end{lemma}

We now return to dimension $4$.    
By combining the previous lemma with our previous arguments, we obtain the following.

\begin{lemma} \label{lem:cusp_deg1} 
If $\dim M = 4$,   $J \in \Jj_D$ is generic and $\Pp = (m_1,\dots,m_b)$ is a partition, the 
 space $\calM^{J,\simp}_{M,A}\lll \T^{\Pp}p\rrr$ is a smooth, oriented  (but not necessarily compact) manifold of dimension 
\begin{align}\label{eq:indA2}
d: = 2c_1(A) - 2 -2 \sum_{i=1}^b m_i.
\end{align}
If $J_0, J_1\in \Jj_D$ are two such generic elements, they may be joined by a generic $1$-parameter family $J_t$ such that the manifold 
$\bigcup_{t\in [0,1]}\calM^{J_t,\simp}_{M,A}\lll \T^{\Pp}p\rrr$ provides an oriented  (but not necessarily compact) cobordism between the manifolds at $t=0,1$.   Further, if $d= 0$ we may assume that all the curves in 
$\bigcup_{t\in [0,1]}\calM^{J_t,\simp}_{M,A}\lll \T^{\Pp}p\rrr$ are immersed.
\end{lemma}
\begin{proof}  
As in the case of   Lemma~\ref{lem:cusps}, the first claim is a slight extension of results from 
\cite{CM1,CM2}.  Lemma~6.5 in \cite{CM1} establishes an analogous result for the moduli space of somewhere injective, rational curves $u$ with $b$ marked points $z_1,\dots,z_b$ such that $u(z_i)$ is tangent 
to order $m_i$ to a given $J$-holomorphic submanifold $Z_i$. This argument adapts readily to our case in which 
the manifolds $Z_i$ are all the same and we also  
 require that $u(z_1)=\dots= u(z_b) = p \in Z_1$.
Since each degeneration of  the constraint $\lll \T^{\Pp}p\rrr$ has codimension at least $2$
by Lemma~\ref{lem:cusp_deg}, such degenerations do not occur along a generic $1$-parameter family of $J$.
Hence,  the standard techniques explained for example in \cite[Ch.3.2]{JHOL} can be adapted to establish the second claim.

Finally, to see that generically the elements in a moduli space of dimension $0$ or $1$ are immersed, notice first that
by \cite[Prop.E.2.4]{JHOL} every $J$-holomorphic curve $u$ can be perturbed to a nearby $J'$-holomorphic immersed curve $(u',J')$ that coincides with the original $(u,J)$ except near the singularity. It then follows from the arguments explained in Remark~\ref{rmk:tanpseud2}~(i) that cusps occur  in codimension at least two.
  \end{proof}

\begin{proof}[Proof of Proposition~\ref{prop:Ppseudo}]  
In view of Lemma~\ref{lem:cusp_deg1},  it remains to
prove that, when the moduli space is compactified,  the image under the evaluation map of the added points 
$$
\ovl{\calM}^J_{M,A,k}\lll \T^\Pp p \rrr \setminus \calM^{J,\simp}_{M,A,k}\lll \T^\Pp p \rrr
$$
consists of a union of strata of codimension at least two. 
  
Firstly, consider the case of a multiply covered curve with smooth domain, i.e. $u \in \calM^J_{M,A,k}\lll \T^\Pp p\rrr \setminus \calM^{J,\simp}_{M,A,k}\lll \T^\Pp p\rrr$.
Although $u$ is not necessarily regular, the underlying simple curve $\ovl{u}$ is regular for generic $J$. 
Assuming that $u$ is a $\kappa$-fold cover of $\ovl{u}$, it follows that $\ovl{u}$
 satisfies the constraint $\lll \T^{\ovl{m}_1-1}p,...,\T^{\ovl{m}_{\ob}-1}p\rrr$ with 
 $\ka \sum_{j=1}^{\ov b}\ovl{m}_j \ge  \sum_{i=1}^b {m}_i$. 
As in the proof of Proposition~\ref{prop:tanpseud} we then have
\begin{align}\notag
\ind(\ovl{u}) &= 2c_1(A)/\kappa - 2 - 2\sum_{j=1}^{\ob}\ovl{m}_j + 2k\\ \notag
&\leq 2c_1(A)/\kappa - 2 - 2\sum_{i=1}^b m_i/\kappa + 2k\\ \notag
&\leq 2c_1(A) - 2 - 2\sum_{i=1}^b m_i - 2(\ka - 1) + 2k \\ \label{eq:ovum2}
&\leq \ind(u) - 2, 
\end{align}
where  the penultimate inequality  uses the fact that 
the curve $\ovl{u}$ (after forgetting the $k$ unconstrained marked points) has nonnegative index,
and the last one uses $\ka>1$.

Now consider a nodal curve $u \in \ovl{\calM}^J_{M,A,k}\lll \T^\Pp p \rrr \setminus \calM^{J}_{M,A,k}\lll \T^\Pp p \rrr$. 
Note that because ghost components satisfy tangency constraints of arbitrarily large order
they may well appear
 with higher than expected dimension. Therefore, in order to show that $u$ lies in a stratum of codimension at least two, we will need to trade the constraints $\lll \T^{m_1'-1}p,\dots,\T^{m'_{b'}-1}p\rrr$ for constraints on the nearby nonconstant components of $u$.
In more detail, similar to the proof of Proposition~\ref{prop:tanpseud}, let
$\ovl{u}$ denote the (nodal, possibly disconnected) curve obtains as follows:
\begin{enumerate}
\item remove each ghost component of $u$ with a constraint of the form $$
\lll \T^{m_1'-1}p,\dots,\T^{m'_{b'}-1}p\rrr\quad\mbox{ where }\;\;  b' \geq 1,\;  m_1',...,m_{b'}' \geq 1,
$$ 
as well as all other ghost components that are connected to it through ghost components
\item add a new marked point at each of the nearby special points on a nonconstant component
\item replace each multiply covered component by its underlying simple curve, 
discard extra 
components in the same branch
when two or more have the same image, and finally parametrize each component of the resulting nodal curve by a new tree.
\end{enumerate}
Then $\ovl{u}$ has the same image under the evaluation map as $u$.
Note that each of the removed ghost components belongs to some maximal tree of ghost components, and we denote these by $T_1,...,T_k$.
For each $T_i$, let $\wt{z}_1^i,...,\wt{z}_{a_i}^i$ denote the corresponding special points of the nonconstant components of $u$ which are adjacent to $T_i$.
By Lemma~\ref{lem:cusp_deg}, each new marked point $\wt{z}^i_j$ satisfies a corresponding constraint 
$$
\lll \T_D^{\ord(u,D;\wt{z}^i_j)-1}\T_{D_2}^{\ord(u,D_2;\wt{z}^i_j)-1} p\rrr.
$$
By using Lemmas~\ref{lem:cusp_deg} and~\ref{lem:cusp_deg1}, and arguing as in the proof of Proposition~\ref{prop:tanpseud}, it is now easy to check that the moduli space of curves of the same combinatorial type as $\ovl{u}$ and satisfying the same tangency constraints 
has dimension at most $\dim \calM_{M,A,k}^{J,\simp}\lll \T^\Pp\rrr - 2$.
It follows as before that $\ev_k, k\ge 1,$ defines a pseudocycle whose bordism class is independent of choices. 
\MS

Finally in the case $k=0$ we may define the count $\wh{N}_{M,A}\lll \T^\Pp p \rrr$ of curves in 
$\calM_{M,A,0}^{J,\simp}\lll \T^\Pp p\rrr$ by using the divisor axiom as explained above in the discussion following  \eqref{eq:k=0}. 
\end{proof}

In the case $k=0$ it is also useful to consider the count 
$N_{M,A}\lll \T^\Pp p \rrr $ of the number of curves satisfying the multibranched constraint $\lll \T^\Pp p\rrr$ where the branches that have the same order of tangency to $D$ are unordered. 
Thus we have
\begin{align}\label{eq:noHatN}
N_{M,A}\lll \T^\Pp p \rrr \in \Z,\qquad N_{M,A}: = \frac1{|\Aut(\Pp)|} \wh N_{M,A},
\end{align}
where $\Aut(\Pp)$ is defined to be the {\bf automorphism group} of $\Pp$, that is, the group of permutations of the entries of $\Pp= (m_1,\dots,m_b)$  that preserve the inequalities $m_1\ge m_2\ge\dots \ge m_b>0$: see Definition~\ref{def:AutP}.  
We will sometimes also denote these invariants by simply $N_{M,A}^T \lll \Pp\rrr, \wh N_{M,A}^T \lll \Pp\rrr$ if the point $p$ is clear from the context or immaterial.

Given local divisors $D_1,\dots,D_r$ at distinct points $p_1,\dots,p_r$  for $i = 1,...,r$, and partitions
$\Pp_i = (m^i_1,...,m^i_{b_i})$, we can straightforwardly extend the above construction to define invariants 
\begin{align}
\wh N_{M,A}\lll \T^{\Pp_1}p_1,\dots,\T^{\Pp_r}p_r\rrr \in \Z
\end{align}
for four dimensional $M$,  which are independent of all choices. 
As we will see in \S\ref{sec:recur}, the recursive algorithm of Theorem~\ref{thm:main_thm}
will  involve invariants of this form even if we are only interested in counting curves with constraints at a single point in $M$.

\sss

\begin{rmk}\label{rmk:multcover}\rm
We restrict Proposition~\ref{prop:Ppseudo} to dimension four  because our method
of showing that the boundary  strata $ \calM^J_{M,A,k}\lll \T^\Pp p\rrr \setminus \calM^{J,\simp}_{M,A,k}\lll \T^\Pp p\rrr$ do not interfere with our count is so primitive.  All parts of the above proof work in higher dimensions except the crucial  
 index calculation~\eqref{eq:ovum2}  showing that the index of a multiple cover is strictly larger than the index of the underlying curve.  The analog of \eqref{eq:indA1} in dimension $2n\ge 6$ is 
$$
\ind (u) = 2n-6 + 2c_1(A) - 2\sum_{i=1}^b (n+m_i-2) + 2k,
$$ 
and it is now possible to have
a multiple cover $u = \ka\ov u$ where $\ind(u) = 0$ and $\ind(\ov u)\ge 0$.  For example, if $n=3$ and
 all $m_i=1$, then
 $\ka = b$ so that $\ov b = 1$, so that if we assume   $k=0$ and $\ind (u)=0$, then  we would also have
$\ind (\ov u) = 0$. 
The difficulty here is that we have no a priori control over the ramification of the cover $\Pp\to \ov \Pp$.
As pointed out in Remark~\ref{rmk:multcover0} a similar problem with multiple covers occurs with standard  point constraints
in the case $2n=6$ and $c_1(A) = 0$.  In this case, the problem may be resolved by using domain dependent perturbations but we do not pursue this here.
\hfill$\er$
\end{rmk}

The next step is to extend the  theory of automatic transversality to  curves that satisfy multi-branched constraints.
One illuminating approach is to transform the tangency constraints by a suitable blow up process into 
a normal crossing divisor, so that  the tangency constraint can be reinterpreted as a constraint relative to this divisor.
Then the automatic transversality statement is immediate. 
  This point of view also helps us understand the structure of curves that satisfy certain constraints: see Remark~\ref{rmk:dualpart}.
 However,  although  it certainly is relevant to the combination result Theorem~\ref{thm:ppt_intro}, it does not seem an easy way to 
 approach this question; see Remark~\ref{rmk:relGW}~(ii).

To explain this new approach,  consider the  symplectic manifolds $M^{(j)}, j\ge 1,$ defined as follows.
Let $M^{(1)}$ denote the blowup of $M$ at the point $p$, with resulting exceptional divisor $\E_1$, and let $D^{(1)}$ denote the proper transform of the local divisor $D$.
Inductively, let $M^{(j)}$ denote the blowup of $M^{(j-1)}$ at the point $D^{(j-1)} \cap \E_{j-1}$, with resulting exceptional divisor $\E_{j}$, and let $D^{(j)}$ denote the proper transform of the local divisor $D^{(j-1)}$.
Note that $M^{(j)}$ is  symplectomorphic to the $j$-fold blowup $\bl^jM$, and as such its homology naturally splits as $H_2(M^{(j)};\Z) = H_2(M;\Z) \oplus \Z\langle [\E_1],\dots,[\E_j]\rangle$.
By slight abuse of notation, we denote by $\E_1 - \E_2,\E_2-\E_3,\dots, \E_{j-1} - \E_j$ the proper transforms of $\E_1,\dots \E_{j-1}$ respectively in $M^{(j)}$
(the notation is justified by the fact that these lie in the homology classes $[\E_1]-[\E_2],\dots,[\E_{j-1}] - [\E_{j}]$ respectively).
Because $J$ is integrable in $\Op(p)$  there is a natural associated almost complex structure  $J^{(j)}$ on $ M^{(j)}$ which is integrable near the normal crossing divisor 
$$
N^{(j)}: = (\E_1 - \E_2)\cup (\E_2-\E_3)\cup \dots \cup (\E_{j-1} - \E_j)\cup \E_j.
$$
Further, by using the methods explained for example in \cite[Ch.7.1]{INTRO} it is not hard to see that one can put a symplectic form $\Tom^{(j)}$ on $M^{(j)}$ that tames $J^{(j)}$ and  that agrees with the pullback of $\om$ outside a small neighbourhood of $N^{(j)}$.
Note that the components $\E_{i-1} - \E_i$ of $N^{(j)}$ are not regular as $J^{(j)}$-holomorphic curves since they have normal bundle of Chern class $-2$.

\begin{lemma}\label{lem:blow}  Given $\Pp = (m_1,\dots, m_b)$ with $a: = m_1\ge \dots \ge m_b$, define 
\begin{align}\label{eq:nj}
n_j = \#\{i\;|\; m_i \geq j\},\quad 1\le j \le a, \qquad \wt{A} := A - \sum_{j=1}^a n_j[\E_j] \in H_2(M^{(a)}).
\end{align}
Then, for each $J_t\in \Jj_D$ in a generic $1$-parameter family  there is a bijective correspondence between the curves counted by $N^T_{M,A}\lll \Pp \rrr$ 
and the set of $J_t^{(a)}$-holomorphic curves in $(M^{(a)}, N^{(a)})$ in the class $\wt{A}$.  
Further each such $J_t^{(a)}$-holomorphic curve is immersed, regular and positively oriented. 
In particular, $N^T_{M,A}\lll \Pp \rrr$ equals the number of such curves.
\end{lemma}
\begin{proof}  Let $C$ be a curve in $M$ that contributes to the count $N_{M,A}\lll \Pp \rrr$, and denote by $C^{(1)},\dots, C^{(a)}$ its proper transforms in $M^{(1)},\dots, M^{(a)}$.
  Since the path $t\mapsto J_t$ is assumed generic, 
we may assume by Lemma~\ref{lem:cusp_deg1}  that $C$ is immersed, and that it has $b$ branches through $p$ where the $i$th branch is tangent 
 with order exactly $m_i$ both to $D$ and to the $j$th branch for $j>i$.  Therefore, under these successive blowups, the $i$th branch first meets $\E^{(1)}$ at its intersection point with   $D^{(1)}$, then meets 
 $\Ee^{(2)}$ at its intersection point with   $D^{(2)}$, until finally it meets $\Ee^{(m_i)}$ at a point different from its intersection with $D^{(i)}$, so that it remains unchanged by subsequent blowups.  Therefore, after blowup this branch  contributes $-([\E_1] + \dots + [\E_{m_i}])$ to the class of $C^{(a)}$.  Moreover, the different branches of $C^{(a)}$ 
 meet each divisor  $\Ee^{(i)}$ in distinct points.   Summing over all branches, we see that $[C^{(a)}] = \TA$.   Thus each such curve $C$ 
 gives rise to a unique $J^{(a)}$-holomorphic immersed curve $C^{(a)}$  in $(M^{(a)}, N^{(a)})$ in class $\TA$ that meets each divisor  $\Ee^{(i)}$ in distinct points.
 Conversely,  each such $J^{(a)}$-holomorphic  curve in class $\TA$ blows down to a unique immersed $A$-curve in $M$ that satisfes the multibranched constraint $\lll \T^\Pp p\rrr$ with tangency orders precisely
 given by $\Pp$.  

Thus there is a bijective correspondence between the curves $C, \TC$.   
Because $\TC$ is an immersed curve in  a zero-dimensional moduli space, one can use standard methods both to prove that it is regular and to orient it.  As explained for example \cite[Rmk.3.2.5]{INTRO}), the orientation is determined by the spectral flow of a  deformation $D^t, t\in [0,1],$ of the linearized normal Cauchy--Riemann operator $D^0$ to a complex linear operator.  In the genus zero case it follows from 
the Riemann-Roch theorem (\cite[Thm.C.1.10]{JHOL}) that each $D^t$ has no kernel.  Hence the curve is regular 
(which means precisely that  $D^0$ is surjective), and because the spectrum of $D^t$ never crosses zero, its normal orientation does not change. Thus, the number of curves counted by 
$N^T_{M,A}\lll \Pp \rrr$  does not vary as $J_t$ varies in a generic $1$-parameter family.   Moreover, each curve counts positively, since complex linear objects are considered to be positively oriented.\footnote
{
This argument is a geometric explanation of why the (holomorphic) tangency constraints $\T_D^\Pp p$ imposed on $C$ always count positively.}
\end{proof}

\begin{rmk}\label{rmk:relGW} \rm (i)  One could interpret the number of $J^{(a)}$-holomorphic $\wt{A}$-curves in $(M^{(a)}, N^{(a)})$ 
as a relative Gromov--Witten invariant for the pair $(M^{(a)}, N^{(a)})$, in which we count the number of rational curves   in $M^{(a)}$ that are holomorphic for a generic almost complex structure on the pair $(M^{(a)}, N^{(a)})$ and have no component in $N^{(a)}$.   This invariant counts spheres $\TC$ in class $\TA \in H_2(M^{(a)})$ that meet each $N_i$ in $k_i$ unspecified points
where $c_1(\TA)= 1$ and $\TA\cdot N_i =k_i\ge 0$  for each component $N_i$ of $N^{(a)}$.  Since there is no easy reference for this theory even in dimension $4$, we prefer not to use it explicitly.
However, note that if the class $\TA$ is that of an exceptional sphere (i.e. if $\TA\cdot \TA = -1$ so that $\TC$ is embedded)  this count is always $1$ and, in particular, does not change  when we allow $\TJ$ to be an arbitrary tame almost complex structure on $M^{(a)}$ for which $N^{(a)}$ is no longer holomorphic.  
This explains the observation in Remark~\ref{rmk:dualpart}  about the counts for partitions $\Pp$ such that the dual class $A_\Pp$ is exceptional. \MS

\NI (ii)
It is also interesting to compare the curves counted by $N_{M,A}(\Pp)$ when $\Pp = (1,1)$ and $\Pp=(2)$.
Blowing up twice to obtain the divisor $\Nn^{(2)} = (\Ee_1-\Ee_2)\cup \Ee_2$, the first counts curves in  class $A - 2[\Ee_1]$  that intersect $\Ee_1-\Ee_2$ in two distinct points, while the second counts curves in class
$A-[\Ee_1]-[\Ee_2]$ that meet $\Ee_2$ at a point $q\in \Ee_2 - (\Ee_1-\Ee_2)$.   If $q$ happens to lie on $ (\Ee_1-\Ee_2)$ then the latter curve generates into the union of $(\Ee_1-\Ee_2)$ with a curve $C_2$ in class $A - 2[\Ee_1]$.  
This nodal curve has two parametrizations as a genus zero stable map since the image of the node can be either 
of the two intersections of $C_2$ with $(\Ee_1-\Ee_2)$.  
Thus the number of elements in the full moduli space of $J^{(2)}$-holomorphic maps in class $A-[\Ee_1]-[\Ee_2]$
is $\Nn^{(2)} + 2\Nn^{(1,1)}$, a number that does not change as we allow $J^{(2)}$ to become a generic
almost complex structure on $M^{(2)}$ with disjoint exceptional divisors  $\Ee_1, \Ee_2$.\footnote
{This can be proved by the kind of elementary gluing argument in \cite{Mcacs}.}
  This is a very special case of the combination formula in Theorem~\ref{thm:ppt_intro}.\hfill$\er$
 \end{rmk}

\begin{rmk}
\label{rmk:dual}\rm 
(i) Given a partition $\Pp = (m_1,\dots,m_b)$, we define its {\bf dual partition} $\Qq = (n_1,\dots,n_a)$ by $n_j = \#\{i\;:\; m_i \geq j\}$. 
In the above proof, for a curve $C$ satisfying a tangency constraint $\lll \T^\Pp p\rrr$, the curve $C^{(a)}$ in $M^{(a)}$ lies in the homology class $A - \sum_{j=1}^a n_j[\E_j]$ determined by $\Qq$. 
We note that in terms of the Young diagram representation of partitions discussed in \S\ref{ss:comb} below, 
$\Qq$ is simply the partition determined by the columns rather than the rows of the partition.
\MS

\NI (ii)
  In higher dimensions there should be an analogous interpretation of the (single-branched) invariant  $N_{M,A}\lll \T^{m-1} p\rrr$.  However, although the first blowup  is at the point $p$, the second blowup is along the copy of $\C P^{n-2}$ formed by the tangent lines to the local divisor $D$, with similar descriptions for the subsequent blowups.
  \hfill$\er$
\end{rmk}

Although for simplicity we stated Lemma~\ref{lem:blow} for a single constraint $\T^\Pp$ an analogous statement also holds for invariants such as 
$$
N^T_{M,A}\lll \Pp_1,\Pp_2,...,\Pp_r\rrr
$$
with constraints at different points.   In particular, all curves contribute positively to the invariant.
As an immediate corollary of this observation, we have the following:

\begin{cor}\label{cor:count_pos2}
For partitions $\Pp_1,\Pp_2,...,\Pp_r$ with $\Pp_1 = (1^{\times b})$, we have
$$ 
N_{M,A}^T\lll \Pp_1,\Pp_2,...,\Pp_r\rrr = N^T_{\bl^1M,A-b[\E]}\lll \Pp_2,...,\Pp_r\rrr.
$$
\end{cor}
\begin{proof}
If
 $\Pp_1 = (1^{\times b})$, the corresponding invariant
$N_{M,A}^T\lll (1^{\times b}), \Pp_2,\dots, \Pp_r\rrr$ does not involve any tangency conditions at $p_1$ but simply counts curves with $b$ branches through $p_1$ and other tangency constraints at $p_2,\dots, p_r$.  If we blow up once at $p_1$, then for generic $J$ there is a bijective correspondence between these curves and the curves in class $A - b[\Ee]$ that satisfy the constraints
$\T^{\Pp_2}p_2,..., \T^{\Pp_r}p_r$.  Indeed, each curve counted by  $N_{M,A}^T\lll (1^{\times b}), \Pp_2,\dots, \Pp_r\rrr$ is immersed and has $b$ branches through $p_1$.  Hence when we blow up the intersection of the proper transform with the exceptional divisor $\Ee$ is precisely $b$, so that this proper transform lies in the class $A - b[\Ee]$.
\end{proof}

The next corollary shows that in dimension four the invariants $N_{M,A}\lll \T^{m-1}p\rrr$ counting curves with a single tangency constraint of full codimension  are almost always nonzero. 
This fact plays an important role in the closed curve upper bounds for the capacities defined in \cite[\S6.2.3]{HSC},
and by the equivalence between tangential and skinny ellipsoid constraints proved in \S\ref{ss:same} gives a different approach
to the existence result proved in \cite{HKerrat,Mint} for degree $d$ curves  in $\C P^2\less \io(E(\eps, \eps x))$ with $x>3d-1$ and one negative end.\footnote
{
Notice that this ellipsoid is skinny for the given constraint.} 
The recursion algorithm of \S\ref{sec:recur}  will compute these numbers explicitly, although since the matrix $A_k^{-1}$ in Lemma~\ref{lem:invertible} has negative entries the nontriviality of $N_{M,A}\lll \T^{m-1}p\rrr$ is not manifest.

\begin{cor}\label{cor:count_pos}  Let $M$ be a symplectic four-manifold, and $A \in H_2(M;\Z)$ be a homology class
with $c_1(A)>0$ that can be represented by a symplectically immersed sphere with positive self-intersections.
Then $N_{M,A}\lll \T^{c_1(A)-2}p\rrr > 0$.  
\end{cor}
 \begin{proof} Let $C\subset M$ be a symplectically immersed sphere with positive self-intersections in class $A$.  After a slight perturbation, we may assume that no more than two branches of $C$ go through any point, and then choose  an $\om$-tame $J$ so that $C$ is $J$-holomorphic.  Choose a point $p\in C$ with only one branch through it, and perturb $J$ so that it is holomorphic near $p$.  Finally, choose a holomorphic local divisor $D$ through $p$ that is tangent to $C$ to order exactly $c_1(C)-2$.  Then the moduli space  
 $\calM^{J,\simp}_{M,A}\lll \T^{c_1(A)-2}p\rrr$ is nonempty.  Moreover, the argument sketched in the second paragraph of the proof of Lemma~\ref{lem:blow}  shows that for generic $J$ its elements are regular and positively oriented.  
 Hence $\# \calM^{J,\simp}_{M,A}\lll \T^{c_1(A)-2}p\rrr > 0$, as claimed.
 \end{proof}

 \begin{rmk}\label{rmk:count_pos}\rm  If an immersed curve $C$ represents $A$ as in the above corollary,
a similar argument applies 
 for any constraint $\lll \T^\Pp p\rrr $ that can be chosen so as to be satisfied by $C$.  For example, if $C$ has $b$ transverse branches through some point $p$ and 
 $\Pp = (m-b+1,\underbrace{1,\dots,1}_{b-1})$
 (where $m = c_1(A) - 1$)
then $$
N_{M,A}\lll \T^\Pp p\rrr > 0.
$$  Thus, since degree $d$ curves of genus zero in $\C P^2$ have one double point, we may conclude that
$
N_{\C P^2, 3[L]}\lll   \T^{(6)} p, p\rrr > 0.
$
\hfill $\er$
\end{rmk}

For a partition $\Pp = (m_1,\dots,m_b)$ with $m_1 \geq ... \geq m_b$, we introduce the notation
\begin{align}\label{eq:delta_def}
\delta(\Pp) := \sum_{i=1}^b (i-1)m_i.
\end{align}
As we show in the
 next lemma, in dimension four this number can be interpreted as the count of double points near $p$ that arise after a generic perturbation of a curve $C \in \calM_{M,A}^{J,\simp}\lll \Pp \rrr$, where we perturb $C$ holomorphically near $p$ and keep it symplectically embedded, but allow  $J$ to vary as needed away from $p$ so that the perturbed curve is still $J$-holomorphic.
After replacing tangency constraints with ellipsoidal ends by the procedure described in 
\S\ref{ss:same},  we will see in \eqref{eq:wnegPp} that this number is also related to the total resulting writhe at the negative ends.

\begin{lemma} \label{lem:dePp}   Consider $C \in \calM_{M,A}^{J,\simp}\lll \Pp \rrr$ where $M$ has dimension four and $J$ is generic.  Then $\de(\Pp)$  is the number of double points near $p$ of a generic perturbation of $C$.
\end{lemma}
\begin{proof}  We prove this by induction on $b$, where $\Pp: = (m_1,\dots, m_b)$.
When $b=1$, we may assume that $C$ has  a single immersed branch through $p$, and
 the claim holds because $\de(\Pp) = 0$.  Suppose, inductively, that it holds for $b-1$ and consider $\Pp$
 with $m_1\ge\dots\ge m_b>0$.  The curve $C$ has $b$ branches through $p$, say $B_1,\dots, B_b$, 
 and we can assume that
  $B_i$ has contact order 
 exactly
 $m_i$ with $D$.  Because $m_b\le m_i$ for all $i$, it follows that  $B_b$  has contact of order $m_b$ with   $B_i$ for all $i<b$.  Hence one can perturb $B_b$ (without moving the other branches) to a curve $C'$ whose branch $B_b'$ has $m_b$ intersection points near $p$ with  each of the other branches.  Further $C'$ satisfies the constraint  $\lll\T^{\Pp'}p\rrr$, where $\Pp' = (m_1,\dots, m_{b-1})$.    By the inductive hypothesis, a further perturbation of $C'$ yields $\de(\Pp')$ other double points near $p$.  Therefore we have
 $$
 (b-1)m_b + \de(\Pp') = \de(\Pp)
 $$
 double points in all. \end{proof}

\begin{rmk}\label{rmk:dePp} \rm One can use Lemma~\ref{lem:dePp}  to show that certain counts must be zero.  Here is a simple example.  
It follows from the adjunction formula (see for example \eqref{eq:adjuc}) that
the number  of double points $\de(A)$ of an immersed $J$-holomorphic (rational) $A$-curve is 
\begin{align}\label{eq:deA}
\de(A) = \tfrac 12(A^2 - c_1(A)) + 1.
\end{align}
On the other hand,  if $C$ is a $J$-holomorphic curve
that satisfies the constraint $\lll \T^{\Pp}p\rrr$ for some $\Pp = (m_1,\dots,m_b)\in \partitions_{c_1(A)-1}$, 
the above lemma shows that $C$ can be perturbed to have at least $\de(\Pp)$ double points.  
Thus $N^T_{M,A}\lll\Pp\rrr = 0$ if $\de(\Pp)> \de(A)$. For example, no degree $3$ curve in $\C P^2$ can satisfy the constraint $\lll \T^{(6,2)}p\rrr$. \hfill$\er$
\end{rmk}

\subsection{Comparison with other invariants}\label{subsec:new_subsec}

Let $M$ be a semipositive symplectic manifold and $A \in H_2(M)$ a homology class. In this subsection we compare the local tangency invariant $N_{M,A}\lll \T^{m-1}p\rrr$ to several closely related invariants appearing in the literature.

\subsubsection{Stationary descendants and modified descendants}

As observed by Tonkonog in \cite[\S2.2]{tonk} (see also \cite[Remark 5.5]{HSC}), the local tangency invariant $N_{M,A}\lll \T^{m-1}p\rrr$ resembles the $1$-point gravitational descendant Gromov--Witten invariant $\gw_{M,A}\lll \psi^{m-1}p\rrr$.
The latter is readily computable; for instance for $d \in \Z_{\geq 1}$ we have $$\gw_{\CP^2,d [L]}\lll \psi^{3d-2}p\rrr = \frac{1}{(d!)^{3}}.$$
 The basic observation is that $\psi$ classes are defined as first Chern classes of natural complex line bundles with fiber corresponding to the cotangent line at a given marked point, and there are natural sections whose vanishing loci involves curves with tangency constraints to a given local divisor $D$.
 However, this heuristic is complicated by the fact that these sections additionally vanish along certain boundary strata of the stable map compactification, leading to additional discrepancy terms. This same phenomenon occurs in the analogous case of Gromov--Witten invariants relative to a global divisor, in which case the discrepancy terms for descendants are quite complicated but well-understand (see e.g. \cite{Gath_rel}).
 
 It would be interesting to give a similar description of the discrepancy terms in the local divisor context.
 For example, in the case of $\CP^2$ one expects the count $T_d$ to agree with $(3d-2)!\,\gw_{\CP^2,d [L]}\lll \psi^{3d-2}p\rrr$, after subtracting off correction terms corresponding to certain nodal configurations for which the marked point lies on a ghost component.
 See Table~\ref{table:CP2_1} for an explicit comparison.
We note that there are some technical challenges in trying to make this comparison precise, since the descendants are not easily defined using the pseudocycle approach of this paper.

\begin{remark}\rm
In \cite{graber2002descendant}, the authors give a modified definition of $\psi$ classes which is more easily interpreted in terms of tangency constraints, and they use it to compute characteristic numbers of rational 
curves in homogeneous projective varieties. These modified $\psi$ classes are defined by 
pulling back  the usual $\psi$ classes under the forgetful maps which forget all but one marked point.
However, this modification is vacuous if there is only one marked point, which is the case most relevant to the invariant $N_{M,A}\lll \T^{m-1}p\rrr$.
\hfill$\er$
\end{remark}

\subsubsection{Gromov--Witten invariants relative to a global divisor}

Recall that the definition of $N_{M,A}\lll \T^{m-1}p\rrr$ involves a choice of a point $p \in M$ and local divisor $D \subset \Op(p)$. 
Since $D$ could be chosen to be a piece of a global divisor $\wt{D}$ in $M$, it is natural to ask whether $N_{M,A}\lll \T^{m-1}p\rrr$ agrees with the analogous relative Gromov--Witten invariant with respect to $\wt{D}$.
Namely, let $R_{M,A}\lll \T_{\wt{D}}^{m-1}p\rrr$ denotes the relative Gromov--Witten invariant counting genus zero curves in $M$ in homology class $A$ which intersect $\wt{D}$ at a specified point $p$ and with contact order $m$, and with all other intersections with $\wt{D}$ unspecified.
Note that there are several technical approaches to Gromov--Witten theory using both algebraic geometry or symplectic geometry, although in the special case of a line in $\CP^2$ these should all 
coincide with the genus zero enumerative invariants studied by Caporaso--Harris \cite{CaH},
after imposing additional point constraints away from the line to cut down the index to zero.

A key feature of relative Gromov--Witten theory is that one effectively forbids curves which are entirely contained in the divisor, whereas in the local divisor case there simply cannot be any (nonconstant) curves contained in $D$.
For instance, let $D$ be a piece of a line $\wt{D}$ passing through $p \in \CP^2$.
Then we have $R_{\CP^2,[L]}\lll \T_{\wt{D}}p\rrr = 0$, since the only line tangent to $\wt{D}$ is $\wt{D}$ itself.
By contrast, we have $N_{\CP^2,[L]}\lll \T p\rrr = 1$, with solution curve $\wt{D}$ (this choice of local divisor and almost complex structure is nongeneric, but the count will persist after a small perturbation).
Similarly, $R_{\CP^2,d[L]}\lll \T^{3d-2}_{\wt{D}}p\rrr$ vanishes for $d > 1$, since a degree $d$ plane curve has total contact order $d$ to any line, whereas $N_{\CP^2,d[L]}\lll \T^{3d-2}p\rrr = T_d > 0$.

On the other hand, if we take $\wt{D}$ to be a smooth curve in $\CP^2$ whose degree is sufficiently large relative to $d$, then there are no degree $d$ curves contained in $\wt{D}$. (In fact, in this example it suffices to take $\deg(\wt{D}) \geq 3$, since in that case $\wt{D}$ has positive genus and hence contains no spheres.)
It seems natural to conjecture that for any fixed $d$ and $\wt{D}$ a smooth divisor with $\deg(\wt{D}) \gg d$, we have $R_{\CP^2,d[L]}\lll \T_{\wt{D}}^{3d-2}p\rrr = N_{\CP^2,d[L]}\lll \T^{3d-2}p\rrr$, and in particular the former is independent of $\wt{D}$.
Verifying this conjecture and its analogues for other symplectic manifolds $M$ could open up further algebro-geometric computational techniques such as localization.

\begin{remark}\rm
Another naive approach is to blow up $M$ at a point and consider Gromov--Witten invariants relative to the resulting exceptional divisor $E$. However, in the example of $\CP^2$, there are no curves in the homology class $d[L] - k[E]$ for $k > d$, and hence most of these counts necessarily vanish.

In \S\ref{ss:tan_multi}, we describe a iterated blowup procedure in dimension four which does allow a reinterpretation of $N_{M,A}\lll \T^{m-1}p\rrr$ as a count of curves relative to a certain global divisor; see in particular Remark~\ref{rmk:relGW}.
However, since this iterated blowup is nongeneric and introduces (necessarily irregular) curves of negative index, we find it too unwieldy for our purposes.\hfill$\er$
\end{remark}

\section{Curves with negative ends on skinny ellipsoids}\label{sec:skin}

In this section we consider another enumerative problem, defined in terms of curves with negative ends on a skinny ellipsoid.
Namely, given a closed symplectic manifold $(M^{2n},\omega)$, we remove a small neighborhood symplectomorphic to a skinny ellipsoid $E^{2n}_\sk$ (defined more precisely below) and consider punctured pseudoholomorphic curves in the symplectic completion of $M \setminus E_\sk$ with negative asymptotic ends \`a la SFT.
We first show in \S\ref{ss:skinhigh} that, for $M$ semipositive, the counts with one negative end on a skinny ellipsoid give well-defined invariants which are independent of all choices involved. Subsequently, in \S\ref{ss:skin4D}, we restrict to dimension four and define more general counts involving multiple negative ends on a skinny ellipsoid, which we also prove are independent of all choices involved. 
There are now two reasons why in this case we restrict to dimension four: besides the problem caused by the index of multiple covers as in Remark~\ref{rmk:multcover}, there is also no analog of Lemma~\ref{lem:cusp_deg}.
Instead we rule out the complicated degenerations that might arise in a one-parameter family in the case of multiple negative ends by using
the relative adjunction formula and writhe estimates from embedded contact homology.
These tools are special to dimension four and are briefly reviewed in \S\ref{ss:skin4D}.

The reader should note the formal parallels between curve counts with one negative end on a skinny ellipsoid and a single tangency constraint, and between several negative ends on a skinny ellipsoid and a multibranched tangency constraint. 
Indeed, in \S\ref{sec:relationships} we will prove that these analogous counts are indeed equivalent 
 in dimension four. 
This means that a posteriori we can think of tangency constraints and skinny ellipsoidal constraints as being essentially interchangeable.
Even though tangency constraints are arguably more natural from an enumerative point of view, for various reasons it is fruitful to have both perspectives.
For one, as we explain in more detail in \S\ref{sec:relationships}, by working with skinny ellipsoidal constraints we can bypass some technical analytic questions about gluing curves satisfying tangency constraints and instead appeal to the obstruction bundle gluing framework of Hutchings--Taubes \cite{HuT}.
Also, on a more conceptual level, the ellipsoidal point of view places our enumerative invariants into a broader SFT framework and suggests various relationships between different types of curve constraints. This perspective plays a role in \S\ref{sec:relationships} and will be further developed in \cite{McDuffSiegel}.

\subsection{Curves with a single end on a skinny ellipsoid}\label{ss:skinhigh}

Let $(M^{2n},\omega)$ be a semipositive symplectic manifold.
We consider a small skinny ellipsoid $$
E(\eps s_1, \eps s_2,...,\eps s_n) \;\;\subset\;\; \C^n,
$$
 with $s_2,\dots,s_n \gg s_1$ and $\eps$ sufficiently small.
 Recall here that the ellipsoid with area factors $a_1,\dots,a_n \in \R_{> 0}$ is defined by
 $$E(a_1,\dots,a_n) := \{ (z_1,\dots,z_n) \in \C^n \;:\; \sum_{i=1}^n \pi|z_i|^2/a_i \leq 1\}.$$
 Although not essential, it will be convenient to assume that $s_1,\dots,s_n$ are rationally independent and we have $1 = s_1 \ll s_2 < \dots < s_n$.
In the sequel we will often denote this ellipsoid simply by $E_\sk^{2n}$ when the precise values of $\eps$ and the $s_i$  are immaterial.
We will denote by $\eta_1$ the simple Reeb orbit of $\bdy E_\sk$ of least action, and let $\eta_2,\eta_3,...$ denote its iterates.
When discussing $E_\sk$ we will typically ignore the other Reeb orbits of $\bdy E_\sk$, which all have action at least $\eps s_2$ and Conley--Zehnder index that is too large to be relevant; see Lemma~\ref{lem:skin}. We will accordingly refer to these as the ``long orbits'' of $\bdy E_\sk$, and we refer to the orbits $\eta_1,\eta_2,\eta_3,\dots$ as the ``short orbits''.

Let $\alpha$ denote the standard contact form on $\bdy E_\sk$ given by the restriction of the Liouville form $\sum_{i=1}^n \tfrac{1}{2}(x_idy_i - y_idx_i)$ on $\C^n$. 
After picking a trivialization of the symplectic vector bundle $(\ker \alpha,d\alpha)$ restricted to $\eta_1$, we can assign to each of the orbits $\eta_i$ a Conley--Zehnder index $\cz_\tau(\eta_i) \in \Z$. By default we will pick the (unique up to homotopy) trivialization $\tau_\ex$ which extends over a spanning disk  in $\p E_\sk$ 
for $\eta_1$ (in which case we will simply write $\cz = \cz_{\tau_\ex}$).
With this choice, we have\footnote{We warn the reader that in the case of ellipsoids there is another natural trivialization $\tau_\sp$ for which we have $\cz_\tau(\eta_m) = n-1$ for all $m < s$.  For more detail see \S\ref{ss:skin4D}.}
\begin{align}\label{eq:CZind}
\cz(\eta_m) = \cz_{\tau_\ex}(\eta_m) = n-1 + 2m, \quad \mbox{ if } \;\; s_2>m.
\end{align}
More generally, for any ellipsoid $E(a_1,\dots,a_n)$ with $1 \leq a_1 < \dots < a_n$ rationally independent, we have precisely $n$ simple Reeb orbits, which have actions $a_1,\dots,a_n$, and the Conley--Zehnder index of the $m$-fold iterate $\ga_{k,m}$  of the $k$th simple orbit is 
\begin{align}\label{eq:CZindn}
\cz(\ga_{k,m}) = \cz_{\tau_\ex}(\ga_{k,m})=  n-1 + 2\sum_{i=1}^n \left\lfloor \frac{ma_k}{a_i} \right\rfloor.
\end{align}
In particular, note that the Conley--Zehnder indices of all Reeb orbits of an ellipsoid have the same parity.

Now suppose we have a symplectic embedding $\io: E_\sk\hookrightarrow M$.
\begin{definition}\label{def:admiss}
The {\bf symplectic completion}
of $M \setminus \io(E_\sk)$ is obtained by gluing in the negative cylindrical end $((-\infty,0] \times \bdy E_\sk,d(e^r \alpha))$,
where $r$ is the coordinate on $(-\infty,0]$.
An almost complex structure $J$ on such a  completion
 is called {\bf admissible} if it is compatible with the symplectic form and if on some neighborhood of  the cylindrical end it is $r$-translation invariant, preserves the contact hyperplanes $\ker \alpha$, and sends $\frac{\p}{\p r}$ 
   to the Reeb vector field for $\alpha$. 
 \end{definition}
 \NI 
 Note that all punctured pseudoholomorphic curves will occur in symplectic completions, and therefore we will 
usually  suppress the completion process from the terminology when no confusion should arise.

Given a generic admissible $J$ and a homology class $A \in H_2(M,  \io(E_{\sk});\Z) \cong H_2(M;\Z)$,
let 
$$
\calM_{M \setminus \io(E_{\sk}),A}^{J,\simp}(\eta_{m}) 
$$
denote the moduli space of simple $J$-holomorphic $A$-planes in the symplectic completion of $M \setminus \io(E_{\sk})$ with one negative end asymptotic to $\eta_{m}$.
This is a smooth (but not necessarily compact) oriented manifold of
dimension
\begin{align}\label{eq:indell}
\ind\, \calM_{M \setminus \io(E_{\sk}),A}^{J,\simp}(\eta_m) 
= 2c_1(A) - 2 -2m.
\end{align}
Notice that this is the same as the dimension given in \eqref{eq:indA} for curves satisfying the single tangency constraint 
$\lll \T^{m-1}p\rrr$.
More generally, recall from \cite{BEHWZ} that the Fredholm index of an 
$A$-curve
 $C$ of genus $g$ with $k$ positive ends asymptotic to Reeb orbits $\gamma^+_1,\dots,\gamma^+_k$  and
$l$ negative ends asymptotic to Reeb orbits $\gamma^-_1,\dots,\gamma^-_\ell$ 
is given by
\begin{align}\label{eq:Find}
\ind(u) = (n-3)\chi(C) + 2c_\tau(A) + \sum_{i=1}^k \CZ_\tau(\gamma_i^+) - \sum_{i=1}^\ell \CZ_\tau(\gamma_j^-).
\end{align}
where $\chi(C) = 2-2g-k-\ell$ is the Euler characteristic of the domain of $C$.
Here $c_\tau(A)$ denotes the relative first Chern number of $A$ with respect to the trivialization $\tau$ along its ends, and one can check that the overall expression does not depend on the choice of $\tau$.  Moreover,  if we use the trivialization $\tau_{\ex}$ mentioned earlier, $c_\tau(A)$ is just the usual first Chern class, which explains \eqref{eq:indell}.\footnote
{For more detail on these formulas see~\eqref{eq:thetasp}.}

\begin{definition}\label{def:skin}    We say that an ellipsoid $E(\eps, \eps s_2,\dots, \eps s_n)$ is {\bf $A$-skinny} 
(or simply  {\bf skinny}) if $s_1 = 1,s_2,\dots,s_n$ are rationally independent and $c_1(A)-1<s_2 < \dots < s_n$.
\end{definition}

\begin{lemma}\label{lem:skin}   If $E: = E(\eps, \eps s_2,\dots, \eps s_n)$ is  $A$-skinny, then every (rational) curve $C$ in $M\less \io(E)$ 
in a class $A'$ with 
 $c_1(A') \le c_1(A)$ and at least one negative end on a long orbit has index $\le -2$.
\end{lemma}
\begin{proof}  
 By \eqref{eq:CZindn}, if $\gamma$ is a long orbit of $\bdy E_\sk$ we have
 $$
 \CZ(\ga) \ge  n-1 + 2(\lfloor s_2 \rfloor + 1).
 $$
 Therefore by \eqref{eq:Find} we have
\begin{align*}
\ind(C) &\le (n-3) + 2c_1(A) - (n+ 1+ 2\lfloor s_2 \rfloor) \\&= 2(c_1(A) - \lfloor s \rfloor -2) \le -2,
\end{align*}
as claimed.
\end{proof}  

\begin{rmk}\label{rmk:skin}\rm  By using the Morse--Bott form of the index calculation (as in \cite[Lemma~3.8]{HK}), one can prove the analog of Lemma~\ref{lem:skin} for any ellipsoid $E(\eps, \eps s,\dots, \eps s)$ with $s>c_1(A)-1$.
Hence we may also take $E_{\sk} = E(\eps, \eps s,\dots, \eps s)$.
\hfill$\er$
\end{rmk}

Now consider the case when  $m = c_1(A) -1$ so that 
$\calM_{M\setminus \io(E_\sk),A}^{J,s}(\eta_m)$
 is a discrete set of signed points. The following proposition, together with Proposition \ref{prop:s_indep} below, shows that the signed count of points in this set is finite and independent of all choices.
We will denote this count by $N^E_{M,A}\lll (m)\rrr \in \Z$.\footnote{As usual, all of our invariants are defined to be zero when they correspond to a count of curves of nonzero Fredholm index (after taking into account all relevant constraints).}
Here we are viewing $\Pp = (m)$ as a partition of length one (more general partitions will be considered in the next subsection).

\begin{prop}\label{prop:oneend}  Let $(M,\om)$ be a semipositive symplectic manifold of dimension $2n$.
For $m = c_1(A) -1$, the count of curves in $\calM_{M \setminus \io(E_{\sk}),A}^{J,\simp}(\eta_m)$ is finite and independent of the choice of 
generic $J$, 
the embedding $\io$
and the parameter $\eps$, provided that $s_2$ is sufficiently large.
\end{prop}
\NI Note that we could also formulate a more general statement involving additional marked points and pseudocycles as in \S\ref{sec:tangency}, but for simplicity we only consider the index $0$ case here.

\MS

In the following, we will make heavy use of the SFT compactness theorem (see \cite{BEHWZ} for more details).
In the case of punctured pseudoholomorphic curves in $M \setminus \io(E_\sk)$, the compactification adds pseudoholomorphic buildings  in a many leveled ambient space 
with a  main level  $M \setminus \io(E_\sk)$ plus one or more  symplectization levels $\mathbb{R} \times \bdy E_\sk$.   
For two consecutive levels, the negative Reeb orbit asymptotics of the part of the building  in the upper level are paired with the positive Reeb orbit asymptotics of that in the lower level.  
The reader should keep in mind that the curves in each level can potentially be disconnected and/or nodal (although for us typically the total topological type of the building will be connected and genus zero).
In addition to the usual stability condition appearing in the definition of the stable map compactification for closed curves, the buildings arising in the SFT compactification are forbidden from having any symplectization levels consisting entirely of trivial cylinders.\footnote{Recall that a {\em trivial cylinder} in a symplectization $\R \times Y$ is an index zero pseudoholomorphic cylinder of the form $\R \times \gamma$, where $\gamma$ is a Reeb orbit in $Y$.}
Terminologically, we will say that a genus zero curve is ``connected'' if the domain parametrizing it is connected but possibly nodal, ``smooth'' if its domain is without nodes, and ``irreducible'' if its domain is both connected and smooth.
By default, ``component'' of a curve means an irreducible component.

For the purpose of index arguments, it will also be convenient to adopt the terminology of {\bf matched components} used for example in \cite[Def~3.3.2]{Ghost}.
Namely, in the context of a pseudoholomorphic building, a matched component is the result after formally gluing together some collection of curve components lying in various levels along a collection of paired ends.
A matched component naturally has an overall domain which is a punctured Riemann surface and is required to be smooth and connected (although it does not have a well-defined conformal structure unless we further pick gluing parameters). 
We define the index of a matched component to be simply the sum of the (Fredholm) indices of each of the constituent curve components.  
The index of a matched component can computed as though it were an honest irreducible pseudoholomorphic curve in a single level with same overall topological type, homology class, and Reeb orbit asymptotics.   
The key point here is that each term in \eqref{eq:Find} is appropriately additive under such matching.
The same remark applies to the {\bf energy}\footnote{Here we using the same conventions for the energy of punctured curves as in \cite[\S3.1]{HSC}. In particular, the energy of a pseudoholomorphic curve in a symplectization only measures the variation in the contact slice direction, and it vanishes for branched covers of trivial cylinders.}
of a matched component.

By way of terminology, a matched component in a symplectization $\R \times Y$ is a matched component as above whose underlying building consists of with one or more levels, each lying in the symplectization $\R \times Y$. Similarly, a matched component in a symplectic cobordism $X$ is by definition a matched component whose underlying building consists of some number (possibly zero) of levels in the symplectization $\R \times \bdy^- X$, a ``main'' level in $X$, and some number (possibly zero) of levels in the symplectization $\R \times \bdy^+ X$.
By default, when we refer to the positive or negative ends of a matched component we mean only those unpaired ends not participating in the formal gluing.

\sss

In this paper we are only considering pseudoholomorphic curves which satisfy a homogeneous version of the Cauchy--Riemann equation, i.e. without the Hamiltonian perturbations or virtual perturbations typically encountered in SFT or Floer theory.
In this setting we can still apply the SFT compactness theorem to produce compactified moduli spaces of curves, but these will not typically be tranversely cut out for a generic almost complex structure, meaning that various boundary strata might appear with higher-than-expected dimension. In what follows, we will rule out such occurrences (at least in a semipositive context) by careful index arguments. 
Similar to \S\ref{sec:tangency}, the basic idea is to use the fact that simple curves are indeed regular for generic $J$, and
although multiply covered curves with negative index do potentially appear,
we can assume their underlying simple curves have nonnegative index. 

We will also utilize the neck stretching procedure from SFT to decompose symplectic cobordisms into simpler pieces (see \cite[\S3.4]{BEHWZ} for details).
Recall that, given symplectic cobordisms $X^+,X^-$ such that the positive contact boundary of $X^-$ and the negative contact boundary of $X^+$ are both identified with a fixed contact manifold $Y$, we can glue along this common contact boundary to construct the concatenated  
 symplectic cobordism $X^- \circledcirc X^+$ from the positive boundary of $X^+$ to the negative boundary of $X^-$.
\footnote{Technically, in order to say that the energy of curves decomposes as expected we should work with contact manifolds with fixed contact forms; see \cite[\S3.1]{HSC} for a more detailed treatment. Notice also that we use the notation $X^- \circledcirc X^+$ to denote the glued (single level) manifold with  negative end  $X^-$ and positive end $X^+$, while
$X^- \notccirc X^+$ denotes the two (or more) level  structure   in which the manifolds $X^-$ and $X^+$ form separate levels.
}
Conversely, given a symplectic cobordism $X$ containing a separating contact hypersurface $Y$, we can split along $Y$ to obtain two symplectic cobordisms $X^+,X^-$. 
Roughly, neck stretching along $Y$ proceeds by defining a one-parameter family of almost complex structures $J_t$ on $X$, $t \in [0,1)$, such that in the limit as $t \rightarrow 1$ curves are forced to degenerate into pseudoholomorphic buildings in the split symplectic cobordism. 
In a context where transversality holds, one expects to get a cobordism of moduli spaces relating curves in $X$ to split buildings in $X^- \notccirc X^+$, i.e. buildings with levels in $X^-$ and $X^+$ and also possibly some number of intermediate levels in the relevant symplectizations $\R \times \bdy^+X^+$, $\R \times Y$, $\R \times \bdy^-X^-$. 
In our setting without perturbations, neck stretching can often give rise to degenerations with dimension that is higher than expected. At any rate, given regular curves in $X^-$ and $X^+$ with paired Reeb orbit asymptotics, we can perform a standard gluing\footnote
{
See for example \cite[Thm.2.54]{Pardcnct} for a precise formulation of the needed regularity assumptions and for a proof.
}
 along cylindrical ends to produce $J_t$-holomorphic curves in $X$ for all $t$ sufficiently close to $1$. 
In \S\ref{ss:same} we will also consider more general obstruction bundle gluing problems which involve curves in $X^-$ and $X^+$ along with a branched cover of a trivial cylinder in $\R \times Y$.

\MS

\begin{lemma}\label{lem:index_bd}
Let $C$ be a connected rational curve in $M \setminus \io(E_\sk)$  that is holomorphic for a generic $J$
and has negative ends $\eta_{m_1},...,\eta_{m_b}$ for some $b,m_1,...,m_b \geq 1$.
Then we have $\ind(C) \geq (4-2n)(b-1)$. 
 In particular, $\ind(C) \ge 0$ if $b=1$.
\end{lemma}
 \begin{proof}
We
 suppose that $C$ is a $\kappa$-fold cover of its underlying simple curve $\ovl{C}$, which we can assume to have nonnegative index.
Suppose that $\ovl{C}$ has negative ends $\eta_{\ovl{m}_1},...,\eta_{\ovl{m}_{\ovl{b}}}$ for some $\ovl{b},\ovl{m}_1,...,\ovl{m}_{\ovl{b}} \geq 1$. Note that we have $\sum_{i=1}^b m_i = \kappa \sum_{i=1}^{\ovl{b}} \ovl{m}_i$.
Letting $A$ denote the homology class of $C$ (so that $A/\kappa$ is the homology class of $\ovl{C}$), we therefore have
\begin{align*}
\ind(C) &= (n-3)(2-b) + 2c_1(A) - \sum_{i=1}^b (n-1+2m_i)\\
&= (2n-6) + (4-2n)b + 2c_1(A) - 2\sum_{i=1}^b m_i\\
&= \kappa\, \ind(\ovl{C}) + (2n-6)(1-\kappa) + (4-2n)(b - \kappa \ovl{b})\\
&\geq 2n-6 + \kappa(6-2n-\ovl{b}(4-2n)) + (4-2n)b\quad \mbox{since }  \ind(\ovl{C})\ge 0\\
& = (\ka-1) (6-2n+ (2n-4)\ov b) -\ovl{b}(4-2n) + (4-2n)b\\ 
& \geq -\ovl{b}(4-2n) + (4-2n)b\\
& \geq (4-2n)(b-1),
\end{align*}
as claimed.
 \end{proof}

\begin{lemma}\label{lem:mc0}
Let $S$ be a matched component in $M \setminus \io(E_\sk)$
\footnote{That is, $S$ is given by formally gluing curves which live in the main level $M \setminus \io(E_\sk)$ and in some number of symplectization levels $\R \times \bdy E_\sk$. Since it has no punctures, it is closed and hence cannot lie entirely in $\R\times \p E_{sk}$.}
which is connected, genus zero, and without punctures,
 and that is holomorphic for a generic $J$.
 Then we have $\ind(S) \geq 2n-6$, and a fortiori $\ind(S) \geq 2n-2$ provided that $S$ does not consist of a single curve component in $M \setminus \io(E_\sk)$.
\end{lemma}
\begin{proof}
If $S$ is a single curve component in $M \setminus \io(E_\sk)$, then we have
$$
\ind(S) = 2n-6+2 c_1(A) \geq 2n-6
$$
since $J$ is generic and $(M,\om)$ is  semipositive.\footnote
{
The semipositivity condition in \cite[Def.6.4.1]{JHOL} states that if $A$ is a  spherical class with $n-3+ c_1(A)\ge 0$ then $c_1(A)\ge 0$. Note that classes with $n-3+ c_1(A)< 0$ are not represented for generic $J$.}
Otherwise, let $C_0$ be one of the curve components of $S$ that lie in  
$M \setminus \io(E_\sk)$ and suppose it has $b\ge 1$ negative ends of multiplicities $m_1,...,m_b$.  Then all the other components of $S$ may be grouped into $b$ matched components,  $D_1,...,D_b$, each of which is topologically a plane with exactly one positive end that is attached to $C_0$. We  calculate $\ind(D_i)$ by
\eqref{eq:Find} using the fact that the Conley--Zehnder indices of all the interior matched ends  of $D_i$ cancel. 
Note also that if  $A'$ is the class represented by any part of $D_i$ that lies in $M \setminus \io(E_\sk)$
then we must have $c_1(A')\ge 0$ by semipositivity since this part of $D_i$ can always be capped off to  a sphere.
Hence, by
Lemma~\ref{lem:index_bd}, we have
\begin{align*}
\ind(S) &= \ind(C_0) +{\textstyle  \sum_i} \ind(D_i)\\
&\geq (4-2n)(b-1) + {\textstyle (n-3)(\sum_{i=1}^b \chi(D_i))  + \sum_{i=1}^b(n-1+2m_i)}\\
&\geq (4-2n)(b-1)  + (n-3)b  + b(n+1)\quad \mbox{since } m_i\ge 1\\
&= 2b -4+2n \\
&\geq 2n-2 \quad \mbox{since } b\ge 1.
\end{align*}
This completes the proof.
\end{proof}

\begin{lemma}\label{lem:mc1}
Let $C$ be a matched component in $M \setminus \io(E_\sk)$
 which is  $J$-holomorphic for generic $J$ and has genus zero.  Assume further that
 it has exactly one negative end $\eta_m$, which in addition
 lies on a curve component in the main level. 
Then we have $\ind(C) \geq 0$,
and the inequality is strict  
whenever
the curve components of $C$ that lie in the main level 
have at least one end other than $\eta_m$.
\end{lemma}
\begin{proof}
Suppose that the negative end of $C$ is $\eta_m$ and lies on a curve component $C_0$ in $M \setminus \io(E_\sk)$. 
We suppose that $C_0$ has $b$ negative ends in all.  Notice that the ``skinny'' assumption implies that all of these ends must lie on a short orbit of $\p E_{\sk}$. 
Indeed, otherwise the underlying simple curve still has a negative end on a long orbit and hence has negative index by Lemma~\ref{lem:skin}, contradicting the fact that $J$ is generic. 
If $b=1$ then $\ind(C_0) = \ind(C) \geq 0$ by Lemma~\ref{lem:index_bd}.
 Otherwise,
 all the other ends of $C_0$ are attached to matched components, say $D_1,\dots, D_{b-1}$  
which, as in Lemma~\ref{lem:mc0}, must each be a plane and have positive ends of multiplicities $m_1,\dots,m_{b-1}$.
  Therefore,    
by  Lemma~\ref{lem:index_bd}, we have
\begin{align*}
\ind(C) &\geq (4-2n)(b-1) + (n-3)(b-1) + \sum_{i=1}^{b-1}(n-1 + 2m_i)\\
&\geq (4-2n)(b-1) + (n-3)(b-1) + (b-1)(n+1)\\
&= 2(b-1) \geq 0,
\end{align*}
and the inequality is strict provided that $b>1$.
\end{proof}

\begin{lemma}\label{lem:mc2}
Let $C$ be a matched component of genus zero in the symplectization
 $\mathbb{R} \times \bdy E_\sk$ having one negative end, with all of its ends asymptotic to short orbits in $\bdy E_\sk$.
Then  $\ind(C) \geq 0$, with equality if and only if each component of $C$ is a trivial cylinder.
\end{lemma}

\begin{proof}    
Suppose first that $C$  consists of a single curve component, with negative end $\eta_m$ and top ends  $\eta_{m_1},...,\eta_{m_a}$ for some $a,m_1,...,m_a \geq 1$. 
Using \eqref{eq:Find} one can check
$$
\ind(C) = 2(a-1 + \sum_{i=1}^a m_i - m),
$$
which we note is independent of $n$.
Furthermore, by nonnegativity of energy we have $\sum_{i=1}^a m_i \geq m$, and hence this index is strictly positive unless $a=1$ 
and $m_1= m$.
In this case, the energy of $C$ is zero, and this implies that $C$ is a (possibly branched) cover of a trivial cylinder.
In fact, 
since its domain has genus zero
$C$ is necessarily an unbranched cover, and hence it is the trivial cylinder over $\eta_m$.

If $C$ is a building with several levels,
the same argument applied to the whole building shows that it has nonnegative 
index.  Moreover, if it has zero index 
it must have a single positive end
of multiplicity  $m$.  In this case,  the top level of $C$ must consist of a single cylinder, because its unique negative end that is attached to 
the bottom level cannot have multiplicity  $<m$ (since the matched component that connects it to the lowest level has  nonnegative energy) or $> m$
(since the top level has nonnegative energy), and so must have multiplicity precisely $m$.  Working inductively down 
the levels of the building, we find that all levels must be cylinders, as claimed.   
\end{proof}

\begin{proof}[Proof of Proposition~\ref{prop:oneend}]
The setup is similar to the proof of Proposition \ref{prop:tanpseud},
except that since our curves are punctured we must replace Gromov's compactness theorem with the SFT compactness theorem \cite{BEHWZ}. 
To see that the count is finite 
when $J$ is generic,
we consider the compactification $\ovl{\calM}_{M\setminus \io(E_{\sk}),A}^J(\eta_k)$ provided by the SFT compactness theorem, and we need to show that  
this coincides with the uncompactified moduli space $\calM_{M\setminus \io(E_{\sk}),A}^{J,\simp}(\eta_k)$.
Let $C$ denote an element of $\ovl{\calM}_{M\setminus \io(E_{\sk}),A}^J(\eta_k)$.
Recall that in general $C$ could be a multilevel pseudoholomorphic building, with top level consisting of a punctured pseudoholomorphic curve (with domain possibly disconnected and/or nodal) in $M \setminus \io(E_{\sk})$, plus some number of levels in the symplectization $\R \times \bdy E_\sk$.
By Lemma~\ref{lem:skin} we may assume that the ends of every curve component lie on short orbits.

First, suppose that $C$ has a single level. That is, $C$ is a nodal curve in $M \setminus \io(E_\sk)$, with one component $C_0$ a plane in $M \setminus \io(E_\sk)$ with negative end $\eta_m$, and the remaining components spheres $S_1,...,S_r$ lying entirely in $M \setminus \io(E_\sk)$.
As in \S\ref{sec:tangency}, we can pass to the underlying simple configuration $\ovl{C}$, and it is easy to show using semipositivity that $\ind (\ovl{C}) \leq \ind(C)$.
Since each node increases the expected codimension by two, if $k \geq 1$ we must have $\ind(\ovl{C}) \leq -2$, so $\ovl{C}$ (and hence $C$) does not appear for generic $J$.  (Compare with the discussion concerning 
\eqref{eq:k=0}.)

Now suppose that $C$ is a building with two or more levels. 
Because the whole building has one (unpaired) negative end and  genus zero, we
 can formally glue the components of $C$ into matched components so that we have
\begin{itemize}
\item a  matched  component $C_0$ with all its constituent curves in $\R \times \bdy E_\sk$ 
and negative end $\eta_m$
 (as in Lemma~\ref{lem:mc2})
that has   $r \geq 1$ positive ends 
that are attached  to the negative boundary of the main level of $C$;
\item matched components $C_1,\dots,C_r$ in $M \setminus \io(E_\sk)$, each of which is topologically a plane with one negative end  on $C_0$ (as in Lemma~\ref{lem:mc1});
\item matched components $S_1,...,S_a$ in $\bigl(M \setminus \io(E_\sk)\bigr)$, each of which is an unpunctured sphere
(as in Lemma~\ref{lem:mc0}).
\end{itemize}

By \eqref{eq:Find}, and using the fact that the Euler characteristics of the curves  $C_i, i\ge 0,$ add to one, and that the index contributions from the top ends of $C_0$ are cancelled by those from the negative ends 
of the $C_i, i\ge 1,$, we have
\begin{align*}
&  \ind(C_0) + \sum_{i=1}^r \ind(C_i) + \sum_{j=1}^a \ind(S_j)  = (n-3)\chi(C_0) - 2m \\
&\qquad\qquad\qquad  + \sum_{i=1}^r \bigl((n-3)\chi(C_i) + 2c_1(A_i)\bigr) +  \sum_{j=1}^a \bigr(2(n-3) + 2c_1(A_j)\bigr)\\
& \qquad\qquad = (n-3) + 2c_1(A) -2m + 2a(n-3) = 2a(n-3)
\end{align*}
since the elements of 
$ \calM^{J,\simp}_{M \setminus \io(E_\sk),A}(\eta_m)$ have index zero.
On the other hand, according to the preceding lemmas we have
$$
\ind(C_0) + \sum_{i=1}^r \ind(C_i) + \sum_{j=1}^a \ind(S_j) \geq a(2n-6),
$$
with the inequality strict unless each $S_j$ is entirely contained in $M \setminus \io(E_\sk)$ and $C_0$ is entirely composed of trivial cylinders.
However, in this case the building $C$ violates the SFT stability condition.  Hence this case does not occur.

Next, to show that the count of curves in $\calM_{M\setminus \io(E_{\sk}),A}^{J,\simp}(\eta_m)$ is independent of $J$, 
suppose that $J_0$ and $J_1$ are two generic admissible almost complex structures on (the completion of)  ${M \setminus \io(E_{\sk})}$, joined by a generic $1$-parameter family $J_t$. 
As in the proof of Proposition \ref{prop:tanpseud}, we consider the associated $t$-parametrized moduli space 
$\calM_{M\setminus \io(E_{\sk}),A}^{\{J_t\},\simp}(\eta_m)$, which is a smooth one-manifold with boundary
$$ 
\calM_{M\setminus \io(E_{\sk}),A}^{J_1,\simp}(\eta_m) \bigcup \left(-\calM_{M\setminus \io(E_{\sk}),A}^{J_0,\simp}(\eta_m) \right).
$$
We claim that the SFT compactification $\ovl{\calM}_{M\setminus \io(E_{\sk}),A}^{\{J_t\}}(\eta_m)$ does not add anything new.
Indeed, a priori we must add pseudoholomorphic buildings consisting of a $J_t$-holomorphic top level in ${M \setminus \io(E_{\sk})}$ for some $t \in (0,1)$,
together with some number of symplectization levels in $\R \times \bdy E_\sk$. 
Since each $J_t$ is admissible as in Definition~\ref{def:admiss},
for each $t$  all symplectization levels appear with the same almost complex structure, and hence the previous index considerations still apply.
By genericity of the family $J_t$, we can also assume that all simple curves in $M \setminus \io(E_\sk)$ have index at least $-1$.
On the other hand, a closer inspection of the previous index argument in fact shows that any nontrivial building arising in $\ovl{\calM}_{M\setminus \io(E_{\sk}),A}^{\{J_t\}}(\eta_m)$ would have to involve an underlying simple component of index at most $-2$.
Hence we may conclude as before that such degenerations do not arise.

Now we consider the dependence on the parameter $\eps$.    Suppose that $0 < \wt{\eps} < \eps$ and that
$\wt{\io}$ is the restriction of $\io$ to $\wt{E}_{\sk}= E(\wt{\eps},\wt{\eps}s_2,...,\wt{\eps} s_n)$.
Then there is a natural symplectomorphism between the completions of ${M \setminus \wt{\io}(\wt{E}_{\sk})}$ and ${M \setminus \io(E_{\sk})}$,
under which an admissible almost complex structure $J$ pulls back to an admissible almost complex structure $\wt{J}$ on ${M \setminus \wt{\io}(\wt{E}_{\sk})}$.
This shows that counts of curves in $\calM_{M\setminus \io(E_\sk),A}^{J,\simp}(\eta_{m})$
and $\calM_{M\setminus \wt{\io}(\wt{E}_\sk),A}^{\wt{J},\simp}(\eta_{m})$ coincide.

Similarly, if $\io$ and $\wt{\io}$ are two symplectic embeddings of
skinny ellipsoids  $E(\eps,\eps s_2,...,\eps s_n)$ and
$E(\wt{\eps}, \wt{\eps} s_2,...,\wt{\eps} s_n)$ respectively, then $\io$ and $\wt{\io}$ are Hamiltonian isotopic after possibly shrinking both $\eps$ and $\wt{\eps}$ to some common sufficiently small value $\eps'$.
Then we have a symplectomorphism between the completions of ${M \setminus \wt{\io}(E_{\sk})}$ and ${M \setminus \io(E_{\sk})}$,
and hence the corresponding counts again coincide.
\end{proof}

\sss
We complete this subsection by discussing independence of the parameters $s_i$.  Since we are interested in counting curves of index zero with negative end on $\eta_m$, the proof of Lemma~\ref{lem:skin} shows that 
$E(\eps,{\eps s_2,\dots,\eps s_n})$ is skinny provided that $s_2>m$.

\begin{prop}\label{prop:s_indep}
For $m = c_1(A) - 1$, the count of curves in $\calM^{J,\simp}_{M\setminus \io(E_\sk),A}(\eta_m)$ is independent of the choice of  parameters $s_2< \dots < s_n$ provided that $s_2>m$.
\end{prop} 
\begin{proof}    Given any two sets of parameters $ s^k_2< \dots < s^k_n$ with $s^k_2>m$ for $k=1,2$, choose parameters $\wt{s_i}, \wt{\eps},  \eps$ so that
\begin{align*} 
m< \wt{s_2}<\dots < \wt{s}_n < \min(s^1_2, s_2^2),\quad  0 <\frac{m-1}{m} \eps < \wt{\eps} < \eps, 
\end{align*}
and then define
$$
E^k_\sk: = 
E^{2n}_\sk = 
E(\eps, \eps s^k_2,\dots,\eps s^k_n),\quad \wt{E}_\sk: = \wt{E}^{2n}_\sk =
 E(\wt{\eps},{\wt{\eps} \wt{s_2},\dots, \wt{\eps} \wt{s}_n}). 
$$
We view $\wt{E}_\sk$ as a subdomain of each $E^k_\sk$ in the obvious way, and denote by $\wt{\io}: \wt{E}_\sk \hookrightarrow M$ the restriction to $\wt{E}_\sk$ of the symplectic embedding $\io^k: E^k_\sk \hookrightarrow M$.
Since we saw above that the count is independent of the choice of $\io$ and $\eps$, 
the counts for $E^1_\sk$ and $E^2_\sk$ will be equal  provided that
 $$
 \calM^{J,\simp}_{M\setminus \io^k(E^k_\sk),A}(\eta_m) = \calM^{J,\simp}_{M\setminus \wt{\io}(\wt{E}_\sk),A}(\eta_m), \qquad k=1,2.
 $$

To prove this, we simplify notation by writing  $E_\sk: = E^k_\sk$, and
 consider the result of neck stretching along $\bdy E_\sk$, which we realize by a family of almost complex structures $J_t$, $t \in [0,1)$, on the symplectic completion of $M \setminus \io(\wt{E}_\sk)$.  
We choose $J_t$ so that both inside and below the neck region the noncompact symplectic divisors $D_i := \{z_i = 0\}$ are  holomorphic for $i=1,\dots,n$.  (This makes sense  
because these divisors  in $\Op(E_\sk)$ can be assumed invariant under neck stretching.) 
However, we assume that $J_t$ is otherwise generic so that all somewhere injective $J_t$-holomorphic curves are regular.
For clarity, we will denote the Reeb orbits of $\bdy E_\sk$ by $\eta_i$, $i \in \Z_{> 0}$ and the Reeb orbits of $\bdy \wt{E}_\sk$ by $\wt{\eta}_i$, $i \in \Z_{> 0}$.
 By arguing as in the proof of Proposition~\ref{prop:oneend} we find that the only possible limiting configurations as $t \rightarrow 1$ are two-level pseudoholomorphic buildings with
\begin{itemize}
\item top level in $M \setminus \io(E_\sk)$ consisting of a plane with negative end $\eta_m$
\item bottom level in $E_\sk \setminus \wt{E}_\sk$ consisting of a cylinder with positive end $\eta_m$ and negative end $\wt{\eta}_m$.
\end{itemize}
Indeed, the arguments involving formally gluing curve components work equally well if we replace the symplectization $\mathbb{R} \times \bdy E_\sk$ with the symplectic cobordism $E_\sk \setminus \wt{E}_\sk$. 
The only place where some care is needed is in the analog of Lemma~\ref{lem:mc1}, which is used to rule out index zero curves in $E_\sk \setminus \wt{E}_\sk$ with more than one positive end, since this involves an energy argument. 
Specifically, note that an irreducible rational curve in $E_\sk \setminus \wt{E}_\sk$ with positive ends $\eta_{m_1},\dots,\eta_{m_b}$ and negative end $\wt{\eta}_m$ has index
$$ 2(b-1) + 2(\sum_{i=1}^b m_i - m).$$
This is positive provided that we have $\sum_{i-1}^b m_i - m \geq 0$.
By nonnegativity of energy, we have
$$ 
\sum_{i=1}^b \eps m_i - \wt{\eps} m \geq 0,
$$
which implies that  $\sum_{i-1}^b m_i - m \geq 0$ 
since by assumption
 $\frac{\wt{\eps}}{\eps} > \frac{m-1}{m}$.

We now complete the argument as follows.  
The next lemma shows that
we may choose  $J$ so that all curves are regular and 
there  is a unique cylinder in $E_\sk \setminus \wt{E}_\sk$ with positive end $\eta_m$ and negative end $\wt{\eta}_m$.  By Pardon~\cite[Thm.2.54]{Pardcnct}, we can uniquely glue this cylinder to any given regular curve in $\calM^{J,\simp}_{M \setminus \io(E_\sk),A}(\eta_m)$ to get a curve in $\calM^{J_t,\simp}_{M \setminus \io(\wt{E}_\sk),A}(\wt{\eta}_m)$ for $t$ sufficiently close to $1$.  Therefore this establishes a bijection between $\calM^{J,\simp}_{M \setminus \io(E_\sk)}(\eta_m)$ and $\calM^{\wt{J},\simp}_{M \setminus \wt{\io}(\wt{E}_\sk)}(\wt{\eta}_m)$.
\end{proof}

\begin{lemma}\label{lem:unique_cob_cyl}
Let $J$ be a generic almost complex structure on the symplectic completion of the symplectic cobordism $E_\sk \setminus \wt{E}_\sk$
such that  the symplectic divisors 
$D_i:= \{z_i = 0\}$ are $J$-holomorphic for $i=1,\dots,n$. 
Then there is a unique $J$-holomorphic cylinder which is positively asymptotic to $\eta_m$ and negatively asymptotic to $\wt{\eta}_m$.
\end{lemma}

\begin{proof}
The basic idea is to use positivity of intersections to argue that any $J$-holomorphic curve $C$ with the given ends must be entirely contained in the divisor $D = D_n$. This inductively reduces the problem to the two-dimensional case, where the result easily follows by branched cover considerations. 
Observe, however, that this argument is complicated by the fact that both $C$ and $D$ are noncompact, and hence there is no canonical homological intersection number between them.
Since $C$ is asymptotic to $D$, this necessitates a discussion of ``intersections at infinity'', and one expects a well-defined homological intersection number after specifying the behavior at infinity. However, since the behavior of $C$ at infinity is a priori unknown, we need a way of controlling its asymptotic behavior in terms of known quantities.

We resolve these issues using the higher dimensional analog of Siefring's intersection theory \cite{Sief,Sief2} for punctured pseudoholomorphic curves, as described in \cite{MorS}. 
If $C$ is not entirely contained in $D$ then we claim that by \cite[Thm.2.2]{MorS} the set $C\cap D$ is compact.
Since we are not precisely in the setting of that result, we explain in more detail why it applies.
That result concerns the asymptotic behavior of finite energy curves in a symplectization $(\R\times X, J)$ (where as always $J$ is $\R$-invariant near infinity) that is equipped with a $J$-holomorphic divisor $D$ that has the form $[R,\infty)\times V^+$ (resp. 
$(-\infty,R]\times V^-$) near the positive (resp. negative) end for suitable codimension two submanifolds $V^\pm$ of $X$.\footnote
{
The paper  \cite{MorS}  actually considers a more general situation in which the divisor $\Tilde V$ is suitably asymptotic to the products $[R,\infty)\times V^+, [-R,-\infty)\times V^-$.}
Our manifold $(M,J)$  is not a global symplectization, though it is such outside a compact set.  However, that suffices for  the proof since that involves analyzing the structure of a $J$-holomorphic curve of finite energy near its two ends that are assumed asymptotic to $D$.
The main result in \cite[Thm.2.2]{MorS} is that if $C$ is not entirely contained in $D$ then 
near infinity
it decays exponentially towards $D$; in fact its distance from $D$ is specified in their formula  (8) which does not vanish.  
Hence, as in \cite[Cor.2.3]{MorS}, 
it follows that  $C\cap D$ is compact.

Next note that because  both $D$ and $C$ are $J$-holomorphic each intersection of $C$ with $D$ counts positively.  Since $D$ is embedded, this is the easy case of positivity of intersections that follows, for example, from the local normal form for the curve $C$ given in \cite[Eq.(E.1.1)]{JHOL}. Therefore, if  $C$ is not entirely contained in $D$, we may conclude that $C \cdot D \geq 0$.

Now observe  that $D$ is naturally identified with the symplectic completion of $E^{2n-2}_\sk \setminus \wt{E}^{2n-2}_\sk$. Moreover, $\bdy E_\sk^{2n-2}$ naturally sits inside $\bdy E_\sk^{2n}$ as a contact submanifold, and the restriction of the contact structure $\xi$ of $\bdy E_\sk^{2n}$ to $\bdy E_\sk^{2n-2}$ naturally splits as $\xi^T \oplus \xi^N$. Here $\xi^T$ denotes the contact structure on $\bdy E_\sk^{2n-2}$ and $\xi^N$ denotes its symplectic orthogonal.
There is a natural trivialization $\tau$ of $\xi^N$ along $\eta_1$, coming from identifying it with the extra $\mathbb{C}$ factor of $E_\sk^{2n}$ compared to $E_\sk^{2n-2}$.
Using this trivialization, any Reeb orbit $\gamma$ in $$
\p E_\sk^{2n-2}\times \{0\}\subset  \p E_\sk^{2n}
$$
has a {\bf normal Conley--Zehnder index} $\cz_\tau^N(\gamma)$,
which measures the rotation of the $2$-dimensional symplectic vector spaces $\xi^N$ along $\gamma$.
For $s_2$ sufficiently large compared to $m$, we have $\cz_\tau^N(\eta_m) = 1$ (c.f. the index computations in \cite[\S2.1]{Gutt-Hu}).

Following \cite{MorS}, let $C_\tau$ denote a perturbation of the curve $C$ which at infinity is specified by the trivialization $\tau$, 
and is given by pushing $C_\tau$  in a direction that is constant w.r.t $\tau$.
 Then $C_\tau$ is disjoint from $D$ near infinity, and can be disjoined from $D$ completely by a further
 perturbation (this time  possibly large, but  compactly supported) in the direction of the extra $\mathbb{C}$ factor
 of $\wt{E}^{2n}_\sk\subset \wt{E}^{2n-2}_\sk\times \mathbb{C}.$   It follows that
 $C_\tau \cdot D = 0$ since this quantity is invariant under   perturbations of compact support.
   On the other hand, as explained after the statement of   \cite[Thm~2.5]{MorS}, one can prove the following formulas
\begin{align*}
C_\tau \cdot D &= C \cdot D - \wind_\tau(C,D), \\
\wind_\tau(C,D) &:= \wind_\tau^+(C,D) - \wind_\tau^-(C,D),
\end{align*}
where $\wind_\tau^+(C,D)$ (resp. $\wind_\tau^-(C,D)$) is defined to be the winding number at positive (resp. negative) infinity of $C$ about $D$ as measured by the trivialization $\tau$.
Since  $C\cdot D\ge 0$ while $C_\tau \cdot D =0$, 
we must have $\wind_\tau(C,D) \ge 0$.
But according to \cite[Cor~2.4]{MorS}, we have the estimates
\begin{align*}
\wind_\tau^+(C,D) &\leq \lfloor \cz^N_\tau(\eta_m)/2 \rfloor\\
\wind_\tau^-(C,D) &\geq \lceil \cz^N_\tau(\wt{\eta}_m)/2\rceil,
\end{align*}
which, because $\cz_\tau^N(\eta_m) = 1$, imply that $\wind_\tau(C,D) \leq -1$.  
This contradiction shows that $C$ must be entirely contained in $D=D_n$, as desired.
\end{proof}

\begin{rmk}\label{rmk:io}\rm    Since we have now shown that the invariants are independent of the inclusion $\io$ and the  parameters $s_i,\eps $, we will often simplify notation by removing $\io$ from the notation, writing $M \setminus E_\sk$ instead of $M \setminus \io(E_\sk)$ and so on.\hfill$\er$
\end{rmk}

\subsection{Curves in dimension four with multiple ends on a skinny ellipsoid}\label{ss:skin4D}

We now restrict to the case that $M$ is a four-dimensional symplectic manifold. 
As in \S\ref{ss:tan_multi}, we consider partitions 
$$
\Pp = (m_1,...,m_b), \qquad m_1\ge m_2\ge \dots \ge m_b  \geq 1
$$ 
of $|\Pp| := \sum_{i=1}^b m_i$, 
and we will define numbers 
$$
N^E_{M,A}\lll \Pp \rrr \in \mathbb{Z}
$$
as follows. 
As above, let $\io: E_\sk \to M$ be 
 a symplectic embedding of a small skinny ellipsoid $E_\sk = E(\eps,\eps s)$ for $s > 1$ sufficiently large and $\eps$ sufficiently small, denote by $\eta_1$ the minimal action simple Reeb orbit of $\bdy E_\sk$, and by $\eta_2,\eta_3,...$ its iterates.
Let $J$ denote an admissible almost complex structure on (the symplectic completion of) $M \setminus \io(E_\sk)$.
We define
$$
\Mm_{M\setminus E_\sk,A}^{J,\simp}\lll \Pp\rrr
$$
 to be  the moduli space of simple (i.e. somewhere injective)  $J$-holomorphic $A$-spheres in $M \setminus \io(E_\sk)$ with $b$ negative punctures asymptotic to $\eta_{m_1},...,\eta_{m_b}$ respectively. 
 This moduli space is naturally oriented and, by \eqref{eq:Find}, has dimension
\begin{align}\label{eq:Aind}
2c_1(A) - 2 - 2\sum_{i=1}^b m_i.
 \end{align}
In the case that this index is zero, for $J$ generic 
 $\calM_{M\setminus E_\sk,A}^{J,\simp}\lll \Pp\rrr$ 
is a smooth zero-dimensional manifold.  We define $N^{E}_{M,A} \lll \Pp \rrr$
  to be the count
  of its elements.
The following proposition shows that this count is well-defined and independent of all choices involved in its construction. 

\begin{prop}\label{prop:multneg}  Let $(M,\om)$ be a symplectic four-manifold.  For any partition $\Pp$ of $c_1(A)-1$, the count $N_{M,A}^E\lll \Pp \rrr$ is finite and independent of the choice of generic $J$.
 It is also independent of the embedding $\io$ 
 and the parameters $\eps$ and $s$, provided that $s$ is sufficiently large and $\eps$ is sufficiently small.
 Moreover, each curve in the moduli space is immersed and counts positively. 
\end{prop}
\NI We emphasize that the proof of Proposition~\ref{prop:multneg} uses four-dimensional techniques in an essential way, as we explain below. In principle it might be possible to use intersection theory for punctured curves to bootstrap our arguments to higher dimensions (c.f. the proof of Proposition~\ref{prop:s_indep}), but we will not attempt this.

\begin{remark} \rm
One subtle point is that for generic $J$ there may be nontrivial pseudoholomorphic buildings in the SFT compactified moduli space $\ov{\calM}_{M\setminus \io(E_\sk),A}^{J}\lll \Pp\rrr$.    
In such situations the compactification 
is not transversely cut out, and one expects to require a virtual perturbation scheme in order to extract counts.
However, these buildings 
will not contribute to $N^E_{M,A}\lll \Pp \rrr$, since this is {\it defined} in dimension four to be a count of curves in ${\calM}_{M\setminus \io(E_\sk),A}^{J,\simp}\lll \Pp\rrr$ and thereby excludes buildings and multiple covers.
Notice that typically the closure of ${\calM}_{M\setminus \io(E_\sk),A}^{J,\simp}\lll \Pp\rrr$ 
in the SFT compactification $\ov{\calM}_{M\setminus \io(E_\sk),A}^{J}\lll \Pp\rrr$ will be a proper subspace.
It is not a priori clear whether analogous counts can be defined in higher dimensions, since we rely on four dimensional technique to rule out various undesirable degenerations.

As a simple example, consider the moduli space
$\calM_{\CP^2 \setminus E_\sk,[L]}^{J,\simp}\lll (1,1)\rrr$ consisting of curves
 in the symplectic completion of $\CP^2 \setminus \io(E_\sk)$ in the class of a line with two negative ends both asymptotic to $\eta_1$.
One can use the methods of  ECH (explained below) to conclude that this moduli space is empty. However, the full SFT compactification $\ovl{\calM}_{\CP^2 \setminus \io(E_\sk)}^J\lll (1,1)\rrr$ includes the two-level building whose top level lies in $\CP^2 \setminus \io(E_\sk)$ and consists of a plane with one negative end on $\eta_2$, and whose bottom level is a pair of pants in the symplectization $\R \times \bdy E_\sk$ given by a double branched cover of the trivial cylinder over $\eta_1$.
\hfill$\er$
\end{remark} 

\MS

Before giving the proof of Proposition~\ref{prop:multneg}, we  
recall some key facts about punctured pseudoholomorphic curves in dimension four.
Together these tools play a central role in the theory of embedded contact homology (ECH). 
For a more detailed introduction see~\cite{Hlect}; a shorter summary of relevant information can also be found in \cite[\S2.2]{Ghost}.
 Most of  these formulas are specific to dimensions three and  four.  However, below we also discuss relevant index formulas   for curves in ellipsoids that are valid in all dimensions.
\begin{itemlist}
\item {\bf Relative adjunction formula.} If $C$ is a somewhere injective punctured pseudoholomorphic curve in a completed four-dimensional symplectic cobordism with positive and negative Reeb orbit asymptotic ends, there is a relative adjunction  formula
 \begin{align}\label{eq:adjuc} c_\tau(C) = \chi(C) + Q_\tau(C) + \wr_\tau(C)  - 2\delta(C).
 \end{align}
Here $\tau$ is a choice of trivialization of the contact distribution $\xi$ over the ends of $C$, and the terms involved are as follows:
 \begin{itemize}
 \item
 $c_\tau(C)$ is the relative first Chern class of $C$;
 \item
 $\chi(C)$ is the Euler characteristic of $C$;
 \item 
 $Q_\tau(C)$ is the ``relative intersection pairing'' of $C$ (a generalization of the homological self-intersection number $[C] \cdot [C]$ in the case that $C$ is closed);
 \item
 $\wr_\tau(C) = \wr_\tau^+(C) - \wr_\tau^-(C)$, where $\wr_\tau^+$ (resp. $\wr_\tau^-)$ is the writhe of the asymptotic link determined by the positive (resp. negative) ends of $C$ (for more detail about these terms,
 see \eqref{eq:w+},~\eqref{eq:w-});
 \item 
 $\delta(C) \geq 0$ is an algebraic count of the singularities of $C$ (with each ordinary double point contributing $+1$).
\end{itemize}

 It is useful to note that this is a topological formula, and so in particular it applies to any smooth compact subdomain of a curve 
  in a symplectically complete manifold  
  (sometimes called a {\bf curve portion} below)  that is sufficiently large for its  boundary components to correspond naturally to braids about Reeb orbits in some limiting contact $3$-manifold.  In this situation,  the braids  are determined by the numerical data satisfied by the ends, giving formulas such as \eqref{eq:w+},~\eqref{eq:w-}.
We will apply this idea in the proof of Proposition~\ref{prop:multneg} below in the context of a $1$-parameter family of smooth curves degenerating to an SFT-type building when the \lq neck' around a contact hypersurface is stretched. In this situation, 
by cutting a curve that is close to breaking just above or below the neck region
we can find a curve portion just before degeneration which 
approximates 
some matched component of the limiting building; and,  as in the discussion before \eqref{eq:w+M} below,  one can often use knowledge of the limiting building to estimate the terms in \eqref{eq:adjuc}.     This will allow us to
rule out certain degenerations by showing that the existence of the corresponding curve portion would violate the relative adjunction formula. For other applications of this technique, see the proof of Lemma~\ref{lem:ppt_breakings}.

\item {\bf Rotation angles.} Let $Y^3$ be a contact three-manifold with a contact form $\alpha$ with nondegenerate Reeb orbits.
We can associate to each Reeb orbit $\gamma$ a linearized return map $P_\gamma$ which is a symplectomorphism of the contact plane $(\xi|_{\gamma(0)},d\alpha)$ without $1$ as an eigenvalue.
The Reeb orbit is {\em elliptic} if the eigenvalues of $P_\gamma$ lie on the unit circle, {\em positive hyperbolic} if the eigenvalues are real and positive, and {\em negative hyperbolic} if the eigenvalues are real and negative.
As we traverse $\gamma$, we can view the eigenspaces of $P_\gamma$ as rotating with respect to a trivialization $\tau$ by a total amount $2\pi\theta$, where $\theta:
 = \theta_\tau$ 
is the ``rotation angle'' measured with respect to $\tau$.   Here $\theta$ is irrational in the elliptic case, an integer in the positive hyperbolic case, and an integer plus $1/2$ in the negative hyperbolic case.
Using this same trivialization $\tau$, we have the formula
$$ \cz_\tau(\gamma) =  \lfloor \theta_\tau\rfloor + \lceil \theta_\tau \rceil.$$

\item {\bf Computations for ellipsoids (in all dimensions).}
In this paper we are mostly interested in either closed symplectic manifolds or symplectic cobordisms with one or more ends modeled on the boundary of an ellipsoid.
Consider $Y = \bdy E(a_1,...,a_n)$ with $a_1 \leq \dots \leq a_n$ rationally independent.
Let $\gamma_{k,m}$ be the $m$-fold iterate of the $k$th simple Reeb orbit, which has action $ma_k$.
If we use the trivialization $\tau_\ex = \tau_{\rm extend}$ which extends over a spanning disk in $Y$ for 
$\gamma_{k,m}$, 
 then
as mentioned in \S\ref{ss:skinhigh} we have 
$$
\cz_{\tau_\ex}(\gamma_{k,m}) = n-1 + 2\sum_{i=1}^n \left\lfloor \frac{ma_k}{a_i} \right\rfloor.
$$ 
There is also another convenient trivialization $\tau_\sp: = \tau_{\rm split}$ which comes from the observation that the $k$th simple Reeb orbit is identified with $Y \cap \{z_k = 0\}$, and therefore the contact distribution along it is naturally identified with the remaining $n-1$ factors of $\C^n$.
Since the latter trivialization respects the product structure on $\C^n$,  the rotation angle of $\gamma_{k,m}$ in the $j$th factor 
 is well defined for $j \neq k$,  and  is given by 
 \begin{align}\label{eq:rotang}
\theta_{\sp}(\gamma_{k,m}) =  ma_k/a_j,
\end{align} 
so that we have
(see \cite[\S2.1]{Gutt-Hu})
\begin{align}\label{eq:thetasp}
\cz_{\tau_\sp}(\gamma_{k,m}) =n-1 + 2\sum_{j\neq k} \left\lfloor \frac{ma_k}{a_j} \right\rfloor.
\end{align}
Correspondingly, because the formula \eqref{eq:Find} is independent of the choice of trivialization, a plane $C$ in the symplectic completion of $E(a_1,...,a_n)$ with top boundary on $\ga_{k,m}$ has
\begin{align}\label{eq:Cspex}
c_{\tau_{\sp}}(C) = m,\qquad  c_{\tau_{\ex}}(C) =  0.
\end{align}

Now consider the special case of $X = M \setminus E_\sk$ with $M^{2n}$ a closed symplectic manifold,
and let  $\eta_m$ be the $m$-fold iterate of the short orbit on $\p E_{\sk}$.
Then by the above (and assuming
as always  that $E_{\sk} = \eps E(1,s_2,\dots,s_n)$ for $s_2$ sufficiently large compared to $m$)  we have
\begin{align}\label{eq:CZ}
\cz_{\tau_\ex}(\eta_m) = n-1+ 2m,\;\;\;\;\; \cz_{\tau_\sp}(\eta_m) = n-1.
\end{align}
Further, given  a curve $C$ in (the symplectic completion of) $X$ with negative ends 
$(\eta_{m_1},\dots,\eta_{m_b})$, we
 can naturally associate to $C$ a homology class $A \in H_2(M;\Z)$, such that we have
\begin{align}\label{eq:c1}
c_1(C) := c_{\tau_\ex}(C) = c_1(A),\qquad Q_{\tau_\ex}(C) = A \cdot A,
\end{align}
while with respect to the trivialization $\tau_\sp$ and bottom partition $\Pp$, we have
\begin{align}\label{eq:c1sp}
c_{\tau_\sp}(C) &= c_1(A) - |\Pp| = c_1(A) - \sum_{i=1}^b m_i,\quad 
Q_{\tau_\sp}(C) = A \cdot A.
\end{align}

\item {\bf Writhe bounds.} 
For a somewhere injective curve $C$ in a four-manifold $X$ that is asymptotic at both ends 
 to some Reeb orbits, the asymptotic writhes $\wr_\tau^\pm(C)$ are sums of contributions from each asymptotic orbit.  
In principle the asymptotic links arising at the ends of $C$ are not uniquely determined by the Reeb orbit asymptotics and homology class of $C$, making them difficult to control in practice. 
It turns out that each end of $C$ is approximately an eigenfunction for an associated ``asymptotic operator'' which depends only on the linearized Reeb flow (see \cite{HWZ_propertiesI}).
This can be used to give bounds for $\wr_\tau^\pm(C)$ in terms of the multiplicities and rotation angles at each end.
More precisely, if $C$ has several positive ends of multiplicities $m_1,\dots,m_b$ on a given simple elliptic Reeb orbit $\al$ with rotation angle $\theta (= \theta_\tau)$, then we have
\begin{align}\label{eq:w+}
 \wr_\tau^+ \leq \sum_{i,j=1}^b \max(\lfloor m_i\theta\rfloor m_j,\lfloor m_j\theta\rfloor m_i) - \sum_{i=1}^b \lfloor m_i\theta \rfloor.
\end{align}
 Similarly, for negative ends, we have
\begin{align}\label{eq:w-}
 \wr_\tau^- \geq \sum_{i,j=1}^b \min(\lceil m_i\theta\rceil m_j,\lceil m_j\theta\rceil m_i) - \sum_{i=1}^b \lceil m_i\theta \rceil.
\end{align}
When calculating writhes we will always use the trivialization $\tau_{\sp}$, so that by \eqref{eq:rotang} the rotation angle of the $k$-fold multiple of the short orbit on $\p E(1,a)$ is  $k/a$.

Note that  the writhe bounds are sharp (i.e. they are equalities) whenever the end partition  satisfies the ECH partition conditions: 
 see \cite[Remark~2.3.2(ii)]{Ghost}, \cite[\S3]{HuN}, and the discussion below about the ECH index.
 Writhe is a   measure of the relative twisting in the $3$-manifold $Y$ of the loop given by the slice $C\cap (\{T\}\times Y)$
 at level $T$ (for $|T|$ sufficiently large) around  the asympotic Reeb orbit $\al$, and is 
 only well-defined because of the way $C$ converges towards $\al$.
However, the fact that writhe is essentially a  
 topological measure means that it can be calculated 
 when one slices any  curve $C$ that lies sufficiently close to  part of a cylinder $[-R,R]\times \al$.  Thus, when we stretch the neck we can apply it to the ends of suitable \lq\lq curve portions'' as in the proof of Proposition~\ref{prop:multneg} below.

   The writhe bounds are also sharp
 for any individual end of a curve which, when taken on its own, satisfies the ECH partition condition for its given multiplicity.
In particular, by \eqref{eq:rotang} the writhe $\wr_\tau^-(\eta_m)$ of any negative end on $\eta_m$ (the $m$-fold iterate of the short orbit on $\p E_{\sk})$ is given by
\begin{align}\label{eq:wneg} 
\wr_{\tau_\sp}^-(\eta_m) = m-1.
\end{align}
More generally, for a curve with negative ends $\eta_\Pp: = (\eta_{m_1},\dots,\eta_{m_b})$ on a skinny ellipsoid forming the partition $\Pp = (m_1,\dots,m_b)$ with $m_1 \geq \dots \geq m_b$,\footnote
 {
Note that the papers  \cite{Hlect} and \cite{Ghost} typically list these multiplicities in increasing order.
} 
one can check that the writhe bound \eqref{eq:w-} gives
\begin{align}\label{eq:wnegPp} 
\wr_{\tau_\sp}^-(\eta_\Pp) \geq 2\delta(\Pp) + |\Pp| - b,
\end{align} 
where $\de(\Pp)$ is as in \eqref{eq:delta_def}.  A key point in the proof of Claim B below is that the above bound is also sharp for generic $J$.

\item {\bf Automatic transversality.}
Another key advantage of dimension four is the fact that pseudoholomorphic curves are sometimes automatically regular without any assumptions on the almost complex structure.
The most general version due to Wendl \cite{Wendl_aut} includes curves with punctures and singularities, although more basic versions for closed curves go back to Gromov and are sketched in  the proof of Lemma~\ref{lem:blow}.
The following simplified version will suffice for our purposes:

\begin{thm}[\cite{Wendl_aut}, see also \cite{HuN}]\label{thm:auttrans}
Let $X$ be a four-dimensional symplectic cobordism with each contact boundary component having nondegenerate Reeb flow, and let $J$ be any admissible almost complex structure on the symplectic completion of $X$.
Then any immersed $J$-holomorphic $C$ curve with Reeb orbit asymptotics is regular, provided that we have
\begin{align}\label{eq:aut_trans_cond}
2g(C) - 2 + h_+(C) < \ind(u),
\end{align}
where $h_+(C)$ denotes the number of ends asymptotic to positive hyperbolic orbits.
\end{thm}
\NI In particular, note that the conditions of this theorem are satisfied whenever we have $g(C) = h_+(C) = \ind(C) = 0$, which is the most relevant case for this paper.
In this case, a standard corollary of automatic transversality states that all such curves count with positive sign (see \cite[Rmk.3.2.5]{JHOL}). 
As we argued in Lemma~\ref{lem:blow}, 
the same is true for index zero rational curves in dimension four which satisfy multibranched tangency constraints.

\item {\bf ECH index.}
 We also mention for completeness that any somewhere injective curve $C$ with Reeb orbit asymptotics has an {\bf ECH index} of the form
 $$I(C) := c_\tau(C) + Q_\tau(C) + \cz_\tau^I(C),$$
 where $\cz_\tau^I(C)$ is a certain sum of Conley--Zehnder index contributions at the ends of $C$. 
 By combining the definition of the ECH index with the relative adjunction formula, the Fredholm index formula, and the writhe bounds, one can prove the following ECH index inequality, which is a cornerstone of embedded contact homology:
 \begin{align*}
 I(C) - \ind(C) \geq 2\delta(C).
 \end{align*}
 
For any given simple Reeb orbit $\gamma$, by looking at all positive ends of $C$ which are asymptotic to some cover of $\gamma$, we get a partition $\Pp = (m_1,...,m_b)$. 
In the case that $I(C) = \ind(C)$ (e.g. if the ECH index is zero), this partition is uniquely determined by the total multiplicity $|\Pp|$ and the rotation angle $\theta$ of $\gamma$, and in this case we say that $C$ {\bf satisfies the ECH partition conditions.}
Indeed, there is a precise formula for this partition in terms of a certain lattice point count associated to the data $(|\Pp|,\theta)$ (see \cite{Hlect} for more details).
A similar situation holds for the negative ends of $C$.
In the general case with $I(C) > \ind(C)$, one can think of $\delta(C)$ as giving an upper bound for how far the actual partitions can differ from the ECH partitions.

For the case most relevant for us, if $C$ has positive ends $(\eta_{m_1},\dots,\eta_{m_b})$ on $\bdy E_\sk$ with total multiplicity $k = \sum_{i=1}^b m_i$, the corresponding ECH partition condition is $m_1 = \dots = m_b = 1$ with $b = k$.
Similarly, if $C$ has negative ends $(\eta_{m_1},\dots,\eta_{m_b})$ on $\bdy E_\sk$ with total multiplicity $k = \sum_{i=1}^b m_i$, the corresponding ECH partition condition is $m_1 = k$ with $b = 1$.
 \end{itemlist}

\begin{proof}[Proof of Proposition~\ref{prop:multneg}] 
We break up the proof into three steps. 

\MS

\NI {\bf Step 1:}  
{\it If $J$ is a generic admissible almost complex structure on the symplectic completion of $M \setminus \io(E_\sk)$, 
the number of elements in $\calM_{M\setminus E_\sk,A}^{J,\simp}\lll \Pp \rrr$  is finite and independent of the choice of $J$.}

\begin{proof}
Let $J_0,J_1$ be generic, admissible (see Definition~\ref{def:admiss})
  almost complex structures, and let $J_t$, $t\in [0,1]$ be a generic homotopy of admissible almost complex structures starting at $J_0$ and ending at $J_1$.
Similar to before, by a standard Sard--Smale argument we can assume that all somewhere injective curves in 
$\calM_{M\setminus E_\sk,A}^{J_i,\simp}\lll \Pp \rrr$ 
have nonnegative index for $i=0,1$ and that all somewhere injective curves in the $t$-parametrized moduli space
$\calM_{M\setminus E_\sk,A}^{\{J_t\},\simp}\lll \Pp \rrr$ have index at least $-1$.
In order to prove both finiteness of $\calM_{M\setminus E_\sk,A}^{J,\simp}\lll \Pp \rrr$ and independence of $J$, we need to show that the $t$-parametrized moduli space 
$\calM_{M\setminus \io(E_\sk),A}^{\{J_t\},\simp}\lll \Pp \rrr$ is already compact.
In light of the SFT compactness theorem, it suffices to show no sequence of curves in 
$\calM_{M\setminus E_\sk,A}^{\{J_t\},\simp}\lll \Pp \rrr$ converges to 
either to a multiply covered curve, or to
a building in
a boundary stratum of $\ovl{\calM}_{M \setminus E_\sk,A}^{\{J_t\}}\lll \Pp\rrr$.  

Observe that such a building would consist of
\begin{itemize}
\item
a top level in $M \setminus \io(E_\sk)$, possibly nodal 
or disconnected or multiply covered
\item 
some nonnegative number of symplectization levels in $\R \times \bdy E_\sk$.
\end{itemize}
Let us formally glue together all the symplectization levels along paired Reeb orbit ends, resulting in a collection of matched components $C_1,...,C_c$ in $\R \times \bdy E_\sk$.

The index of any component $C$ in $M \setminus \io(E_\sk)$ is nonnegative, and strictly positive unless $C$ is simple.
Indeed, if $C$ is a $\kappa$-fold cover of its underlying simple curve $\ovl{C}$,
then we have from \eqref{eq:Aind} that
\begin{align}\label{eq:Aind1}
\ind(C) = \kappa\, \ind(\ovl{C}) + 2\ka - 2.
\end{align}
We can assume that $\ind(\ovl{C})\ge -1$ and hence $\ind(\ovl{C})\ge 0$ since the index is necessarily even by \eqref{eq:Find}.
It follows that $\ind(C)$ is nonnegative, and we can only have $\ind(C)$ if $\kappa = 1$.
Thus the limit cannot be multiply covered.

Similarly, the index of any matched component $C_i$ 
in $\R\times \p E_{\sk}$ 
 is nonnegative, and strictly positive if $C$ has more than one positive end.
Indeed, by \eqref{eq:Find} and \eqref{eq:CZ}, if $C_i$ has  positive ends $\eta_{m_1},...,\eta_{m_k}$ and negative ends $\eta_{m_1'},...,\eta_{m_l'}$ ends, then we have
\begin{align}\label{eq:indCi}
\ind(C_i) = 2k - 2 + 2\sum_{i=1}^k m_i - 2\sum_{j=1}^{l} m_j'.
\end{align}
By nonnegativity of energy, we have 
\begin{align}\label{eq:energ}
\sum_{i=1}^k m_i - \sum_{j=1}^{l}m_j' \geq 0,
\end{align} 
and hence we have $\ind(C_i) \geq 0$, with equality only if $k = 1$.

We saw above that
each component of the limiting building 
has index $\ge -1$.  Thus  
it follows from the above index considerations that we must have
\begin{itemize}
\item the top level in $M \setminus \io(E_\sk)$ is smooth and connected with one end for each
matched component $C_1,\dots,C_c$;
\item the matched components $C_i$ necessarily have zero energy and one positive end,
and at least one must have more than one negative end, since otherwise they are all  trivial cylinders by
Lemma~\ref{lem:unique_cob_cyl}, which would  violate the stability condition of the SFT compactness theorem.
\end{itemize}

It therefore remains to rule out the case that a matched component $C_i$ has more than one negative end.
For this we will appeal to the relative adjunction formula and the aforementioned writhe estimates.
Note that these tools cannot be directly applied to the matched component $C_i$, but it turns out that they can be applied to an approximation of $C_i$ and this suffices.
Namely, suppose that $C_i$ has say positive end $\eta_m$ and negative ends $\eta_{m_1},\dots,\eta_{m_{k}}$ where we must have $\sum_{i=1}^{k} m_i = m$ since $\ind(C_i) = 0$.
By hypothesis, the entire building is approximated by a sequence of somewhere injective curves in ${M \setminus \io(E_\sk)}$. 
Now  
cut such an approximating  curve just above the neck (which we assume sufficiently long) 
in a region where the curve
is asymptotic to a braid around 
$\eta_m$.  Then 
one component $C_-$ of 
its lower part 
is a curve portion with boundary that is a close approximation to the matched component $C_i$, 
and has
 well-defined writhes $\wr_\tau^+$ and $\wr_\tau^-$ at its top and bottom.\footnote
{For a detailed write-up of an argument of this kind together with references that establish the required exponential convergence properties,  see \cite[\S3]{HuN}  and in particular \cite[Lem.3.2]{HuN}.
This refers to the writhe of the positive end of a curve that is cut just below the neck, but essentially the same analysis applies to the writhe of the negative end of a curve cut just above the neck.}
Moreover, because we cut just {\it above} the neck, 
the top end of $C_-$ approximates a curve with {\it negative} end on $\eta_m$, so that 
its top writhe $\wr_{\tau_\sp}^+(C_-)$ 
has lower bound given by  Equation~\eqref{eq:w-}.  Thus
 $$
\wr_{\tau_\sp}^+(C_-) \geq \lceil m\theta\rceil m - \lceil m\theta\rceil = m-1,
$$
 where $\theta$ is the rotation angle of $\eta_m$ with respect to $\tau_\sp$
 and so is arbitrarily small as in \eqref{eq:thetasp}.
In fact, by the explanation following Equation~\eqref{eq:w-}, this inequality is sharp, i.e. we have 
\begin{align}\label{eq:w+M}
\wr_{\tau_\sp}^+(C_-) = m-1.
\end{align}
 Similarly, we can use Equation~\eqref{eq:w-} to estimate 
 $\wr_{\tau_\sp}^-(C_-)$.
Assuming without loss of generality that $m_1\ge m_2\ge \cdots \ge m_{k}$ 
where $k>1$, we have
\begin{align}\label{eq:w-M} \notag
\wr_{\tau_\sp}^-(C_-)&\; \geq\; \sum_{1\le i\le j\le m} m_j - k\; \\ \notag
&\;=  
m + 2\sum_{1\le i<j \le m} m_j - k\\ 
&\; \ge\;  m+2-k.
\end{align}
In this situation we have $c_{\tau_{\sp}}(C_-) = 0$, $Q_{\tau_\sp}(C_-) = 0$, and $\chi(C_-) = 1-k$,
and hence the relative adjunction formula \eqref{eq:adjuc} gives
\begin{align}\label{eq:wcontrad} \notag
0& \le 2 \delta  = \chi(C_-) + \wr_\tau^+(C_-) - \wr_\tau^-(C_-) \\
&  \le  1-k + m-1 - (m+2-k) = -2
\end{align}
which is impossible.  
\end{proof}

\NI {\bf Step 2:}  {\it  For generic admissible $J$ on $M\less E_{\sk}$, every element in $\calM_{M\setminus \io(E_\sk),A}^{J,\simp}\lll \Pp \rrr$ is immersed and counts with positive sign.}

\begin{proof} 
The moduli space $\calM_{M \setminus E_\sk,A}^{J,\simp}\lll \Pp \rrr$  consists of somewhere injective curves,
and any non-immersed curve would have a cusp point (i.e. a point where the derivative vanishes). The arguments in
Lemma~\ref{lem:cusp_deg1} that show that cusps do not occur in generic $1$-parameter families of curves with tangency constraints apply equally well in the current situation since they are local in nature (both in the domain and target).
This means that condition \eqref{eq:aut_trans_cond} holds, so Wendl's automatic transversality theorem applies,
and it then follows as in \cite[Prop~3.15]{HK} that all of the curves in $\calM_{M \setminus E_\sk,A}^{J,\simp}\lll \Pp \rrr$ count positively.
The rough idea here is that the space of almost complex structures $J$ for which the moduli space $\calM_{M \setminus E_\sk,A}^{J,\simp}\lll \Pp \rrr$ is regular is {\em path connected}. Indeed, we can join any two almost complex structures $J_0,J_1$ by a generic family $J_t$, and the corresponding parametrized moduli space will induce a cobordism between $\calM_{M \setminus E_\sk,A}^{J_0,\simp}\lll \Pp \rrr$ and $\calM_{M \setminus E_\sk,A}^{J_1,\simp}\lll \Pp \rrr$. Typically this cobordism will undergo bifurcations at isolated values of $t$, but in the present situation automatic transversality implies that $\calM_{M \setminus E_\sk,A}^{J_t,\simp}\lll \Pp \rrr$ is regular for all $t$ and hence that no bifurcations can occur. 
This means that this cobordism is in fact trivial, so we get a one-to-one correspondence between the curves in $\calM_{M \setminus E_\sk,A}^{J_0,\simp}\lll \Pp \rrr$ and $\calM_{M \setminus E_\sk,A}^{J_1,\simp}\lll \Pp \rrr$.
From this it follows that the curves in the two moduli spaces $\calM_{M \setminus E_\sk,A}^{J_0,\simp}\lll \Pp \rrr$ and $\calM_{M \setminus E_\sk,A}^{J_1,\simp}\lll \Pp \rrr$ count with the same signs, and it is not hard to show that these signs must all be positive since we can deform $J$ so as to be integrable near any given curve $C$ 
provided that 
no more than two branches of $C$ intersect at any point; see also the proof of Lemma~\ref{lem:blow}.
\end{proof}
\MS

\NI {\bf Step 3:}  {\it The number of elements in $\calM_{M\setminus \io(E_\sk),A}^{J,\simp}\lll \Pp \rrr$ does not
depend on  the choice of embedding $\io$,  scale factor $\eps$ and shape parameter $s$.}

\begin{proof}  The independence of  the choice  of $\io$ and $\eps$
follows exactly as in Proposition~\ref{prop:oneend}.
As for the dependence on $s$, put $E_\sk = E(\eps,\eps s)$ and $\wt{E}_\sk = E(\wt{\eps},\wt{\eps}\wt{s})$,
 for 
$0 < \wt{\eps} < \eps$ 
and $0 < \wt{\eps}\wt{s} < \eps s$, and let $\wt{\io}$ denote the restriction of $\io$ to $\wt{E}_\sk$. 
As in the proof of Proposition~\ref{prop:s_indep}, 
 we consider the effect of stretching the neck of  $M \setminus \wt{\io}(\wt{E}_\sk)$ along $\bdy E_{\sk}$.
For $\wt{s} > |\Pp|$ and $\frac{\wt{\eps}}{\eps} > \frac{|\Pp|-1}{|\Pp|}$ 
 the analog of Equation~\eqref{eq:energ} still holds, and  
essentially the same argument as above shows that the only possible breaking is a two level building in $(E_\sk \setminus \wt{E}_\sk) \notccirc (M \setminus \io(E_\sk))$ consisting of a connected somewhere injective top level $C_{top}$ in $M \setminus \io(E_\sk)$ and a union of cylinders in $E_\sk \setminus \wt{E}_\sk$.
Below we  prove the following claims:
\begin{itemlist}\item[{\rm (A)}] Each component in $E_\sk \setminus \wt{E}_\sk$ is necessarily the unique cylinder from
$\eta_k$ to $\wt{\eta}_k$  for some $k \geq 1$.
\item [{\rm (B)}]  If the almost complex structure  $J$  on  $M \setminus \io(E_\sk)$ is generic, then
 two different curves in $M \setminus \wt{\io}(\wt{E}_\sk)$ cannot degenerate to the same building in $(E_\sk \setminus \wt{E}_\sk) \notccirc (M \setminus \io(E_\sk))$.
\end{itemlist}
Granted this, it follows that we have 
\begin{align}\label{eq:ineq}
\# \calM^{J,\simp}_{M \setminus \wt{E}_\sk,A}\lll \Pp \rrr \;\; \leq\;\; \# \calM^{J,\simp}_{M \setminus E_\sk,A}\lll \Pp \rrr,
\end{align}
in the case
\begin{align*} 
|\Pp| < \wt{s}, s,
 \;\;\;\;\; \frac{\eps(|\Pp|-1))}{|\Pp|} < \wt{\eps} < \min(\eps, \frac{s\eps}{\wt{s}}).
\end{align*}
In particular, if $\wt{s} \le s$, then we can choose $\wt{\eps}<\eps$ sufficiently close to $\eps$ to satisfy the lower bound.
Hence beceause the counts are independent of the choice of $\eps, \wt{\eps}$ the inequality \eqref{eq:ineq} holds in this case.
Similarly, the inequality \eqref{eq:ineq} holds if $s< \wt{s} < \frac {|\Pp|}{|\Pp|-1}$ since in this case we can also choose suitable $\eps, \wt{\eps}$.  Hence by repeating 
 this procedure, this inequality also holds for all $s< \wt{s}$.  But then we must have equality in \eqref{eq:ineq}.
 \end{proof}
\MS

\NI {\bf Proof of Claim A:} 
This is a more general version of 
 Lemma~\ref{lem:unique_cob_cyl} that holds in dimension four and applies with no restriction on $J$ provided
 that it is admissible in the sense of Definition~\ref{def:admiss}.
 First note that by Lemma~\ref{lem:unique_cob_cyl} the signed count of cylinders in $E_{\sk}\less \wt{E}_{\sk}$ from $\eta_k$ to $\wt{\eta}_k$  is one.   If  there is more than one such cylinder when $k=1$, 
 we may apply the relative adjunction formula \eqref{eq:adjuc} to their union to get
$$ \wr_\tau = \wr_\tau^+ - \wr_\tau^- =  2\delta \ge 0,
$$
while the writhe bounds \eqref{eq:w+} give
\begin{align*}
\wr_\tau^+ \leq 0, \qquad 
\wr_\tau^- \geq 2,
\end{align*}
which is a contradiction.  A similar argument
 rules out the existence of a somewhere injective cylinder in the symplectic completion of $E_\sk \setminus \wt{E}_\sk$ from $\eta_k$ to $\wt{\eta}_k$ for some $k \geq 2$:   the relative adjunction formula would give
$$\wr_\tau = \wr_\tau^+ - \wr_\tau^- = 2\delta \ge 0,$$
while the writhe bounds give
\begin{align*}
\wr_\tau^+ \leq 0, \qquad 
\wr_\tau^- \geq k-1,
\end{align*}
which is a contradiction for $k \geq 2$.

\MS

\NI {\bf Proof of Claim B:}  
Suppose by contradiction that two different curves in $M \setminus \wt{\io}(\wt{E}_\sk)$ degenerate under neck stretching to the same building in $(E_\sk \setminus \wt{E}_\sk) \notccirc (M \setminus \io(E_\sk))$.
Then in particular we can find compact curve portions $C_1$ and $C_2$ which closely approximate\footnote
{
I.e. they are cut off so close to their negative ends that the writhes of their lower boundaries are well defined and can be calculated  as in
the above proof of Proposition~\ref{prop:multneg},  Step 1.
}
 the same top level $C$ in $M \setminus \io(E_\sk)$. 
If $\delta(C)$ denotes the count of singularities of $C$, we must have $C_1 \cdot C_2 \geq 2\delta(C)$ (this is most evident when $C_1$ and $C_2$ are immersed, which holds generically).
We will show that this violates the relative adjunction formula applied to the pair $C_1 \cup C_2$.

Indeed, if $\Pp = (m_1,\dots,m_b)$ where $m_1\ge m_2\ge \dots \ge m_b$
we have defined
$$
|\Pp|: = \sum_{i=1}^b m_i, \qquad \de(\Pp) : =  \sum _{i=1}^{b}  (i-1)m_{i}.
$$
Combining the adjunction inequality for a curve $C\in \calM_{M\setminus \io(E_\sk),A}^{J,\simp}\lll \Pp \rrr$ with the writhe bound \eqref{eq:w-}, we get
\begin{align}\label{eq:claimB}
2\de(C) & = \chi(C) + Q_{\tau_\sp}(C) - c_{\tau_\sp}(C) - \wr_{\tau_\sp}^-(C)\\ \notag
& \leq (2-b) + A^2 - (c_1(A) -|\Pp|) - (|\Pp| - b + 2\de(\Pp))\\ \notag
& =  2 + A^2 - c_1(A) - 2 \de(\Pp).
\end{align}
The key point now is that when $J$ is generic the lower bound in \eqref{eq:w-} for $\wr_{\tau_\sp}^-(C)$ for negative ends on a skinny ellipsoid is exact.
We noted at the end of our discussion of the ECH index that any end on $\eta_m$ (for $m<s$) satisfies 
the ECH partition condition for a negative end on $\p E_{\sk}$.  Hence by \cite[Prop~6.1]{Huind} the writhe estimate for this single end is exact.   In particular, it has the predicted winding number,  namely $1=  \lceil m_i/s\rceil$.
 It is clearly explained in  \cite[\S6.3]{Huind}  how to calculate the linking numbers of the different ends of $C$ and hence $\wr_{\tau_\sp}^-(C)$.  
Each end has an asymptotic expansion in terms of eigenfunctions of the asymptotic operator $A$, and, by
\cite[\S3]{HWZ_propertiesI},  for each winding number there is a  corresponding two-dimensional space $\R(\xi_1,\xi_2)$ of eigenfunctions (corresponding to the lowest two positive eigenvalues of $A$) that give rise to trajectories with that winding number.  
Therefore the lowest terms in the asymptotic expansions for the ends of $C$ must have the form $c_1\xi_1+ c_2\xi_2$ for some $(c_1,c_2)\in \R^2\less \{(0,0)\}$.
  The proof of \cite[Lemma~6.9]{Huind}  shows that the lower bound for $\wr_{\tau_\sp}^-(C)$ is exact if 
 each end of $C$ gives rise to a different pair $(c_1,c_2) $.  It remains to observe that  since the family $(J_t)_{t\in [0,1]}$ is  generic, we may assume that this condition is satisfied since it has codimension $2$.
Hence we may conclude that 
 $$
 2\de(C)  = 2 + A^2 - c_1(A) - 2 \de(\Pp).
 $$

Similarly, if $C_1, C_2$ are two distinct curves in $\calM_{M\setminus \io(E_\sk),A}^{J,\simp}\lll \Pp \rrr$, the writhe bound \eqref{eq:w-} gives that
\begin{align*}
 \wr_{\tau_\sp}^-(C_1\cup C_2) = 4|\Pp| - 2b + 8 \de(\Pp) 
 \end{align*}
so that
\begin{align*}
2\de(C_1\cup C_2) 
& = 2(2-b) + 4A^2- 2(c_1(A) -|\Pp|) - (4|\Pp| - 2b + 8 \de(\Pp))\\
& =  4 + 4A^2 - 2c_1(A) -2 |\Pp| - 8 \de(\Pp).
\end{align*}
The count of singularities of the union is given by
$$\delta(C_1 \cup C_2) = \delta(C_1) + \delta(C_2) + C_1 \cdot C_2 = 2\delta(C) + C_1 \cdot C_2,$$
and therefore we have 
\begin{align*}
C_1 \cdot C_2 &= \left(2 + 2A^2 - c_1(A) - |\Pp| - 4\delta(\Pp)\right) - \left(2 + A^2 - c_1(A) - 2\delta(\Pp)\right)\\
&= A^2 - |\Pp| - 2\delta(\Pp).
\end{align*}
But this implies that $C_1 \cdot C_2 < 2\delta(C)$, which as we noted above is impossible if $C_1$ and $C_2$ are both very close to $C$. 
 This completes the proof of Claim B and hence the proof of Proposition~\ref{prop:multneg}.
\end{proof}

Given partitions $\Pp_1,\dots,\Pp_r$ we can also define the invariants
\begin{align}\label{eq:defmanyPp}
N_{M,A}^E\lll \Pp_1,..., \Pp_r\rrr \in \Z
\end{align}
in essentially the same way
as $N_{M\setminus E_{\sk},A}\lll\Pp\rrr$, using a disjoint collection of small skinny ellipsoids $E_{\sk}^1,\dots, E_{\sk}^r$ in $M$.
We omit the proof that the corresponding curve counts are independent of all choices since it is essentially the same as the proof of Proposition \ref{prop:multneg}  above, except with slightly more cumbersome notation.

\section{Relationships between constraints}\label{sec:relationships}

In this section we establish some fundamental relationships between the various curve counts defined in the previous two sections. Taken together these will form the backbone of the recursion algorithm in the next section.
In \S\ref{ss:same} we prove an equivalence between curve counts with multibranched tangency constraints and curve counts with ends on a skinny ellipsoid. 
Our proof circumvents the gluing analysis for curves with tangency constraints via an argument which utilizes automatic transversality and obstruction bundle gluing, and hence is only valid in dimension four.
Secondly, in \S\ref{ss:comb} we give a formula which describes how to replace curve counts with ends on two disjoint skinny ellipsoids with curve counts with ends on a single skinny ellipsoid.
In light of \S\ref{ss:same}, this equivalently describes how to combine multibranched tangency constraints at two distinct points in the target into multibranched tangency constraints at a single point. 

\subsection{Equivalence of tangency and  ellipsoidal constraints in dimension four}\label{ss:same}

Let $M$ be a  symplectic four-manifold and $A \in H_2(M;\Z)$ a homology class.
By the results of the previous two sections, for each partition $\Pp = (m_1,\dots,m_b)$ of $c_1(A) - 1$ we have two enumerative invariants:
\begin{itemize}
\item $N_{M,A}^T\lll \Pp\rrr$ counts curves in $M$ that satisfy a multibranched tangency constraint $\lll \T^{m_1-1}p,\dots,\T^{m_b-1}p\rrr$ at a point
\item $N_{M,A}^E\lll \Pp \rrr$ counts curves in $M \setminus E_\sk$ with negative ends $\eta_{m_1},\dots,\eta_{m_b}$.
\end{itemize}
Our main goal in this subsection is to prove Theorem~\ref{thm:counts} which we restate here for the convenience of the reader.  
Note that the coefficient
\begin{align*}
\Pp! := \frac{|\Pp|!}{m_1!\dots m_b!} = \frac{(m_1+\dots+m_b)!}{m_1!\dots m_b!}.
\end{align*}
below represents the number of ways of 
decomposing an ordered set with $|\Pp|$ elements into $b$ unordered parts of sizes $m_1,...,m_b$.

\begin{thm}\label{thm:T=E}
Let $M$ be a symplectic four-manifold and $A \in H_2(M;\Z)$ a homology class.
For any partition $\Pp \in \partitions_{c_1(A) -1}$, we have $$
N_{M,A}^T\lll \Pp \rrr = N_{M,A}^E \lll \Pp \rrr =: N_{M,A} \lll \Pp \rrr.
$$
Further, the number $N_{M,A}$ of $A$-curves through $c_1(A) - 1$ generic points  is given by:
$$
N_{M,A} = \sum_{\Pp\in \partitions_{c_1(A) -1}} \Pp!\, N_{M,A}\lll \Pp\rrr.
$$
\end{thm}

The basic heuristic reason why Theorem~\ref{thm:T=E} holds is that essentially the only rigid punctured curves in a skinny ellipsoid which satisfy a $\lll \T^{m-1}p\rrr$ constraint are $m$-fold covers of the plane passing through $p$ and positively asymptotic to the simple short Reeb orbit in $\bdy E_\sk$. As a consequence, given a $\lll \T^{m-1}p\rrr$ constraint, surrounding it with $E_\sk$ and neck stretching along $\bdy E_\sk$ should replace it with a negative $\eta_m$ end; conversely, gluing should fill in a negative $\eta_m$ end with a $\lll \T^{m-1}p\rrr$ constraint in $E_\sk$. A similar heuristic applies for the multibranched constraints. However, we note that this type of gluing problem is somewhat nonstandard, since it involves curves satisfying tangency constraints. 
In particular, in the situations of interest we often want to glue a curve in $M \setminus E_\sk$ to a multiply covered curve in $E_\sk$ with a branch point at $p$, which satisfies the tangency constraint 
$\lll \T^\Pp p\rrr$ but in a rather degenerate way. 
In order to sidestep such issues, we will instead give a somewhat indirect argument which utilizes automatic regularity in dimension four and obstruction bundle gluing. 
These ideas 
will be placed in a broader framework in \S\ref{ss:comb}.

In more detail, this subsection is structured as follows:
\begin{itemlist}
\item 
In Lemma~\ref{lem:push_ineq}, we establish the inequality $$
N_{M,A}\;\; \leq\;\; \sum _{\Pp} \,\Pp!\, N_{M,A}^T \lll \Pp \rrr, 
$$
 where the sum is over all partitions $\Pp =(m_1,...,m_b)\in \partitions_m$ where $m={c_1(A) - 1}$.
The basic idea here is to push all of the points $p_1,\dots,p_m$ along $D$ until they all coincide with the point $p$.
We argue that, after compactifying this parametrized moduli space, every curve degenerates into a curve with a multibranched tangency constraint, and moreover no two geometrically distinct curves can degenerate to the same limit.
\item
In Lemma~\ref{lem:T_to_E}, we establish the inequality $N_{M,A}^T\lll \Pp \rrr \leq N_{M,A}^E \lll \Pp \rrr$.
This essentially follows by a neck stretching argument after surrounding the point $p$ by a small skinny ellipsoid $E_\sk$.
Here we rule out undesirable degenerations using essentially the same techniques as in \S\ref{ss:skin4D}.
\item 
Finally, in Lemma~\ref{lem:glue_E}, we establish the inequality $$
\sum_{\Pp} \,\Pp!\, N_{M,A}^E \lll \Pp\rrr \;\;\leq\;\; N_{M,A}.
$$
 Here the idea is that we can use obstruction bundle gluing as in \cite{HuT} to fill in all of the negative ends by a collection of holomorphic planes in $E_\sk$, each of which passes through one of the points $p_1,\dots,p_m$.
\end{itemlist}
\begin{proof}[Proof of Theorem~\ref{thm:T=E}]
By combining the aforementioned lemmas, we have
\begin{align*}
N_{M,A}\;\; \leq\;\; \sum_{\Pp} \,\Pp!\,N_{M,A}^T \lll \Pp \rrr \;\;\leq\;\; \sum_{\Pp} \Pp!\, N_{M,A}^E \lll \Pp \rrr \;\;\leq\;\; N_{M,A}. 
\end{align*}
Hence all of the inequalities are equalities; in particular, we must have $N_{M,A}^T \lll \Pp \rrr = N_{M,A}^E\lll \Pp \rrr$ for each $\Pp$.
\end{proof}

 \sss

\begin{lemma}\label{lem:push_ineq}
We have $N_{M,A} \leq \sum_{\Pp}\,\Pp!\, N_{M,A}^T \lll \Pp \rrr$, where the sum is over all partitions $\Pp \in \partitions_{c_1(A) - 1}$.
\end{lemma}
\begin{proof}
Put $m = c_1(A) - 1$.
Let $D$ be a local divisor at $p \in M$, and let 
$\bp_1,\bp_2,\bp_3,\dots$ be a sequence in $D^{\times m} \subset M^{\times m}$ which converges to 
$\bp_{\infty} := (p,\dots,p)$.
Consider the moduli space $\calM_{M,A,m}^{J,\simp}$ of simple $A$-curves with $m$ marked points.
By semipositivity, for generic $J \in \Jj_D$ the evaluation map
$$
\ev_m:  \calM_{M,A,m}^{J,\simp} \to M^{\times m}
$$
is a pseudocycle of dimension $4m$,
and moreover we can assume that each of the points $\bp_k \in M^{\times m}$ is a regular value.
Then the signed count of points in $\ev_m^{-1}(\bp_k)$ is $N_{M,A}$ for each $k \in \Z_{> 0}$, and in fact by 
automatic transversality (c.f. Theorem~\ref{thm:auttrans}) this is also the unsigned count. 

By forgetting the constraints, we have a natural inclusion map $\ev_m^{-1}(\bp_k) \hookrightarrow \ovl{\calM}_{M,A,m}^J$ for $k \in \Z_{> 0}$.
Let $\lim_{k \rightarrow \infty}\ev_m^{-1}(\bp_k)$ denote the set of limit points as $k \rightarrow \infty$ in this ambient compact space.
That is, each element of $\lim_{k \rightarrow \infty}\ev_m^{-1}(\bp_k)$ is a limit $C_\infty = \lim_{k \rightarrow \infty}C_k$ of stable maps $C_k \in \ev_m^{-1}(\bp_k)$.
Note that each marked point of $C_\infty$ maps to $p$, although in general $C_\infty$ will be a nodal configuration consisting of one or more ghost components.
Since $J \in \Jj_D$ is generic, by \eqref{eq:ord}
and index considerations similar to those in the proof of Proposition~\ref{prop:tanpseud}, we can assume that
$C_\infty$ has exactly one nonconstant component, which must represent the homology class $A$.
Let us prune away all of the ghost components from $C_\infty$ and mark the nearby special points on 
the nonconstant component 
as in steps (1) and (2) in the proof of Proposition~\ref{prop:tanpseud}.
The resulting curve $\ovl{C}_\infty$ has marked points $z_1,\dots z_a$ for some $a \leq m$ (with some of these newly added) and satisfies the condition $\sum_{i=1}^a \ord(\ovl{C}_\infty,D;z_i) = m.$\footnote{Note that it is crucial here that all of the points involved in the sequence $\bp_1,\bp_2,\bp_3,\dots$ lie in $D$, so that the total contact order with $D$ is preserved by the limiting curve $\ovl{C}_\infty$. If we instead took an arbitrary sequence of points in $M^{\times m}$ limiting to $\bp_\infty$, the limiting curves could satisfy more complicated constraints which depend on the collision directions of the sequence and are unrelated to the divisor $D$.}
In other words, $\ovl{C}_\infty$ defines an element of $\bigcup_{\Pp \in \partitions_m} \calM^{J,\simp}_{M,A}\lll \T^\Pp p\rrr$.

Let $N := N_{M,A}$, and choose a labelling $(C_{k,i})_{i=1}^N$ of the $N$ elements in $\ev_m^{-1}(\bp_k)$ for each $k \in \Z_{> 0}$.
By passing to a subsequence, we may assume that each sequence $(C_{k,i})_{k  \geq 1}$ converges to a limiting curve 
$$
C_{\infty,i} \in \lim_{k \rightarrow \infty}\ev_m^{-1}(\bp_k).
$$
Recall from \eqref{eq:noHatN} that $N^T_{M,A}\lll \Pp \rrr$ is the count of curves whose branches through $p$ are {\em unordered}.
Note each sequence $C_{1,i},C_{2,i},C_{3,i},\dots$ determines not just a pruned limiting curve $\ovl{C}_{\infty,i}$, but also a 
decomposition $\Qq_i$ of $\{1,\dots,m\}$ into subsets whose sizes are given by a partition $\Pp: = (m_1,\dots,m_b)$ of $m$,
where $\Qq_i$ is determined  by recording which marked points of $C_{k,i}$ coalesce into each branch of $\ovl{C}_{\infty,i}$ through $p$.
Since $\Pp!$ is precisely the number of such decompositions $\Qq_i$ that correspond to a given partition $\Pp$,
it suffices to show that  if $\Qq_i = \Qq_j$ for some pair of  sequences $(C_{k,i})_{k \geq 1}$ and $(C_{k,j})_{k \geq 1}$ with the same limiting curve $C := \ovl{C}_{\infty}$ then we must have $i = j$.

To prove this last point, suppose by contradiction that we have $i \neq j$, and let
$C': = C_{k,i}, C'': = C_{k,j}$ for some large $k$.
 We will show that this is impossible by counting singularities and applying the adjunction formula.\footnote
{Note the analogy with the proof of Claim B of  Proposition~\ref{prop:multneg}.}
According to the adjunction formula, the count $\delta(C)$ of singularities of $C$ is given by
$$ \delta(C) = \delta(A) = \tfrac{1}{2}(A^2 - c_1(A)) + 1,$$
while the singularity at $p$ gives a contribution to $\de(\Cc)$ of $\delta(\Pp)$ by Lemma~\ref{lem:dePp}.
Therefore we can assume that $C$ has an additional $\delta(A)-\delta(\Pp)$ double points away from $p$,
and, since $C',C''$ are close approximations of $C$, they therefore  have at least $2(\delta(A) - \delta(\Pp))$ intersection points away from $p$.
We now claim that the singularity at $p$ contributes at least $|\Pp| + 2\delta(\Pp)$ intersection points between $C'$ and $C''$ near $p$.
To see this, denote the local branches of $C'$ and $C''$ near $p$ by $B'_1,\dots,B'_{b}$ and $B''_1,\dots,B''_{b}$ respectively, where as usual we assume $\Pp = (m_1,\dots, m_b)$ with $m_1\ge \dots \ge m_b$. 
By construction for each $\be\in \{1,\dots,b\}$ the branches $B'_\be$ and $B''_\be$ go through the same subset of the tuple $\bp_k$ of size $m_\be$, so that $B''_{\be}\cdot B'_{\be} \geq m_\be.$
  Further, 
if $\al< \be$ and $k$ is sufficiently large, then $B'_{\al}$ is very close to a curve that has contact order $m_\al\ge m_\be$ with $D$ at $p$ and so $B''_\be$, which by construction meets $D$ in $m_\be$ points near $p$, must also meet $B'_\al$ in  $m_\be$ points near $p$.
Therefore, for $\al <\be$ we have  $B'_{\al}\cdot B''_{\be} \geq m_\be$ and also 
$B''_{\al}\cdot B'_{\be} \geq m_\be$.  Thus for sufficiently large $k$  there are at least 
$$
\sum_\be  m_\be + 2\sum_\be (\be-1)m_\be = |\Pp| + 2\de(\Pp)
$$
intersection points of $C'$ with $C''$ near $p$ as claimed.
Therefore we have
\begin{align*}
C' \cdot C'' &\geq 2(\delta(A) - \delta(\Pp)) + |\Pp| + 2\delta(\Pp)\\
& \geq A^2 - c_1(A) + 2 + c_1(A) - 1\\ &\geq A^2 + 1,
\end{align*}
which is impossible.
\end{proof}

Our next task is to understand what happens when we surround a multibranched tangency constraint by a skinny ellipsoid and stretch the neck.
As a preliminary step, the next lemma shows that every index zero  curve in a skinny ellipsoid satisfying a multibranched tangency constraint $\lll\T^\Pp p\rrr$ is a multiple cover of the  unique plane though $p$.  In particular,
since the space of $|\Pp|$-fold covers is connected,    
the structure of $\Pp$ is not really seen.

\begin{lemma}\label{lem:E_sk_curves}
Let $C$ be a $J$-holomorphic punctured sphere in $E_\sk$ which has positive ends asymptotic to $\eta_{k_1},\dots,\eta_{k_a}$ for some $a,k_1,\dots,k_a \geq 1$. Assume that $C$  has index zero and satisfies a multibranched tangency constraint $\lll \Pp \rrr$ for some partition $\Pp = (m_1,\dots,m_b)$.
Then we must have $a = 1$ and $|\Pp| = k_1$; further $C$ is a $k_1$-fold cover of the unique plane passing through $p$ with top end on $\eta_1$.
\end{lemma}

\begin{proof}
If $C$ is simple, we can estimate its top writhe from Equation~\ref{eq:w+} as $w^+_{\tau_\sp} \leq 0$ since
in this trivialization  $\lfloor m\theta\rfloor = 0$ for all $m<s$.  From the relative adjunction formula \eqref{eq:adjuc} we also have
\begin{align*}
w^+_{\tau_\sp} = \sum_{i=1}^a k_i + (a - 2) + 2\delta \geq \sum_{i=1}^a k_i + a - 2.
\end{align*}
Combining these inequalities, we must have $a = 1$ and $k_1 = 1$ as claimed. 
A similar calculation applied to the union $C$ of two distinct planes each with $a = k_1 = 1$ shows that in this case we would have $\de(C)<0$, which is impossible. 
Thus there is only one  possibility for $C$, namely the unique plane through $p$ with positive end  $\eta_1$.  

Now suppose that  $C$ is a $\ka$-fold cover of this unique plane $\ovl{C}$.  By \eqref{eq:Find}  and
 \eqref{eq:CZindn} 
we have 
$$
\ind(C) =  -1 + a + 2\sum_{i=1}^a k_i - 2 |\Pp|,
$$
where we subtracted $2 |\Pp|$ to take account of the constraint at $p$.  Since $|\Pp| \le \ka = \sum_{i=1}^a k_i$
and $\ind(C) = 0$ by hypothesis, we must have $a=1$ and $|\Pp| = \ka = k_1.$  
\end{proof}

\begin{lemma}\label{lem:T_to_E}
For any $\Pp \in \partitions_{c_1(A)-1}$ we have $N_{M,A}^T\lll \Pp \rrr \;\; \leq\;\; N_{M,A}^E \lll \Pp \rrr$.
\end{lemma}
\begin{proof}
Let $E_\sk \subset M$ denote a small skinny ellipsoid containing the point $p$.
We consider the effect of neck stretching along $\bdy E_\sk$.
Note that we can assume that the family of almost complex structures $J_t$, $t \in [0,1)$, realizing the neck stretching is constant near $p$, so that in particular $J_t$ is integrable near $p$ and $D$ is $J_t$-holomorphic. Moreover, we can assume that the limiting almost complex structure on $\wh{E_\sk}$ is generic, such that Lemma~\ref{lem:E_sk_curves} applies.
A priori, a limiting configuration as $t \rightarrow 1$ is a multilevel pseudoholomorphic building in the split symplectic cobordism $E_\sk \notccirc (M \setminus E_\sk)$, consisting of
\begin{enumerate}
\item a top level in $M \setminus E_\sk$
\item some number of levels in the symplectization $\R \times \bdy E_\sk$
\item a bottom level in the completion $\wh{E_\sk}$ of $E_\sk$.
\end{enumerate}

Let us formally glue together the curves in the symplectization levels to get a collection of (connected) matched components in $\R \times \bdy E_\sk$.
By equations \eqref{eq:indCi} and  \eqref{eq:energ}, it follows as in
 the proof of Proposition~\ref{prop:multneg} that all of these matched components have nonnegative index; and in fact this index  is strictly positive if there  is more than one positive end.
Similarly, by \eqref{eq:Aind1} all of the components in the top level have nonnegative index, and this is strictly positive unless they are simple.
Furthermore, observe that the constraint $\lll \Pp \rrr$ is distributed amongst the components in the bottom level. By Lemma~\ref{lem:E_sk_curves}, each of these components is a multiple cover of an embedded plane so that, if its top has multiplicity $m_i$, it   cannot satisfy a constraint $\lll \T^{\Pp_i}p \rrr$ with $|\Pp_i| > m_i$.  Hence it satisfies a constraint with $|\Pp_i| \le m_i$ and so has nonnegative index.  
Since the building has total index zero, it follows that all of these (matched) components must have index zero. In particular, each of the matched components in $\R \times \bdy E_\sk$ has exactly one positive end, each of the components in the top level is simple, and each component in the bottom level is a $k$-fold cover of the unique plane through $p$ with top asymptotic to $\eta_1$.   It follows that the top component $C_{top}$ is connected, and that each of the matched components in $E_{\sk}$ is a plane.

We now claim that $C_{top}$ has precisely $b =|\Pp|$ negative ends of multiplicities $m_1,\dots,m_b$.   To see this, note that  the limiting building can be approximated by a curve that may be cut just above the neck into a top portion  that approximates $C_{top}$ with $b'$ negative ends, and a bottom portion
with $b'$ connected components  each of which satisfies some part of the $\lll \T^\Pp p\rrr$ constraint.
More precisely, the partition $\Pp = (m_1,\dots,m_b)$ is divided into $b'$ subpartitions $(m'_{11},\dots, m'_{1k_1}),\dots, (m'_{b'1},\dots, m'_{b'k_{b'}})$ with  
$b = \sum_{i=1}^{b'} k_i$. 
 If $b'<b$ then at least one of these subpartitions, say $\Qq : = (m'_{11},\dots, m'_{1k_1}),$ contains $k_1>1$ elements.  
The corresponding bottom curve portion $C'$ satisfies a tangency constraint with at least two branches and so 
has  double points; more precisely, $\de(C') > 0$  by Lemma~\ref{lem:dePp}.  
Because $C'$ has index $0$, its top is an approximation to $\eta_{|\Qq|}$.
Thus the adjunction formula \eqref{eq:adjuc} for this curve portion (with $\tau = \tau_{\sp}$) gives
\begin{align*}
2\de(C') &= \chi(C') + Q_\tau(C') - c_\tau(C') + \wr_\tau(C')  \\
& = 1 + 0 - |\Qq| + (|\Qq|-1) = 0
\end{align*}
where we use \eqref{eq:Cspex} to calculate $c_\tau(C')$ and \eqref{eq:w+} 
to calculate  $\wr_\tau(C')$.
Thus, this scenario is impossible.  
Hence  $C_{top}$ has $b$ negative ends of multiplicities $m_1,\dots,m_b$.  

Thus, neck stretching associates to each curve in $\calM_{M,A}^{J_t,\simp}\lll \Pp \rrr$ with $t$ close to $1$ a curve in $\calM_{M \setminus E_\sk,A}^{J,\simp}\lll \Pp \rrr$.
Moreover, this association is injective, since if  two distinct $J_t$-holomorphic curves 
were very close to the same building their union would violate the relative adjunction formula
 exactly as in Claim B in Step 2 of the proof of Proposition~\ref{prop:multneg}. The key calculation here is \eqref{eq:claimB}.
\end{proof}

Here is the final ingredient needed to prove Theorem~\ref{thm:T=E}:

\begin{lemma}\label{lem:glue_E}
We have $\sum_{\Pp} \Pp!\, N_{M,A}^E \lll \Pp \rrr = N_{M,A}$, where the sum is over all partitions
$\Pp \in \partitions_{c_1(A) -1}$.
\end{lemma}
\begin{proof}
Put $\Pp = (m_1,\dots,m_b)$ 
so that each curve $C \in \calM_{M \setminus E_\sk,A}^{J,\simp}\lll \Pp \rrr$ has precisely $b$ negative ends on $\bdy E_\sk$.
At the $i$th end, consider a pseudoholomorphic building $B_i$ of the following form:\begin{itemize}
\item top level in the symplectization $\R \times \bdy E_\sk$ consisting of a $m_i$-fold cover of the trivial cylinder over $\eta_{m_i}$, with a single positive end asymptotic $\eta_{m_i}$ and $m_i$ negative ends each asymptotic to $\eta_1$,
\item bottom level consisting of $m_i$ planes $P^i_1,\dots,P^i_{m_i}$ in 
the symplectic completion  $\wh{E_\sk}$, 
each with positive end asymptotic to $\eta_1$, and collectively passing through $m_i$ of the constraints $p_1,\dots,p_k$.
\end{itemize}
More precisely, note that there is a multi-dimensional family of such buildings, due to the moveable branch points of the multiple covers in $\R \times \bdy E_\sk$.
After choosing a partition of the point constraints $p_1,\dots,p_k$ into $b$ parts of sizes $m_1,\dots,m_b$ (there are $\Pp!$ possible choices), the collection of planes $\{P^i_j\;:\; 1 \leq i \leq b,\, 1 \leq j \leq m_i\}$ is uniquely determined. 

For each curve $C$ as above, since the ECH partition conditions are satisfied at each negative end, we can glue the building $B_i$ to the $i$th negative end of $C$ via obstruction bundle gluing as in \cite[Thm~1.13]{HuT} to produce a closed $A$-curve in $M$ passing through the points $p_1,\dots,p_k$.
In this situation, it follows from \cite[Example~1.28]{HuT} that the gluing coefficient is $1$, so this gluing is unique.
More precisely, this gluing procedure depends on a gluing parameter $\rho$ controlling the lengths of the necks, with associated almost complex structure $J_\rho$.
By performing this gluing for each curve in $\calM_{M \setminus E_\sk,A}^{J,\simp}\lll \Pp \rrr$ and taking the gluing parameter $\rho$ sufficiently large, this produces $\sum_{\Pp} \Pp!\, N_{M,A}^E \lll \Pp\rrr$ curves in 
$\calM_{M,A}^{J_\rho,\simp}$, where the factor $ \Pp!$  comes from the different ways of assigning the point constraints.  Since, for sufficiently large gluing parameter $\rho$,  gluing sets up a bijective correspondence between the  moduli space of  unglued buildings and the moduli space of $J_\rho$-holomorphic glued curves we find that
$$
\sum_{\Pp} \Pp!\, N_{M,A}^E \lll \Pp \rrr = N_{M,A}
$$
as claimed.
\end{proof}

\begin{cor}\label{cor:T=Er}  (i)  For all partitions $\Pp_1,\dots \Pp_r$ we have
$$
N^T_{M,A}\lll \Pp_1,\dots,\Pp_r\rrr = N^E_{M,A}\lll \Pp_1,\dots,\Pp_r\rrr
$$

\NI (ii) For partitions $\Pp_1,\Pp_2,...,\Pp_r$ with $\Pp_1 = (1^{\times b})$, we have
$$ 
N_{M,A}\lll \Pp_1,\Pp_2,...,\Pp_r\rrr = N_{\bl^1M,A-b[\E]}\lll \Pp_2,...,\Pp_r\rrr.
$$
\end{cor} 
\begin{proof}
Claim (i)  is the generalization of Theorem~\ref{thm:T=E} to the case where  constraints are imposed at $r$ different points $p_1,\dots,p_r$, instead of just one point.  The proof for $r>1$ is essentially the same as that for $r=1$.  Further details are left to the reader. Claim (ii) follows by combining Theorem~\ref{thm:T=E} with Corollary~\ref{cor:count_pos2}.
\end{proof}

\begin{example}\label{ex:same}\rm  Consider the case $M = \C P^2$ and $A = d[L]$.  In this case, we abbreviate $(M,A)$ by the degree $d$ of $A$.  Thus $N_d$, for example, is the number of genus zero degree $d$ curves in $\C P^2$ through $3d-1$ points.
\MS

\NI
 (i) If $d=1,2$ then the only partition with $ N_d\lll{\Pp}\rrr\ne 0$   is $\Pp: = (3d-1)$, and we have  $N_d \lll (3d-1) \rrr=1$.
\MS

\NI (ii) If $d=3$, $N_3=12$ is the number of degree $3$ spheres  through $8$ generic points.   Since each such curve has one node,  the only possible partitions $\Pp$ have $\de(\Pp)\le 1$.  Hence $\Pp$ is  $(8)$ or $(7,1)$.    Note that both 
$N_3^E\lll (8)\rrr$ and $N_3^E\lll (7,1)\rrr$ are nonzero by Corollary~\ref{cor:count_pos} and Remark~\ref{rmk:count_pos}~(i).  Further,
Theorem~\ref{thm:T=E} implies that  
$$
N_3^E\lll (8)\rrr + 8 N_3^E\lll (7,1)\rrr= 12.
$$
  Hence we must have
 $$
N_3^E\lll (8)\rrr = 4, \quad N_3^E\lll (7,1)\rrr = 1.
$$   
Here is another noteworthy point. 
 If $C$ is a nodal cubic curve  that intersects $p$ with the maximal order $8$ then it satisfies the constraint
$\lll (8)\rrr$ unless the unique double point is at $\Pp$, in which case it satisfies the constraint $\lll (7,1)\rrr$.  However, if $C$ has a cusp that happens to go through $p$ then it is not clear how to interpret the constraint since some perturbations satisfy $\lll (7,1)\rrr$ while others satisfy $\lll (8)\rrr$.  As we point out in  Remark~\ref{rmk:tanpseud2}, this is not a problem since it does not happen for generic $J$.
\end{example}

\subsection{Combining constraints}\label{ss:comb}

Our goal for this subsection is to prove Theorem~\ref{thm:ppt_intro} from the introduction.
In the sequel, it will sometimes be convenient to represent partitions by Young diagrams (c.f. Example \ref{ex:Young_prod} below).
Let $\Y_k$ denote the set of Young diagrams with $k$ boxes. Thus we have 
$|\Y_1| = 1$, $|\Y_2| = 2$, $|\Y_3| = 3$, $|\Y_4| = 5$, $|\Y_5| = 7$, $|\Y_6| = 11$, and so on.
For $k \in \Z_{> 0}$, let $\Q\langle \Y_k\rangle$ denote the $|\Y_k|$-dimensional rational vector space spanned by $\Y_k$.
For $k,k' \in \Z_{>0}$, we define a map 
$$*: \Q\langle \Y_k\rangle \otimes \Q\langle \Y_{k'}\rangle \rightarrow \Q\langle \Y_{k+k'}\rangle$$
as follows.
Suppose that $y$ corresponds to the partition $(m_1,...,m_b)$ and $y'$ corresponds to the partition $(m_1',...,m_{b'}')$.
Choose subsets $S \subset \{1,...,b\}$ and $S' \subset \{1,...,b'\}$,
both of the same cardinality $|S| = |S'| = \ell$,
which we can write as $S = \{j_1,...,j_\ell\}$ for $j_1 < ... < j_\ell$ and
$S' = \{j_1',...,j_\ell'\}$ for $j_1' < ... < j_\ell'$.
The tuples $(m_{j_1},...,m_{j_\ell})$ and $(m_{j_1'}',...,m'_{j'_\ell})$ give the lengths of the corresponding rows of the Young diagrams $y$ and $y'$ respectively.
Given a permutation $\sigma \in \Sigma_\ell$, we now define $y *^\sigma_{S,S'} y'$ to be the 
 diagram with $b + b' - \ell$ rows, such that
\begin{itemize}
\item for $i = 1,...,\ell$, the $i$th row of $y *_{S,S'} y'$ has $m_{j_i} + m'_{\sigma(j_i)}$ blocks
\item the next $b - \ell$ rows are the remaining $b - \ell$ rows of $y$
\item the next $b' - \ell$ rows are the remaining $b' - \ell$ rows of $y'$,
\end{itemize}
after which we reorder the rows if necessary to get a valid Young diagram.
In other words, we combine $\ell$ rows of $y$ with $\ell$ rows of $y'$, and vertically stack the remaining rows.
Now define
\begin{align}\label{eq:*}
y * y': = \sum \left\{ y *^\sigma_{S,S'}y'\;\, \Big |\,  \begin{array}{ll}   \sigma \in \Sigma_\ell,  S \subset \{1,...,b\},\, S' \subset \{1,...,b'\},\\
  |S| = |S'| = \ell,\; \ell\ge 0 \end{array} \right\}
\end{align}
\NI
We denote the coefficients of $*$ with respect to the natural bases of Young diagrams by $\langle y * y, y''\rangle$
for $y \in \Y_k$, $y' \in \Y_{k'}$, and $y'' \in \Y_{k+k'}$,
or alternatively by $\langle \Pp * \Pp',\Pp''\rangle$ if $\Pp,\Pp',\Pp''$ are the partitions corresponding to $y,y',y''$ respectively.
Thus, 
\begin{align}\label{eq:*P}
 \Pp * \Pp'  = \sum_{\Pp'': |\Pp''| = |\Pp| + |\Pp'|}\;  \langle \Pp * \Pp',\Pp''\rangle\,\Pp''.
\end{align}

\begin{example}\label{ex:Young_prod}\rm
The partitions $(3,1,1)$ and $(2,2)$ correspond to Young diagrams
$y \in \Y_5$ and $y' \in \Y_4$ given by
\begin{align*}
{\Yvcentermath1 \tiny  y = \yng(3,1,1),\;\;\;\;\; y' = \yng(2,2)}\;,
\end{align*}
and we have 
\begin{align*}
{\Yvcentermath1 \tiny y * y' = \left(\;\yng(3,2,2,1,1)\;\right) + \left(2\;\yng(5,2,1,1) + 4\;\yng(3,3,2,1)\;\right) + \left(4\;\yng(5,3,1) + 2\;\yng(3,3,3)\;\right)}\;.
\end{align*}
Here the terms in parentheses correspond to the subsets having lengths $|S| = |S'| = 0,1,2$ respectively.\hfill$\er$
\end{example}

\begin{definition}\label{def:AutP} 
Given a partition $\Pp = (m_1,...,m_b)$, we denote by $\Aut(\Pp) \subset \Sigma_b$  the {\bf automorphism group} of $\Pp$, i.e. those permutations of the rows which leave the diagram fixed, and write $|\Aut(\Pp)|$ for its order.
\end{definition}

For example, for the partitions 
\begin{align*}
\Pp_1 = {\Yvcentermath1 \tiny  \yng(5)}\;,\;\;\;\;\;\Pp_2 = {\Yvcentermath1 \tiny  \yng(3,2,2,1,1)}\;,\;\;\;\;\; \Pp_3 = {\Yvcentermath1 \tiny  \yng(2,1,1,1)}\;,
\end{align*}
we have $|\Aut(\Pp_1)| = 1$, $|\Aut(\Pp_2)| = 4$, and $|\Aut(\Pp_3)| = 6$.
Our goal is to prove Theorem~\ref{thm:ppt_intro}, which we restate here for the convenience of the reader.

\begin{thm}\label{thm:ppt}
For any partitions $\Pp_1,...,\Pp_r$, we have
\begin{align*}
N_{M,A}\lll \Pp_1,\Pp_2,\Pp_3...,\Pp_r\rrr = \sum_{\Pp \in \partitions_{|\Pp_1| + |\Pp_2|}} \frac{\langle \Pp_1*\Pp_2,\Pp\rangle \;|\Aut(\Pp)|}{|\Aut(\Pp_1)| \; |\Aut(\Pp_2)|}   \;N_{M,A}\lll \Pp,\Pp_3,...,\Pp_r\rrr.
\end{align*}
\end{thm}

The idea for proving Theorem \ref{thm:ppt} is as follows.
We will base our discussion on the invariant $N_{M,A}^E\lll \Pp_1,...,\Pp_r\rrr$, although a similar argument could perhaps be given using $N_{M,A}^T \lll \Pp_1,...,\Pp_r\rrr$ as well.
We find it convenient to introduce slightly modified invariants by
\begin{align}\label{eq:whN}
\wh{N}_{M,A} \lll \Pp_1,...,\Pp_r\rrr := |\Aut(\Pp_1)|\;...\;|\Aut(\Pp_r)|\; N_{M,A}\lll \Pp_1,...,\Pp_r\rrr
\end{align}
(with an added $T$ or $E$ superscript if we wish to emphasize which version of the 
invariant we are using).
Note that $\wh{N}_{M,A}\lll \Pp_1,...,\Pp_r\rrr$ can be interpret geometrically in the same way as $N_{M,A}\lll \Pp_1,...,\Pp_r\rrr$ except that we order all marked points satisfying the same constraint, or equivalently all punctures with the same  asymptotic Reeb orbit.\footnote
{Since we always write $\Pp = (m_1,\dots, m_b)$  with $m_1\ge m_2\ge \dots \ge m_b$ it suffices to order only the repeated entries in $\Pp$.}
 In particular, if $\Pp_i = (m_1^i,\dots,m_{b_i}^i)$,  the number
 $$
 \wh{N}_{M,A}^E\lll \Pp_1,...,\Pp_r\rrr
 $$
 is the count of $A$-curves in $M\less (\cup_{i=1}^r E_\sk^i)$ with ordered ends  $$
 \bigl(\eta^i_{m_j^i}\bigr)_{1\le j\le m_i}\quad \mbox{ on }\;\; \p E_\sk^i,
 $$
  where the ordering only affects the ends on the same orbit $\eta^i_{j}$ on $\p E_\sk^i$.

To keep the notation simple, we will assume that $r=2$ and prove the formula
\begin{align}\label{eq:NHat}
\wh{N}^E_{M,A}\lll \Pp_1,\Pp_2\rrr = \sum_{\Pp \in \partitions_{|\Pp_1| + |\Pp_2|}} \langle \Pp_1*\Pp_2,\Pp\rangle \cdot  \wh{N}^E_{M,A}\lll \Pp\rrr,
\end{align}
the general argument being essentially the same.

We can assume without loss of generality that the union of the small skinny ellipsoids $E_\sk^{1},E_{\sk}^{2}$ is contained in a third small skinny ellipsoid $E_{\sk}^{1,2}$. 
We perform a neck-stretching along $\bdy E_{\sk}^{1,2}$ and analyze how the curves corresponding to $\wh{N}_{M,A}^E\lll \Pp_1,\Pp_2\rrr$ degenerate.
Roughly, this has the effect of converting negative ends in $\bdy E_{\sk}^{1}$ and $\bdy E_{\sk}^{2}$ into negative ends in $\bdy E_{\sk}^{1,2}$.
A priori, the degenerations could be arbitrary pseudoholomorphic buildings in $\left(E_{\sk}^{1,2}\setminus (E_{\sk}^{1}\cup E_{\sk}^{2})\right) \notccirc \left(M \setminus E_{\sk}^{1,2}\right)$
consisting of
\begin{enumerate}[label=(\alph*)]
\item
a top level in $M \setminus E_{\sk}^{1,2}$;
\item 
some number of levels in the symplectization $\R \times \bdy E_{\sk}^{1,2}$;
\item
a level in $E_{\sk}^{1,2} \setminus (E_{\sk}^{1} \cup E_{\sk}^{2})$;
\item some number of levels in the symplectization $\R \times (\bdy E_{\sk}^{1} \cup \bdy E_{\sk}^{2})$.
\end{enumerate}
It turns out that we can greatly narrow down the possibilities via action and index considerations, together with the relative adjunction formula and techniques from ECH 
that place powerful (and perhaps unexpected) restrictions on the structure of the limit;
in particular, by (ii), (iii) below  the only branching of constraints occurs in the neck region rather than in the cobordism level.
The next lemma describes the possible limiting buildings in more detail.

\begin{lemma}\label{lem:ppt_breakings}
We can arrange that the limiting building has the following properties:
\begin{itemize}
\item[{\rm(i)}] the top level (a) is a connected somewhere injective curve of index $0$;
\item[{\rm(ii)}]  each component of a symplectization level (b) is either a trivial cylinder or an index zero branched cover of a trivial cylinder which is a pair of pants with one positive end and two negative ends;
\item[{\rm(iii)}]  each component of (c) is a cylinder; moreover there is a unique such cylinder between any  pair of ends of the same index; 
\item[{\rm(iv)}]  there are no symplectization levels (d).
\end{itemize}
\end{lemma}
\begin{proof}  We prove (i).   As in  Step 1 of Proposition~\ref{prop:multneg}, we can assume that each component of the top level  has nonnegative index and is somewhere injective.
Since the limiting building is connected and has one negative end, to see that its top level   is connected it suffices to show that every component in the lower levels  has exactly one positive end.

The argument  in  Step 1 of Proposition~\ref{prop:multneg} also shows 
that the index of each component of type (b) or (d) is nonnegative, and is strictly positive unless it has one positive end.   
We next prove a similar statement for the components of type (c)
 by an argument based on considerations of energy.
Consider a curve
$C$ in the cobordism $E_{\sk}^{1,2} \setminus (E_{\sk}^{1} \cup E_{\sk}^{2})$ with  positive ends $\eta_{m_1},...,\eta_{m_k}$ and negative ends $\eta_{m_1'},...,\eta_{m_l'}$.  Then, as in \eqref{eq:indCi} we have
\begin{align*}
\ind(C) = 2k - 2 + 2\sum_{i=1}^k m_i - 2\sum_{j=1}^{l} m_j'.
\end{align*}
Further, we may choose $s \gg s_1 + s_2$ so large that the ellipsoids $E(\eps',\eps' s_1)$ and $E(\eps',\eps' s_2)$ fit one after the other into $E(\eps,\eps s)$ with $\eps$ arbitrarily close to $\eps'$, and we may suppose the energy of the orbits $\eta_1^1, \eta_1^2$  to be so close to that of $\eta_1^{1,2}$ that positivity of energy implies that 
\begin{align*} 
\sum_{i=1}^k m_i - \sum_{j=1}^{l}m_j' \geq 0.
\end{align*} 
Hence we have $\ind(C) \geq 0$, with equality only if $k = 1$.
Since the whole building has index zero, all components must have index zero and  the components of type (b), (c) or (d) have a single positive end.

\MS

We next prove (iv).  We do this by showing that
all components of type (d) must also have one negative end.  If so, 
 they must all be trivial cylinders by Lemma~\ref{lem:mc2}, so that they do not occur.
 
  To see this, we use the fact that our building is a limit of somewhere injective curves $C_k$ in 
  (the completion of)
  $X: = M\less (E^1_{\sk} \cup E^2_{\sk})$.  Although we are taking a limit as the neck around $\p E^{1,2}_{\sk}$ is stretched, this neck is not  relevant to the bottom end of the limiting building; instead, we are concerned with understanding  the limit of a sequence of curves in one of the fixed negative ends  $(-\infty,0] \times \p E^i_{\sk} $ of $X$.  
  What governs this situation are results concerning SFT compactness for curves in $X$, and in particular the fact that in a converging sequence,  the curves $C_k$ break into separate \lq packets' $P_{k,\ell}: = C_k\cap \bigl(
  [t_{2\ell}, t_{2\ell-1}] \times \p E^i_{\sk}\bigr)$ of finite energy joined by increasingly long cylindrical \lq necks' $P'_{k,\ell}: = C_k\cap \bigl([t_{2\ell+1}, t_{2\ell}]\times \p E^i_{\sk}\bigr)$ that have very small energy and are such that each of its connected components is exponentially close to a braid around the limiting Reeb orbit.  
If  there is a 
matched component $C$ of the limiting building  in $\R\times \p E_{\sk}^i$  with more than one negative end,
then one of these finite energy packets, say $P_{k, \ell}$,  must have a component  with more than one negative end while its top end is connected as before.   Now, for sufficiently large $k$, look at the appropriate component $C'$ of 
$C_k\cap \bigl([t_{2\ell}, t_{2\ell-2}]\times \p E^i_{\sk}\bigr)$, which is the union of a  finite energy bottom piece with an almost constant top piece.    This is a curve portion both of whose ends approximate {\it from above} some multiple of the  orbit $\eta^i_1$.  Thus we are precisely in the situation considered in the second half of the proof of
Step 1 in Proposition~\ref{prop:multneg}, and the calculations in \eqref{eq:w+M}, \eqref{eq:w-M}, and \eqref{eq:wcontrad}  show that 
this situation does not occur. This completes the proof of (iv).
\MS

Now consider a matched component  $C_j$ that is a union of curves of types (b) and (c).  
Since as we saw above, there are no components of type (d), the bottom ends of $C_j$ form a subset 
of the bottom ends  of the original moduli space of curves. 
We  claim that $C_j$ has at most one negative end on each $\bdy E_{\sk}^{i}$ for $i=1,2$.
To justify this, suppose by contradiction that $C_j$ has positive end $\eta_k$ on $\p E_{\sk}^{1,2}$ and negative ends $\eta_{k_1},...,\eta_{k_b}$, where the first $r\ge 2$ ends lie on $\bdy E_{\sk}^{1}$ and the others lie on $\bdy E_{\sk}^{2}$.
Since $C_j$ has index zero, we must have $k = \sum_{i=1}^b k_i$.  
Arguing as in the proof of Proposition~\ref{prop:multneg},
 we can approximate the total building 
by a somewhere injective curve in $M \setminus (E_{\sk}^{1} \cup E_{\sk}^{2})$, and by cutting off at a suitable level just above the neck around $\p E^{1,2}_{\sk}$, we  may find a portion of that curve that approximates $C_j$. 
This curve portion has well-defined positive and negative writhes $\wr_\tau^+,\wr_\tau^-$ (where we take $\tau=\tau_{\sp}$), and since the positive end corresponds to a negative end of a somewhere injective curve in $M \setminus E_{\sk}^{1,2}$, by \eqref{eq:wneg} we have
$\wr_\tau^+ = k-1$.
Since the negative ends also correspond to negative ends of a somewhere injective curve in 
$M \setminus (E_{\sk}^{1}\cup E_{\sk}^{2})$, we have by \eqref{eq:w-} the lower bound 
\begin{align*}
\wr_\tau^- &\geq \sum_{i,j=1}^r \min\left(k_i,k_j\right) + \sum_{i,j=r+1}^b \min\left(k_i,k_j\right) - b.
\end{align*}
For this curve portion we also have $c_{\tau_{\sp}} = k-\sum_{i=1}^b k_i =0$ (by \eqref{eq:Cspex}), $\chi = 1-b$, and $Q = 0$;
so using the relative adjunction formula \eqref{eq:adjuc} we have 
\begin{align*}
 b-1 + 2\delta = \wr_\tau^+ - \wr_\tau^-  \leq  (k-1)+b - \sum_{i,j=1}^r \min\left(k_i,k_j\right) -\sum_{i,j=r+1}^b \min\left(k_i,k_j\right).
\end{align*}
Since $\de\ge 0$, for this to hold we need
\begin{align*}
k   = \sum_{i=1}^b k_i \ge  \sum_{i,j=1}^r \min\left(k_i,k_j\right) +\sum_{i,j=r+1}^b \min\left(k_i,k_j\right),
\end{align*}
which is false if either $r>1$ or $b-r>1$.
\MS

It follows that the matched component $C_j$ is either a cylinder or a pair of pants.  We next claim that  no  component of type (c) 
(i.e. one entirely contained in the cobordism level)
can be a  pair of pants.  This implies  (ii) and the first claim in (iii).  
As usual this follows by a writhe calculation, but this time one for a curve rather than a curve portion.  
Suppose by contradiction that there exists such a curve with positive end $\eta_k^{1,2}$ and negative ends $\eta^1_{k_1},\eta^2_{k_2}$ on $\bdy E_{\sk}^{1}, \bdy E_{\sk}^{2}$ respectively.
By index considerations we can assume that $k_1 + k_2 = k$,
and, after possibly passing to the underlying simple curve, we can also assume that this curve is somewhere injective.
Such a curve has $c_1 = 0$, $\chi = -1$, $Q = 0$, and the writhe bounds from \eqref{eq:w+}ff give 
$\wr_\tau^+ \leq 0$ and $\wr_\tau^- \geq k_1+k_2-2$.
The relative adjunction formula then gives
\begin{align*}
\wr_\tau = 1 + 2\delta \leq -k_1 -k_2 + 2,
\end{align*}
which is a contradiction.

Finally, we observe that for each $i=1,2$ and $k\ge 1$ there is a unique cylinder from $\eta_k^{1,2}$ to $\eta_k^i$.  This follows by Claim A of Proposition~\ref{prop:multneg}, since the proof given there is not affected by the presence of the second skinny ellipsoid.
This completes the proof of (iii) and hence of the lemma.
\end{proof}

Observe that each of the top components in Lemma \ref{lem:ppt_breakings} represents a summand of 
$\wh{N}_{M,A}^E\lll \Pp \rrr$ for $\Pp \in \partitions_{|\Pp_1|+|\Pp_2|}$ with $\langle \Pp_1 * \Pp_2,\Pp\rangle \neq 0$.
In fact, if it were true that each of the matched components $C_i$ existed uniquely as an honest pseudoholomorphic curve in $E_{\sk}^{1,2} \setminus (E_{\sk}^{1} \cup E_{\sk}^{2})$, Theorem \ref{thm:ppt} would follow immediately by considering all possible ways of gluing a curve  $C_{top}$ in $M \setminus E_{\sk}^{1,2}$ to a configuration of curves in $E_{\sk}^{1,2} \setminus (E_{\sk}^{1}\cup E_{\sk}^{2})$.
Unfortunately, many of the matched components in $E_{\sk}^{1,2} \setminus (E_{\sk}^{1}\cup E_{\sk}^{2})$ are necessarily represented by two-level buildings, since as we saw in Lemma~\ref{lem:ppt_breakings} there are no pair of pants  of type (c), i.e. in  the cobordism $E_{\sk}^{1,2} \setminus (E_{\sk}^{1} \cup E_{\sk}^{2})$ itself.
On the other hand, such a curve {\em can} be represented by a building, with top level a pair of pants $\Si$ in $\R \times \bdy E_{\sk}^{1,2}$ which is an index $0$ branched cover of a trivial cylinder, and bottom level two cylinders in $E_{\sk}^{1,2}\setminus (E_{\sk}^{1}\cup E_{\sk}^{2})$.
Note that the index $0$ branched cover $\Si$  in $\R \times \bdy E_{\sk}^{1,2}$ exists in a two-dimensional moduli space $\Mm_R$ (essentially because the branched point can be moved around), and hence it appears with higher-than-expected dimension.
This means we are not in a position to prove Theorem \ref{thm:ppt} by a standard gluing argument for regular moduli spaces.

Fortunately, we can use the method of  obstruction bundle gluing as described by Hutching--Taubes in  \cite[\S1.8]{HuT}.
As they explain, for each sufficiently large gluing parameter $\rho$, there is a bundle $\Oo$ over the moduli space $\Mm_\rho$ of branched covers and a section $\s: \Mm_\rho\to \Oo$ 
such that the preglued curve formed by attaching the branched cover $\Si$ to an appropriate end of a (regular, somewhere injective)  top curve $C_{top}$ and two bottom cylinders can be perturbed to an honest holomorphic curve exactly when $\s(\Si) = 0$.  The number of zeros of $\s$ is called the gluing coefficient of this problem. 
If $C_{top}$ has many ends, one has to perform this gluing operation at each of its ends.  The resulting gluing coefficient is the product of the coefficients at each end, i.e. the problem is local in the target, and also does not depend on the particular choice of $C_{top}$.    (Note that the cylinders are unique by Lemma~\ref{lem:ppt_breakings}.)
   Our current situation, though close to that considered in \cite[Thm~1.13]{HuT}, is not precisely the same because the bottom cylinders are multiply covered in general.\footnote
   {
   However, the top of each cylinder on $\eta_{m_i}$ does have a maximal partition, as is required by the definition of a gluing pair in \cite[Def~1.8]{HuT}.}
      Hence we cannot use their calculation of the gluing coefficient.  Nevertheless we can conclude that this coefficient does not depend on the choice of $C_{top}$.  
Thus, we have the following result.

\begin{lemma}\label{lem:weak_ppt}
For any partitions $\Pp_1,...,\Pp_r$, we have
\begin{align*}
\wh{N}^E_{M,A}\lll \Pp_1,\Pp_2,\Pp_3...,\Pp_r\rrr = \sum_{\Pp \in \partitions_{|\Pp_1| + |\Pp_2|}} C_{\Pp_1,\Pp_2}^{\Pp} \cdot  \wh{N}^E_{M,A}\lll \Pp,\Pp_3,...,\Pp_r\rrr
\end{align*}
for some coefficients $C_{\Pp_1,\Pp_2}^{\Pp}$ which depend only on the partitions $\Pp_1,\Pp_2,\Pp$.
Moreover $C_{\Pp_1,\Pp_2}^{\Pp}\ne 0$ only if $\langle \Pp_1*\Pp_2,\Pp\rangle\ne 0$. 
\end{lemma}

The upshot is that it suffices to show that each of the coefficients $C_{\Pp_1,\Pp_2}^{\Pp}$ is equal to $\langle \Pp_1*\Pp_2,\Pp\rangle$.
Note that, by Lemma \ref{lem:ppt_breakings}, all of the relevant gluings are either standard gluings along cylindrical ends or obstruction bundle gluing involving an index $0$ branched cover which is a pair of pants.
The following lemma is the main remaining ingredient for the proof of Theorem \ref{thm:ppt}.

\begin{lemma}\label{lem:two_ends}
For any $m,m' \geq 1$, we have
\begin{align*}
\wh{N}^E_{M,A}\lll (m),(m')\rrr = \wh{N}^E_{M,A}\lll (m,m')\rrr + \wh{N}^E_{M,A}\lll (m+m')\rrr.
\end{align*}
\end{lemma}
\begin{proof}
According to Lemma \ref{lem:weak_ppt}, we have
\begin{align}\label{eq:weak_ppt_pair}
\wh{N}_{M,A}^E \lll (m),(m')\rrr = C_{(m),(m')}^{(m,m')}\wh{N}_{M,A}^E\lll (m,m')\rrr + C_{(m),(m')}^{(m+m')} \wh{N}^E_{M,A}\lll (m+m')\rrr
\end{align}
with the coefficients $C_{(m),(m')}^{(m,m')}$ and $C_{(m),(m')}^{(m+m')}$ a priori unknown.
We saw in Lemma~\ref{lem:ppt_breakings}
that
for $i=1,2$, there is a unique cylinder in $E_{\sk}^{1,2} \setminus (E_{\sk}^{1} \cup E_{\sk}^{2})$ from
 $\eta^{1,2}_m$ on $\bdy E_{\sk}^{1,2}$  to $\eta^i_m$ on $\bdy E_{\sk}^{i}$.
It follows that 
for each matching of ends  of multiplicities $m, m'$  of   a somewhere injective curve $C_{top}$ in $M\setminus \p E_{\sk}^{1,2}$ with ends of these multiplicities on the two ellipsoids  
$\p E_{\sk}^i, i={1,2}$
there is a unique\footnote
{
Note that because the cylinder is a $m$-fold covering  it has only one rather than $m$ different  pregluings to a top curve.
One could also argue that this coefficient is one by using a model example as in the next paragraph.}
 way to glue in this pair of regular cylinders to 
 $C_{top}$ 
  in order to produce a curve in $M \setminus (E_{\sk}^{1} \cup E_{\sk}^{2})$ with negative ends
  on $\bdy E_{\sk}^{1}\sqcup \bdy E_{\sk}^{2}$.  If $m \ne m'$ there is only one way to make this matching once the appropriate ends of $C_{top}$ are chosen.
  However, if $m=m'$ there are two ways to make this matching for each pair of ends of $C_{top}$.  But if we order the ends of $C_{top}$ that have the same multiplicity, then again we can make a unique matching since the two ellipsoids $E_{\sk}^1, E_{\sk}^2$ are ordered.\footnote
  {Equivalently, the partitions $\Pp_1 = (m)$ and $\Pp_2 = (m')$ are ordered.}   
 Hence there is a bijective correspondence between the curves $C_{top}$ (with ordered ends) counted by $\wh{N}^E_{M,A}\lll (m,m')\rrr$ and the curves  counted by
 $\wh{N}^E_{M,A}\lll (m),(m')\rrr$,.
Thus the coefficient $C_{(m),(m')}^{(m,m')}$ is equal to $1$.
\MS

As for the coefficient $C_{(m),(m')}^{(m+m')}$, we can determine it by using a model example.
Pick $d$ such that $2d \geq m + m'$.
Let $M$ denote the blowup $\bl^{2d+1-m-m'}\CP^2$, and let $A$ denote the homology class $$
A: = d [L] - (d-1)[\E_1] - [\E_2] - ... - [\E_{2d+1-m-m'}] \in H_2(M;\Z).
$$
We claim that there is a unique embedded curve in $M$ in the homology class $A$ passing through pairwise distinct points $p_1,...,p_m,p_1',...,p_{m'}' \in M$.
 i.e. we have $$
N_{M,A}\lll p_1,...,p_m,p_1',...,p_{m'}' \rrr = 1.
$$
To see this, consider  the blowup $M' = \bl^{2d+1}\CP^2$ and the homology class 
$$
A' = d [L] - (d-1)[\E_1] - [\E_2] - ... - [\E_{2d+1}] \in H_2(M';\Z).
$$
Note that $A'$ is the class of an exceptional sphere,\footnote
{
One can see this either by directly constructing a degree $d$ curve of genus $0$ with a multiple point  of order $d-1$ from $d$ suitably intersecting lines, and then blowing up appropriately, or by reducing this class to $[\Ee_1]$ by a  sequence of Cremona transformations; cf \cite[Ch~13.4]{JHOL},}
which means that the Gromov--Witten invariant $\gw_{M',A'}$ is equal to $1$.
The claim then follows by Corollary~\ref{cor:count_pos2}  by trading the last $m+m'$ blowup constraints of $A'$ for point constraints.

Next, we choose  small skinny ellipsoids with $E_{\sk}^{1} \cup E_{\sk}^{2} \subset E_{\sk}^{1,2}$ which are disjoint from the exceptional divisors $\E_1,...,\E_{2d+1-m-m'}$ in $M: = \bl^{2d+1-m-m'}\CP^2$.  Then 
put the points $p_1,...,p_m \in E_{\sk}^{1}$ and $p_1',...,p_{m'}' \in E_{\sk}^{2}$, and 
then stretch the neck around the boundary $\p E_{\sk}^{1} \cup \p E_{\sk}^{2}$.  
We claim that the unique $A$-curve converges to a
rational curve $C_{top}$  in $M \setminus (E_{\sk}^{1} \cup E_{\sk}^{2}) $ in the homology class $A$ with one negative end $\eta_m$ on $\bdy E_{\sk}^{1}$ and one negative end $\eta_{m'}$ on $\bdy E_{\sk}^{2}$.
The easiest way to see this is to note that the $A$-curve, because it is a embedded sphere of Fredholm index zero, has ECH index zero, so that when the neck is stretched it must decompose into constituent curves that satisfy the ECH partition conditions, which means in this case that there is one negative end on each skinny ellipsoid.   Relevant references are \cite{Hlect} and \cite[\S2.2]{Ghost}.
Alternatively, one can deduce this by arguing from the adjunction formula, estimating the writhe  as in the proof of 
Lemma~\ref{lem:ppt_breakings}.

To see that there is a unique  curve $C_{top}$ in the top level with the specified ends, suppose by contradiction that there were two such curves.
Then by \eqref{eq:c1sp} their union would satisfy
\begin{align*}
&c_{\tau_{\sp}} = 2(c_1(A) - (m_1+m_2)) =  2 \\
&\chi = 0,   \qquad Q_{\tau_{\sp}} = 4A\cdot A = -4 + 4m + 4m',
\end{align*}
and so the relative adjunction formula \eqref{eq:adjuc} would give 
$$ 
\wr_{\tau_{\sp}}^- = -6 + 4m + 4m'  - 2\delta.
$$
On the other hand, the lower bound for $\wr_\tau^-$ from  \eqref{eq:w-} gives
$\wr_{\tau_{\sp}}^- \geq 4m + 4m' - 4$, which is a contradiction.
\MS

Finally, we consider how this curve $C_{top}$  degenerates as we stretch the neck along $\bdy E_{\sk}^{1,2}$. 
According to Lemma \ref{lem:ppt_breakings}, the limiting buildings have one of the following types:
\begin{itemize}\item 
they  have two levels, 
where the top curve in
$M \setminus E_{\sk}^{1,2}$ has two negative ends that are each joined to a cylinder in $E_{\sk}^{1,2}\setminus (E_\sk^1 \cup E_\sk^2)$, or 
\item they have three levels, where the top curve has 
negative end on $\eta^{1,2}_{m+m'}$, there is a pair of pants in $\R \times \bdy E_{\sk}^{1,2}$ with negative ends on $\eta^{1,2}_m, \eta^{1,2}_{m'}, $ and then two cylinders in $E_{\sk}^{1,2} \setminus (E_{\sk}^1 \cup E_{\sk}^2)$.
\end{itemize}
We claim that the first situation cannot occur, since there are no somewhere injective curves in $M \setminus E_\sk^{1,2}$ in class $A$  with two negative ends $\eta_m$ and $\eta_{m'}$ on the {\it same} ellipsoid $\p E^{1,2}_{\sk}$.
Indeed, for such a cylinder the relative adjunction formula would give
$$
\wr_{\tau_{\sp}}^- = -2 + m + m' - 2\delta,
$$
while the writhe bound from \eqref{eq:w-} would give 
$$\wr_{\tau_{\sp}}^- \geq m + m' + 2\min(m,m') -2,
$$
 which is impossible.\footnote
 {
 Again, one can bypass this argument by saying that the top curve, because it has ECH index $0$  has to satisfy the ECH partition conditions, and hence has only one negative end.} 

It follows that the limiting building must have three levels.
As before, we can show that this top level has a unique representative.
Indeed, if two distinct such planes existed, the relative adjunction formula applied to their union would give 
$$
\wr_{\tau_{\sp}}^- = -4 + 4m + 4m' - 2\delta,
$$
 while the writhe bound would give 
 $$
 \wr_{\tau_{\sp}}^- \geq 4m + 4m' - 2,
 $$ 
 which is again impossible.
Thus the formula \eqref{eq:weak_ppt_pair}
reduces in this situation to 
$$ 
1 = C_{(m),(m')}^{(m,m')} \cdot 0 + C_{(m),(m')}^{(m+m')} \cdot 1,
$$
from which it follows that $C_{(m),(m')}^{(m+m')} = 1$.
\end{proof}

\begin{proof}[Proof of Theorem \ref{thm:ppt}]  We saw above that it suffices to prove the identity
\begin{align}\label{eq:NHat1}
\wh{N}^E_{M,A}\lll \Pp_1,\Pp_2\rrr = \sum_{\Pp \in \partitions_{|\Pp_1| + |\Pp_2|}} \langle \Pp_1*\Pp_2,\Pp\rangle \cdot  \wh{N}^E_{M,A}\lll \Pp\rrr.
\end{align}
Notice that this involves the invariants $\wh{N}^E_{M,A}$  that count curves in which the negative ends on the same ellipsoid are ordered. 
By Lemma \ref{lem:weak_ppt}, it suffices to show that $$
C_{\Pp_1,\Pp_2}^\Pp = \langle \Pp_1 * \Pp_2,\Pp\rangle
$$
for all partitions $\Pp_1,\Pp_2,\Pp$.
Note that it follows from Lemma \ref{lem:ppt_breakings} that we have $C_{\Pp_1,\Pp_2}^\Pp = 0$ whenever $ \langle \Pp_1 * \Pp_2,\Pp\rangle = 0$.

In general, for each triple $\Pp, \Pp_1,\Pp_2$ the number $\langle \Pp_1 * \Pp_2,\Pp\rangle$ is precisely the number of relevant configurations of matched components in $E_{\sk}^{1,2} \setminus (E_{\sk}^{1} \cup E_{\sk}^{2})$ with top end corresponding to the partition $\Pp$ and bottom ends in $E_{\sk}^{i}$ corresponding to the partition $\Pp_i$ for $i = 1,2$. To see this, note
that each such matched component is either a cylinder corresponding to a single level building or a pair of pants represented by a two level building.  Thus each configuration of matched components pairs each entry $m_i$  of $\Pp$ either with a pair $(m, m')\in \Pp_1\times \Pp_2$ with $m+m' = m_i$ or with a single entry $m_i$ in either $\Pp_1$ or $\Pp_2$.  
Each such configuration of cylinders and pairs of pants uniquely determines the subsets $S,S'$ of ends on $\p E_\sk^1$ and $\p E_\sk^2$ that are summed (since they are the \lq feet' of the pairs of pants) and also the permutation $\si$ that determines how the subsets $S, S'$  are paired up.  Further, because   
the configuration determines an injection from the set of ends on $\p E_{sk}^i$ to those on $\p E_\sk^{1,2}$,
an ordering of the entries   in  $\Pp$ uniquely determines an ordering of the entries in $\Pp_i$   for $1,2$.  
Thus the sum on the right hand side of \eqref{eq:NHat1} is the number of buildings that when glued will yield the curves counted on the LHS. 

 It remains to note that each gluing coefficient is $1$.
Following Lemma \ref{lem:two_ends}, in the cylinder case we have a unique representative and are in the situation of standard gluing, while in the case of a pair of pants we have obstruction bundle gluing with gluing coefficient also equal to $1$.
\end{proof}

\section{A recursive algorithm}\label{sec:recur}

In this final section we discuss the recursive algorithm from the introduction. 
In \S\ref{ss:recur1} we  discuss some basic properties of the recursion and present some example computations. Subsequently, in \S\ref{ss:recur2} we describe the algorithm in detail.

\subsection{Overview and computations}\label{ss:recur1}
 Let $M$ be a closed symplectic $4$-manifold and let $A \in H_2(M;\Z)$ be a homology class.
By the results of \S\ref{ss:tan_multi}, we have well-defined integer invariants
$$
N_{M,A} \lll (\T^{m_1^1-1}p_1,...,\T^{m^1_{b_1}-1}p_1),...,(\T^{m_1^r-1}p_r,...,\T^{m^r_{b_r}-1}p_r) \rrr 
$$
for all $r \geq 0$, $b_1,...,b_r \geq 1$ and $m^i_j \geq 1$ with $1 \leq j \leq b_i$.
For each $i$, we will assume that we have $m_1^i \geq ... \geq m_{b_i}^i$, and we write the corresponding partition as $\Pp_i = (m^i_1,\dots,m^i_{b_i})$,
so that the above invariant can also be written more succinctly as
$$
N_{M,A}\lll \Pp_1,...,\Pp_r\rrr \quad\mbox{ or } \quad N_{M,A}\lll (m_1^1,...,m_{b_1}^1),\dots,(m_1^r,...,m_{b_r}^r)\rrr 
$$
Recall that for basic index reasons this invariant can only be nontrivial when we have
$$ 
-1 + c_1(A) - \sum\limits_{i=1}^r\sum\limits_{j=1}^{b_i}m^i_j = 0.
$$

The following theorem is proved in \ref{ss:recur2}.  

\begin{thm}\label{thm:recur}
There is explicit recursive algorithm which describes the invariant $N_{M,A}\lll \Pp_1,\dots,\Pp_r\rrr$ as a linear combination of invariants of the form $N_{M,A}\lll \Pp_1',\dots,\Pp_{r'}'\rrr$, where 
each of the partitions $\Pp_i'$ is of the type $(1,\dots,1)$.
 \end{thm}

\begin{cor}\label{cor:recur}
The invariant $N_{M,A}\lll {\Pp_1},\dots,\Pp_r\rrr$ can be written as an explicit linear combination
of blowup Gromov--Witten invariants of the form
$\gw_{\bl^a M, A - n_1[\E_1]-...-n_a[\E_a]}$
for $a, n_1,...,n_a> 0$.
\end{cor}
\begin{proof}    This follows from Theorem~\ref{thm:recur} by repeatedly applying Corollary~\ref{cor:T=Er}~(ii).
\end{proof}

In the case $M = \CP^2$, G\"ottsche--Pandharipande \cite[Theorem 3.6]{GP} used associativity of the quantum cup product to give an explicit recursive algorithm for the rational blowup Gromov--Witten invariants $\gw_{\bl^r M,A - \sum_{i=1}^s n_i[\E_i]}$.
More generally, Gathmann \cite{Gath_blowup} also gave a recursive algorithm which computes the rational Gromov--Witten invariants of $\bl^r M$ for any convex projective variety $M$ in terms of the rational Gromov--Witten invariants of $M$. 
One can combine this with our recursion from Theorem \ref{thm:recur} to completely compute the invariants 
$N_{\CP^2,d [L]}\lll \Pp_1,\dots,\Pp_r\rrr$.

\begin{example}[{\bf Computations for $\CP^2$}]\rm
Table \ref{table:CP2_1} shows the invariants $N_d\lll \T^{3d-2}p\rrr$ for small values of $d$.
For purposes of comparison, we also include the numbers $N_d\lll p_1,...,p_{3d-1}\rrr$ computed by Kontsevich's recursion formula, as well as the analogous descendent invariant 
$\gw_{\CP^2,d[L]}\lll \psi^{3d-2}\rrr$ with a full gravitational descendant constraint at a point.
Table \ref{table:CP2_2} shows the nonzero invariants $N_d\lll \T^{m_1-1}p,...,\T^{m_b-1}p \rrr$ for small values of $d$.
Note for most partitions $(m_1,...,m_b)$ these invariants are zero by the adjunction inequality.  For further comment on these numbers see Remark~\ref{rmk:dualpart}.
All of these computations were made using a computer program.\footnote{
A Python implementation of the recursive algorithm is available by request or on the website of KS.}

\begin{table}[h]
\begin{tabular}{|l|l|l|l|}
\hline
d & $N_{d}\lll \T^{3d-2}p\rrr$ & $N_{d}\lll p_1,...,p_{3d-1}\rrr$ & (3d-2)!$\gw_{\CP^2,d[L]}\lll \psi^{3d-2}p\rrr$ \\ \hline
1 & 1 & 1 & 1 \\ \hline
2 & 1 & 1 & 3 \\ \hline
3 & 4 & 12 & 70/3 \\ \hline
4 & 26 & 620 & 525/2\\ \hline
5 & 217 & 87304 & 18018/5 \\ \hline
6 & 2110 & 26312976 & 56056 \\ \hline
7 & 22744 & 14616808192 & 6651216/7 \\ \hline
8 & 264057 & 13525751027392 & 68590665/4 \\ \hline
9 & 3242395 & 19385778269260800 & 2921454250/9 \\ \hline
\end{tabular}
\caption{Counts of degree $d$ curves in $\CP^2$ with a full index tangency constraint. For comparison we also include the analogous counts with generic point  constraints, and also with a full index gravitational descendant condition at a point. Note that we have $\gw_{\CP^2,d[L]}\lll \psi^{3d-2}p\rrr = (d!)^{-3}$ \cite{givental_Jfunc}.}
\label{table:CP2_1}
\end{table}

\begin{table}[h]
\begin{tabular}{|l|l|l|l|l|l|}
\hline
$N_1\lll(2)\rrr$ & 1 & $N_5\lll(12,1,1)\rrr$ & 34 & $N_6\lll(11,6)\rrr$ & 56 \\ \hline
$N_2\lll(5)\rrr$ & 1 & $N_5\lll(12,2)\rrr$ & 57 & $N_6\lll(12,3,1,1)\rrr$ & 15 \\ \hline
$N_3\lll(7,1)\rrr$ & 1 & $N_5\lll(13,1)\rrr$ & 182 & $N_6\lll(12,3,2)\rrr$ & 25 \\ \hline
$N_3\lll(8)\rrr$ & 4 & $N_5\lll(14)\rrr$ & 217 & $N_6\lll(12,4,1)\rrr$ & 84 \\ \hline
$N_4\lll(8,3)\rrr$ & 1 & $N_6\lll(8,8,1)\rrr$ & 1 & $N_6\lll(12,5)\rrr$ & 114 \\ \hline
$N_4\lll(9,1,1)\rrr$ & 1 & $N_6\lll(9,6,2)\rrr$ & 1 & $N_6\lll(13,1,1,1,1)\rrr$ & 1 \\ \hline
$N_4\lll(9,2)\rrr$ & 3 & $N_6\lll(9,7,1)\rrr$ & 4 & $N_6\lll(13,2,1,1)\rrr$ & 31 \\ \hline
$N_4\lll(10,1)\rrr$ & 14 & $N_6\lll(9,8)\rrr$ & 13 & $N_6\lll(13,2,2)\rrr$ & 32 \\ \hline
$N_4\lll(11)\rrr$ & 26 & $N_6\lll(10,4,3)\rrr$ & 1 & $N_6\lll(13,3,1)\rrr$ & 210 \\ \hline
$N_5\lll(8,6)\rrr$ & 1 & $N_6\lll(10,5,1,1)\rrr$ & 1 & $N_6\lll(13,4)\rrr$ & 230 \\ \hline
$N_5\lll(9,4,1)\rrr$ & 1 & $N_6\lll(10,5,2)\rrr$ & 2 & $N_6\lll(14,1,1,1)\rrr$ & 69 \\ \hline
$N_5\lll(9,5)\rrr$ & 3 & $N_6\lll(10,6,1)\rrr$ & 14 & $N_6\lll(14,2,1)\rrr$ & 418 \\ \hline
$N_5\lll(10,3,1)\rrr$ & 5 & $N_6\lll(10,7)\rrr$ & 22 & $N_6\lll(14,3)\rrr$ & 487 \\ \hline
$N_5\lll(10,4)\rrr$ & 9 & $N_6\lll(11,3,3)\rrr$ & 4 & $N_6\lll(15,1,1)\rrr$ & 771 \\ \hline
$N_5\lll(11,1,1,1)\rrr$ & 1 & $N_6\lll(11,4,1,1)\rrr$ & 5 & $N_6\lll(15,2)\rrr$ & 892 \\ \hline
$N_5\lll(11,2,1)\rrr$ & 12 & $N_6\lll(11,4,2)\rrr$ & 6 & $N_6\lll(16,1)\rrr$ & 2414 \\ \hline
$N_5\lll(11,3)\rrr$ & 27 & $N_6\lll(11,5,1)\rrr$ & 34 & $N_6\lll(17)\rrr$ & 2110 \\ \hline
\end{tabular}
\caption{The nonzero invariants of the form $N_{d}\lll \Pp \rrr$ for curves of degree up to  $6$.}
\label{table:CP2_2}
\end{table}
\end{example}

\begin{example}[Computations for $\CP^1 \times \CP^1$]\rm
Another important example is the case of $M = \CP^1 \times \CP^1$, which is used in \cite[\S6.2.3]{HSC} to give upper bounds (sometimes sharp) for symplectic capacities of polydisks.
In the case of curves of bidegree $(d,0)$ with $d>1$, all invariants 
$$
N_{\CP^1 \times \CP^1,d [L_1]}\lll \Pp_1,...,\Pp_r\rrr 
$$ 
vanish because they count somewhere injective curves, while the class $d[L_1]$ has no such representatives.   (Any somewhere injective curve in class $A = d[L]$ can be perturbed to be symplectically immersed, and hence would have normal line bundle of Chern class $2d-2$ and thus $A\cdot A>0$.)  Thus, in this case all terms in the  recursion vanish. 

Next, note that all curves of bidegree $(d,1)$ are embedded.  (This again follows from known properties of spheres in $\CP^1 \times \CP^1$.)  Hence 
 $$
 \gw_{\bl^a M,A - n_1[\E_1]-...-n_a[\E_a]} = 0
 $$
  unless $n_1\le 1$ for all $i$. Further, 
 $N_{M,d [L_1] + [L_2]}\lll \Pp_1,...,\Pp_r\rrr = 0$  whenever some $\Pp_i = (m_1^i,\dots,m_{b_i}^i)$  has length $b_i>1$
 and hence $\de(\Pp_i)>0$.
Therefore,  since there is a unique curve of bidegree $(d,1)$ through any generic set of $2d+1$ points, the recursion formula shows that
 $$ 
 N_{M,d [L_1] + [L_2]}\lll (m_1),\dots, (m_r)\rrr 
=\begin{cases} 1  &\mbox {  if  } \sum_i m_i  =  2d+1\\ 
0 &\mbox{ otherwise.} \end{cases}
 $$
 
 In general, recall that there is a symplectomorphism between the blowup  of $\CP^1 \times \CP^1$ at $a \geq 1$ points and the blowup of $\CP^2$ at $a+1$ points.  Indeed, 
let $[L]$ denote the line class in $\bl^2 \CP^2$ and let $[\E_1],[\E_2]$ denote the two exceptional divisor classes.
 Similarly, denote the two line classes in $\bl^1(\CP^1 \times \CP^1)$ by $[L_1]$ and $[L_2]$ and the exceptional divisor class by $[\E]$.
There is a symplectomorphism from $\bl^1 (\CP^1 \times \CP^1)$ to $\bl^2 \CP^2$ whose action on homology sends $[L_1]$ to $[L] - [\E_1]$, $[L_2]$ to $[L] - [\E_2]$ and $[\E]$ to $[L] - [\E_1] - [\E_2]$.
Using this, we can reduce blowup Gromov--Witten invariants of $\CP^1 \times \CP^1$ to corresponding blowup Gromov--Witten invariants of $\CP^2$.
\end{example}

We now outline the main properties of the algorithm.
We define the {\bf  complexity} $\Cc(\Ii)$ of an invariant
\begin{align*}
\Ii: = N_{M,A} \lll (m_1^1,... m_{b_1}^1),\dots, (m_1^r,... m_{b_r}^r)\rrr
\end{align*}
with $m^i_j \geq 2$ for some $i,j$, to be
\begin{align}\label{eq:invar}
\Cc(\Ii) : = \max\limits_{i\in I}\Bigl(\sum\limits_{j=1}^{b_i}m^i_j\Bigr), \qquad I: =\bigl \{i\in \{1,\dots,r\} \ \big| \  \exists j, m^i_j \geq 2\bigr\} 
\end{align}
and set $\Cc(\Ii): = 1$ if all $m^i_j=1$.
Further we say that $\Ii'$ has 
 {\bf strictly smaller complexity} than $\Ii$
 if either 
 \begin{itemize}\item $\Cc(\Ii') < \Cc(\Ii)$, or 
 \item  $\Cc(\Ii') = \Cc(\Ii)\ge 2$, and the number of indices $i$ for which the maximum for $\Ii'$ is achieved is smaller
 than the corresponding number for $\Ii$.
 \end{itemize}
For  example, the complexity of $N_{M,A} \lll (3,1,1)\rrr$ is $5$,
while that of  $N_{M,A} \lll (1,1,1), (2)\rrr$ 
is $2$.
The main ingredient for the recursion underlying Theorem \ref{thm:recur} is the following formula, which is a 
special case of Theorem~\ref{thm:ppt}:
\begin{align}\label{eq:ppt}
& \wh{N}_{M,A}\lll (m),(m_1,\dots,m_b),-\rrr  \; = \; 
\wh{N}_{M,A}\lll (m, m_1,\dots,m_b),-\rrr + \\
&\notag\qquad \qquad  
\sum_{i=1}^b \wh{N}_{M,A} \lll (m_1, m + m_i, \dots m_b),-\rrr.
\end{align}
Notice that the statement involves the numbers $\wh N$  defined in \eqref{eq:whN} rather than the geometrically defined numbers $N$ considered above.  (The reason why it is easiest to work with $\wh N$ should be clear from the proof of Lemma~\ref{lem:two_ends}.)
Here, as before, the symbol $-$ denotes any given additional constraints taking place away from the first two points $q$ and $p$.
By applying this formula ``in reverse'', we will show how to trade invariants involving tangency conditions at say $j$ distinct points in $M$ for invariants with strictly smaller complexity but now involving $j+1$ distinct points in $M$. 
By iterating this finitely many times, we reduce down to invariants of the form 
$$
\wh N\lll (1^{\times n_1},\dots, 1^{\times n_a})\rrr
$$
 for some $a , n_1,\dots, n_a \geq 1$.
 
 The resulting algorithm is fully described in \S\ref{ss:recur2}.  It is generally too computationally intensive to implement by hand except in simple cases. However, one can improve it by feeding in extra information as in the following examples.

\begin{example}\label{ex:deg3} \rm {(Degree $3$ curves in $\CP^2$)}
Since immersed rational curves of degree $3$ have just one node,  the only nonzero invariants of the form \eqref{eq:invar} 
have  $b_i>1$ for at most one $i$, say $i=1$.  Moreover we must have $\de(\Pp_1) \le 1$.  (Recall from  Lemma~\ref{lem:dePp}  that $\de(\Pp)$ is the number of double points of a generic perturbation of a curve satisfying the constraint $\lll\Pp\rrr$.) Thus the only partitions that can arise at a point are $(1)$, $(m)$, or $(m-1,1)$ for some $1 \leq m \leq 8$.
Applying \eqref{eq:ppt} iteratively, and abbreviating $\wh{N}_{3}(\dots); = \wh{N}_{\CP^2,3[L]}(\dots)$,   gives 
\begin{align*}
12 \; & = 
\wh N_3\lll q_1,\dots,q_7,p\rrr \\
& = \wh{N}_{3} \lll q_1,\dots,q_6, (p,p)\rrr + \wh{N}_{3} \lll q_1,\dots,q_6, (\T p)\rrr \\
& = 3 \wh{N}_{3} \lll q_1,\dots,q_5, (\T p,p)\rrr + \wh{N}_{3} \lll q_1,\dots,q_5, (\T^2p)\rrr \\
& = 4 \wh{N}_3 \lll q_1,\dots,q_4, (\T^2 p,p)\rrr + \wh{N}_{3} \lll q_1,\dots,q_4, (\T^3p)\rrr \\
& = \dots\\
& = 8 \wh{N}_{3} \lll  (\T^6 p,p)\rrr + \wh{N}_{3} \lll (\T^7 p)\rrr\\
&  = 8 \wh{N}_{3} \lll  (7,1)\rrr + \wh{N}_{3} \lll (8)\rrr .
\end{align*}
Next recall from Corollary~\ref{cor:T=Er}~(ii) that 
\begin{align*}
\wh{N}_3 \lll q_1,\dots,q_6, (p,p)\rrr & = 2 {N}_{\CP^2,3[L]} \lll q_1,\dots,q_6, (p,p)\rrr \\
& =  2 \gw_{\Bl^7 \CP^2 , 3[L] - [\E_1] - \dots - 2[\E_7]}= 2
\end{align*}
Hence the recursion formula gives
\begin{align*}
2& = \wh{N}_{3} \lll q_1,\dots,q_6, (p,p)\rrr\\
& = 2\wh{N}_{3} \lll q_1,\dots,q_5, (\T p,p)\rrr \\
&= 2\wh{N}_{3} \lll q_1,\dots,q_4, (\T^2 p,p)\rrr \\
&= \dots \\
& = 2\wh{N}_{3} \lll (\T^6p,p)\rrr\;\;  =\;\;  2\wh{N}_{3} \lll (7,1)\rrr.
\end{align*}
Thus 
$$
N_{3} \lll(8)\rrr = 4,
$$
as claimed in Table~\ref{table:CP2_2}.  This calculation should be compared with that in Example~\ref{ex:same} which has  more geometric input but  is less systematic. 

We now compare the above invariants with the Caporaso--Harris calculation that there are $7$ rational cubics that are tangent to a  given line $L$ to order $2$ at a given point $p\in L$ and go through  $5$ other generic points.  It is most natural to compare this count with our invariant $
\wh{N}_{3} \lll q_1,\dots,q_5, (\T^2p)\rrr$.  The  recursion formula shows that
\begin{align*}
10 &= \wh{N}_{3} \lll q_1,\dots,q_6, \T p\rrr \\
& = 
\wh{N}_{3} \lll q_1,\dots,q_5,  (\T p, p)\rrr+
\wh{N}_{3} \lll q_1,\dots,q_5, (\T^2p)\rrr \\
& = 1 + \wh{N}_{3} \lll q_1,\dots,q_5, (\T^2p)\rrr.
\end{align*}
Therefore,  $\wh{N}_{3} \lll q_1,\dots,q_5, (\T^2p)\rrr = 9$. 
To understand why $\wh{N}_{3} \lll q_1,\dots,q_5, (\T^2p)\rrr$ is larger than the Caporaso--Harris count, observe that the condition that the local divisor $D$ extend to a $J$-holomorphic line $L$  is not satisfied by a generic element in $\Jj_D$.  If $J\in \Jj_D$ does have this property then some of the curves counted by $\wh{N}_{3} \lll q_1,\dots,q_5, (\T^2p)\rrr = 9$ could be degenerate.  Indeed there is a unique conic  $Q$ through $q_1,\dots,q_5$ and its union with a line can be parametrized as a degree three genus zero stable map in two essentially different ways, depending on which of the two intersection points $Q\cap L$ is designated as the image of the unique node in the domain of the stable map.  These two curves do contribute to $\wh{N}_{3} \lll q_1,\dots,q_5, (\T^2p)\rrr$, however they do not contribute to the Caporaso--Harris count.
This example clearly illustrates that considering tangency to a line $L$ instead of
to a generic local divisor $D$  makes a real difference to the invariant.  
\hfill$\er$
\end{example}

\begin{rmk}\rm\label{rmk:dualpart}  (Partitions $\Pp$ with $N_{M,A}\lll\Pp\rrr = 1$).   It is not hard to check that
if $\Qq = (n_1,\dots,n_a)$ is the dual partition to $\Pp = (m_1,\dots,m_b)$ in the sense of
Remark~\ref{rmk:dual} then
$$
\sum_{j=1}^a n_j^2 = |\Pp|+ 2\de(\Pp),  
$$
where, by Lemma~\ref{lem:dePp}, $\de(\Pp)$ is the number of  double points near $p$ of an $A$-curve that satisfies the constraint 
$\lll\T^\Pp p\rrr$.  If $\de(\Pp) = \de(A)  = \frac 12 (A^2-c_1(A)) + 1$ (where $\de(A)$ is as in \eqref{eq:deA}),
then the immersed $A$-curves that satisfy $\lll \T^\Pp p\rrr $ have no other double points.  
Therefore
$\de(A_\Pp) = 0$, which implies that, if in addition $|\Pp|  = c_1(A) - 1$, 
the class $A_\Pp$ must be the class of an exceptional sphere.
In such a case, if one puts  all the  blow up points inside a skinny ellipsoid and  stretches the neck,
 then the arguments in \S\ref{ss:same} should adapt to show that the unique curve in class $A_\Pp$  must decompose into a building whose top in $M\less E_\sk$ has bottom partition $\Pp$; in other words, one should have   $N^E_d\lll\Pp\rrr = 1$.
Conversely, if one looks at Table~\ref{table:CP2_2}, it turns out that for every partition $\Pp$ with $N_d\lll\Pp\rrr = 1$ the dual class $A_\Pp$ is an exceptional class.

Note that in general the number of $A$-curves satisfying the constraint $\lll \T^\Pp p\rrr$ is different from the count of curves in the dual class $A_\Pp: = A - \sum_{j=1}^a n_j[\E_j]$. Indeed,  if one blows up at $a$ distinct points inside a single skinny  ellipsoid and 
looks to see what happens to a curve in class $A_\Pp$ as one
 stretches the neck,  the partition given by the ends of the top level of the split curve  need not in general equal $\Pp$ since some of the double points of the $A_\Pp$-curve can move into the neck. 
 (The possible partitions in this situation can be deduced
 from Theorem~\ref{thm:ppt}.)    For example, if $\Pp = (8)$ and $A = 3[L]$ in $\C P^2$ then 
$$
N_3\lll \T^\Pp p \rrr = 4,\quad\mbox{ while } \;\;  \gw_{3[L]-\sum_{i=1}^8 [\Ee_i]} = 12.
$$
  Explicit examples of this neck stretching process may be found in \cite[\S3.5]{Ghost}, that discusses the case of degree $d$ curves in $\C P^2$ for $d=3,4$ with end on the (skinny) ellipsoid $E(1,x)$ for $x> 3d-1$. 
See Remark~\ref{rmk:relGW} for a different  geometric approach to this question.\hfill$\er$
\end{rmk}

\begin{remark}\rm (i)
Since the formula \eqref{eq:ppt} is based only on local considerations near the point $p$, one could also formulate a version of Theorem \ref{thm:recur} for punctured pseudoholomorphic curves in four-dimensional symplectic cobordisms. However, in this case the statement is somewhat more involved, since in general such counts do not define numerical invariants but rather chain maps between chain complexes, 
which necessitates working in the framework of \cite{HSC} except in special cases.
\MS

\NI
(ii)  We do not know whether there is a natural analog of Theorem \ref{thm:recur} for higher dimensional symplectic manifolds. However, see \cite{Gath_blowup} for a higher dimensional analog of Formula \eqref{eq:ppt} in the case of $\CP^n$ and a single first order tangency constraint.\hfill$\er$
\end{remark}
\MS

\subsection{Existence of the algorithm}\label{ss:recur2}
As in \S\ref{ss:comb}, it will be convenient to represent partitions by Young diagrams.
We define a (total) ordering of $\Y_k$ using the rule that for $y,y' \in \Y_k$, we have $y' > y$ if for some $i$ the $i$th row (from the top) of $y'$ has more boxes than the $i$th row of $y$, and for all $j < i$, the $j$th rows of $y$ and $y'$ have the same number of boxes.
For example, the ordering of $\Y_5$ is given by
\begin{align*}
{\Yvcentermath1 \tiny \yng(1,1,1,1,1) < \yng(2,1,1,1) < \yng(2,2,1) < \yng(3,1,1) < \yng(3,2) < \yng(4,1) < \yng(5)}.
\end{align*}
For some given $k \in \Z_{> 0}$, let $y_1,...,y_{|\Y_k|}$ denote the elements of $\Y_k$ in increasing order.
Let $\Ynonhor_k := \{y_1,...,y_{|\Y_k|-1}\}$ denote the subset of $\Y_k$ consisting of all Young diagrams except for the horizontal one,
and let $\Ynonver_k := \{y_2,...,y_{|\Y_k|}\}$ denote the subspace of $\Y_k$ consisting of all Young diagrams except for the vertical one.
Let $V_k := \Q\langle \Ynonhor_k \rangle$ denote the $(|\Y_k|-1)$-dimensional rational vector space with basis $\Ynonhor$, and similarly put $W_k := \Q \langle \Ynonver_k\rangle$.
We define a linear map $\Phi_k: W_k \rightarrow V_k$ as follows:
\begin{itemize}
\item For a Young diagram $y \in \Ynonver_k \cap \Ynonhor_k$, the matrix coefficient $\langle \Phi_k(y),y\rangle$ is $1$.
\item For distinct Young diagrams $y \neq y'$ with $y \in \Ynonver_k$ and $y' \in \Ynonhor_k$, the matrix coefficient $\langle \Phi_k(y),y'\rangle$ is the number of ways of removing the top row of $y'$ and adding it to the end of one of the lower rows of $y'$ such that the result is $y$ (after reordering the rows to obtain a valid Young diagram).
\end{itemize}
Using the natural bases of $W_k$ and $V_k$ by Young diagrams, we can also represent $\Phi_k$ by a
$(|\Y_k|-1)\times(|\Y_k|-1)$ matrix, which we denote by $A_k$.
Here the entry $\langle \Phi_k(y),y' \rangle$ occurs in the row labeled by $y'$ and the column labeled by $y$ (see the examples below). 
Observe that for each $y \in \Ynonver_k$, $\Phi_k(y)$ is a linear combination of Young diagrams
$y' \in \Ynonhor_k$ with $y' \leq y$.
This translates into the fact that the matrix $A$ is 
almost upper triangular; in fact it has nonzero entries $a_{i,j}$ only when $j\ge i-1$.
We will show below that, with respect to suitable choices of bases, the map $\Phi_k$ realizes the equation \eqref{eq:ppt}, with the rows of $A_k$ indexing constraints of the form 
$\lll(m),(m_1,...,m_b),-\rrr$
and the columns of $A_k$ indexing constraints of the form  
$\lll(m, m_1,...,m_b),-\rrr$.
\begin{example}\label{ex:A_4}\rm
For $k = 4$, we have $|\Y_4| = 5$, and corresponding Young diagrams $y_1 < ... < y_5$, with $y_1$ the vertical Young diagram and $y_5$ the horizontal Young diagram.
Then $A_4$ is the following $4 \times 4$ matrix:
\begin{align*}
\begin{blockarray}{cccccc}
&& {\tiny\yng(2,1,1)} & {\tiny\yng(2,2)} & {\tiny\yng(3,1)} & {\tiny\yng(4)}\\ 
&&y_2&y_3&y_4&y_5\\ 
\begin{block}{cc(cccc)}
{\tiny\yng(1,1,1,1)}&y_1&3&0&0&0\\
{\tiny\yng(2,1,1)}&y_2&1&0&2&0\\
{\tiny\yng(2,2)}&y_3&0&1&0&1\\
 {\tiny\yng(3,1)}&y_4&0&0&1&1\\
\end{block}
\end{blockarray}
\end{align*}
One can easily check that $\det(A_4) = -6$, and in particular $A_4$ is invertible.\hfill$\er$
\end{example}

The following combinatorial lemma generalizes the above example to arbitrary positive integers $k$ and will be central for our recursion algorithm.
\begin{lemma}\label{lem:invertible}
For all $k \in \Z_{> 0}$, we have $\det(A_k) = \pm(k-1)!$. In particular, $A_k$ is invertible.
\end{lemma}

\noindent Taking Lemma \ref{lem:invertible} for granted for the moment, we now describe the recursion algorithm and prove Theorem \ref{thm:recur}.

\begin{proof}[Proof of Theorem \ref{thm:recur}]    We will work with the invariants $\wh N$ rather than $N$, since the recursion  is more transparent in this notation. 
Consider an invariant of the form
$$\wh N_{M,A} \lll (m^1_1,...,m^1_{b_1}),\dots,(m^r_1,...,m^r_{b_r})\rrr.
$$
Assume that $m^i_j \geq 2$ for some $i,j$.
After  
rearranging the constraints so that the first one has a maximal  sum $\sum\limits_{j=1}^{b_i}m^i_j$
we can write such an invariant more succinctly as 
$$
\wh N_{M,A} \lll (m_1,...,m_{b_1}),-\rrr,
$$
where we also assume without loss of generality that we have $m_1 \geq ... \geq m_b$.
Our goal is to write any such invariant as a linear combination of invariants of strictly smaller complexity.

Put $k = \sum_{i=1}^b m_i$.
Note that the tuple $(m_1,...,m_b)$ naturally corresponds to a Young diagram $y \in \Y_k$ such that the $i$th row has $m_i$ blocks. 
As before, let $y_1,...,y_{|\Y_k|} \in \Y_k$ denote the ordered list of Young diagrams with $k$ boxes.
Now let $w_1,...,w_{|\Y_k|}$ denote all of the corresponding invariants of the form 
$\wh{N}_{M,A} \lll \Pp,-\rrr$
corresponding to some partition $\Pp$ of $k$.
Recall here that $\wh{N}_{M,A} \lll \Pp,-\rrr$ is defined to be simply $N_{M,A} \lll \Pp,-\rrr$ times the extra combinatorial factor $|\Aut(\Pp)|$. 
We thus have 
\begin{align*}
& w_1 \; = \wh{N}_{M,A} \lll \underbrace{(1,...,1)}_k,-\rrr\\
& w_2 \; = \wh{N}_{M,A} \lll (2,\underbrace{1,...,1}_{k-2}),-\rrr\\
&\quad \dots\\
& w_{|\Y_k|} \;=\wh{N}_{M,A}\lll (k),-\rrr.
\end{align*}
Similarly, let $v_1,...,v_k$ denote the analogous invariants given by replacing 
\begin{align*}
\wh{N}_{M,A} \lll (m_1,...,m_b),-\rrr \quad \text{with}\quad 
\wh{N}_{M,A} \lll (m_1),(m_2,...,m_b),-\rrr.
\end{align*}
That is, we have 
\begin{align*}
& v_1 \; = \wh{N}_{M,A} \lll (1),\underbrace{(1,...,1)}_{k-1},-\rrr\\
& v_2 \; = \wh{N}_{M,A} \lll (2),\underbrace{(1,...,1)}_{k-2},-\rrr\\
&\quad \dots\\
& v_{|\Y_k|} \; = \wh{N}_{M,A} \lll (k),-\rrr.
\end{align*}

\NI
Note that we have $v_{|\Y_k|} = w_{|\Y_k|}$, and that, because we are decomposing one of the constraints of maximal complexity,  each of the invariants $v_1,...,v_{|\Y_k|-1}$ has strictly smaller complexity than each of the invariants $w_2,...,w_{|\Y_k|}$.
By applying Formula \eqref{eq:ppt} once for each of the invariants $v_1,...,v_{|\Y_k|-1}$, we find that the column vectors $\vec{v} := (v_1,...,v_{|\Y_k|-1})^T$ and $\vec{w} := (w_2,...,w_{|\Y_k|})^T$ are related by the following equation:
\begin{align}\label{eq:w_to_v}
 \vec{v} = (w_1,0,...,0)^T + A_k \vec{w}.
\end{align}
In other words, each $v_i$ is a sum of terms $A_{ij} w_j$ where the coefficient $A_{ij}$ is the number of ways of removing the top row of the Young diagram for $v_i$ and adding it to one of the lower rows to obtain (after rearrangement) the Young diagram for $w_j$.   
Then, since $A_k$ is invertible by Lemma \ref{lem:invertible}, we have
\begin{align}\label{eq:recur2}
\vec{w} = A_k^{-1}\left(\vec{v} - (w_1,0,...,0)^T\right),
\end{align}
which gives the desired recursion.
\end{proof}

\begin{remark}\rm
A noteworthy feature of the above recursion is that the coefficients of the relation \eqref{eq:recur2} are not in general integers.  However, all the invariants $N_{M,A}\lll\Pp\rrr$ are of course integers.
\hfill$\er$\end{remark}

\begin{example}\rm
In the case $k = 4$ (c.f. Example~\ref{ex:A_4}), equation~\eqref{eq:w_to_v} amounts to the following system of equations:
\begin{align*}
\wh{N}_{M,A}\lll (q),(p,p,p),-\rrr  &\;=\; \wh{N}_{M,A}\lll (p,p,p,p),-\rrr + 3\wh{N}_{M,A}\lll (\T p,p,p),-\rrr \\
\wh{N}_{M,A}\lll (\T q),(p,p),-\rrr &\;=\; \wh{N}_{M,A}\lll (\T p,p,p),-\rrr + 2\wh{N}_{M,A}\lll (\T^2 p,p),-\rrr \\ 
\wh{N}_{M,A}\lll (\T q),(\T p),-\rrr &\;=\; \wh{N}_{M,A}\lll (\T p,\T p),-\rrr + \wh{N}_{M,A}\lll (\T^3 p),-\rrr \\
\wh{N}_{M,A}\lll (\T^2 q),(p),-\rrr &\;=\; \wh{N}_{M,A}\lll (\T^2 p,p),-\rrr  + \wh{N}_{M,A}\lll (\T^3 p),-\rrr.
\end{align*}
The invariants on the right hand sides   all have complexity four, apart from the blow up invariant
$\wh{N}_{M,A}\lll (p,p,p,p),-\rrr$ which is assumed known, while the invariants on the left hand side have complexity at most three.   Further,  one can solve for 
\begin{align*}
\wh{N}_{M,A}\lll (\T p,p,p),-\rrr,& \quad \wh{N}_{M,A}\lll  (\T^2 p,p),-\rrr,\\
 \wh{N}_{M,A}\lll  (\T^3 p) ,-\rrr, & \quad \wh{N}_{M,A}\lll  (\T p,\T p),-\rrr
\end{align*}
 in turn in 
terms of other invariants that all have strictly smaller complexity.\hfill$\er$
\end{example}

\MS

\begin{proof}[Proof of Lemma \ref{lem:invertible}]
Assume by induction that we have $\det(A_{k-1}) = \pm (k-2)!$.
Let $\iota: \Y_{k-1} \rightarrow \Y_k$ denote the injective order-preserving set map which sends a Young diagram $y \in \Y_{k-1}$ to the Young diagram $\iota(y) \in \Y_k$ obtained by adding a single box to the top row of $y$.
Thus the image of $\iota$ consists of all diagrams in which the top row is strictly longer than the second row. 
Let $A_k'$ denote the $(|\Y_{k-1}|-1) \times (|\Y_{k-1}|-1)$ submatrix of $A_k$ corresponding to rows from $\iota(\Ynonhor_{k-1})$ and columns from $\iota(\Ynonver_{k-1})$.
It is not hard to check that
 $A_k'$ coincides with the matrix $A_{k-1}$, and in particular has determinant $\pm(k-2)!$ by our inductive hypothesis.

Also, one can readily check that the rows of $A_k$ corresponding to $\iota(\Ynonhor_{k-1})$ have nonzero entries only in the columns corresponding to $\iota(\Ynonver_{k-1})$ as well as the column corresponding to $y_2 \in \Ynonver_k$.
Moreover, the row of $A_k$ corresponding to $y_1 \in \Ynonhor_k$ consists of the entry $k-1$ in the first column and zeroes elsewhere.
This means that we can perform type III elementary row operations (i.e. adding a multiple of a row to another row) to transform $A_k$ to a matrix $B_k$ without creating any new nonzero entries, such that the first column of $B_k$ consists of the entry $k-1$ in the first row and zeroes elsewhere.
Note that in this process the submatrix $A_k'$ remains unchanged.
Next,  since the rows of $B_k$ corresponding to $\iota(\Ynonhor_{k-1})$ are nonzero only in the columns corresponding to $\iota(\Ynonver_{k-1})$, we can perform type III elementary row operations to transform $B_k$ to a matrix $C_k$ without creating any new nonzero entries, such that the columns of $C_k$ corresponding to $\iota(\Ynonver_{k-1})$ are nonzero only in the rows corresponding to $\iota(\Ynonhor_{k-1})$.  Here we use the fact that $A_k'$ is invertible,and so can be assumed to be upper triangular.   For example, in Example~\ref{ex:A_5} below we first clear the term $1$ in the place $(2,1)$ and then clear the terms in the third row and columns $4,5$.

After we have done this the resulting matrix $C_k$ decomposes into a product.  More precisely, 
let $C_k^\perp$ denote the $(|\Y_k| - |\Y_{k-1}|-1)\times (|\Y_k| - |\Y_{k-1}|-1)$ submatrix of $C_k$ with rows corresponding to $\Ynonhor_k \setminus \left(\iota(\Ynonhor_{k-1}) \cup\{y_1\}\right)$ and columns corresponding to $\Ynonver_k \setminus \left(\iota(\Ynonver_{k-1}) \cup\{y_1\}\right)$.  
Then  $\det(A_k) = \pm k! \det(C_k^\perp)$.
Moreover, we have that $C_k^\perp$ is upper triangular with all diagonal entries $1$, and hence $\det(C_k^\perp) = 1$.
\end{proof}

\begin{example}\label{ex:A_5}\rm
For $k = 5$, we have $|\Y_5| = 7$ and corresponding Young diagrams $y_1 < ... < y_7$ as above, with $y_1$ the vertical Young diagram and $y_7$ the horizontal Young diagram.
Then $A_5$ is the following $6\times 6$ matrix:
\begin{align*}
\begin{blockarray}{cccccccc}
&& {\tiny\yng(2,1,1,1)}&{\tiny\yng(2,2,1)}&{\tiny\yng(3,1,1)} & {\tiny\yng(3,2)} & {\tiny\yng(4,1)} & {\tiny\yng(5)} \\
&& y_2&y_3&y_4&y_5&y_6&y_7 \\
\begin{block}{cc(cccccc)}
{\tiny\yng(1,1,1,1,1)}&y_1&4& 0 & 0 & 0 & 0 & 0 \\
{\tiny\yng(2,1,1,1)}&y_2&1& 0 & \boxed{3} & \boxed{0} & \boxed{0} & \boxed{0} \\
{\tiny\yng(2,2,1)}&y_3&0& 1 & 0 & 1 & 1 & 0 \\
{\tiny\yng(3,1,1)}&y_4&0& 0 & \boxed{1} & \boxed{0} & \boxed{2} & \boxed{0} \\
{\tiny\yng(3,2)}&y_5&0& 0 & \boxed{0} & \boxed{1} & \boxed{0} & \boxed{1} \\
{\tiny\yng(4,1)}&y_6&0& 0 & \boxed{0} & \boxed{0} & \boxed{1} & \boxed{1} \\
\end{block}
\end{blockarray}
\end{align*}
Again, one can easily check that $\det(A_5) = \pm 4!$, and in particular $A_5$ is invertible.
Here the boxed entries correspond to the submatrix $A_5'$, which one can check agrees with the matrix $A_4$ from Example \ref{ex:A_4} above.\hfill$\er$
\end{example}

\bibliographystyle{plain}

\bibliography{denum_bib}

\end{document}